\documentclass[reqno]{amsart}
\usepackage{amsfonts,amsmath,amssymb,amsrefs}
\usepackage{color}
\usepackage[utf8]{inputenc}  
\usepackage[T1]{fontenc}

\numberwithin{equation}{section}

\newtheorem{thm}{Theorem}[section]
\newtheorem*{thm*}{Theorem}
\newtheorem{lem}{Lemma}[section]

\newtheorem{cor}{Corollary}[section]
\newtheorem*{cor*}{Corollary}
\newtheorem{prop}{Proposition}[section]
\newtheorem{claim*}{Claim}

\newtheorem{rem}{Remark}[section]
\newtheorem{quest}{Question}
\newtheorem{Step}{Step}[section]

\DeclareMathOperator*{\esssup}{ess\,sup}
\begin{document}

\title[Extremal for the perturbed Moser-Trudinger inequalities]
{When does a perturbed Moser-Trudinger inequality admit an extremal?}

\author{Pierre-Damien Thizy}
\address[Pierre-Damien Thizy]{Universit\'e Claude Bernard Lyon 1, CNRS UMR 5208, Institut Camille Jordan, 43 blvd. du 11 novembre 1918, F-69622 Villeurbanne cedex, France}
\email{pierre-damien.thizy@univ-lyon1.fr}
\date{February 2018}

\begin{abstract}
In this paper, we are interested in several questions raised mainly in \cite{MartMan} (see also \cite{McLeodPeletier,Pruss}). We consider the perturbed Moser-Trudinger inequality $I_\alpha^g(\Omega)$ below, at the critical level $\alpha=4\pi$, where $g$, satisfying $g(t)\to 0$ as $t\to +\infty$, can be seen as a perturbation with respect to the original case $g\equiv 0$. Under some additional assumptions, ensuring basically that $g$ does not oscillates too fast as $t\to +\infty$, we identify a new condition on $g$ for this inequality to have an extremal. This condition covers the case $g\equiv 0$ studied in \cite{CarlesonChang,StruweCrit,Flucher}. We prove also that this condition is sharp in the sense that, if it is not satisfied, $I_{4\pi}^g(\Omega)$ may have no extremal.
\end{abstract}

\maketitle

\section{Introduction}\label{SectIntro}
Let $\Omega$ be a smooth, bounded domain of $\mathbb{R}^2$ and let $H^1_0=H^1_0(\Omega)$ be the standard Sobolev space, obtained as the completion of the set of smooth functions with compact support in $\Omega$, with respect to the norm $\|\cdot\|_{H^1_0}$ given by
$$\|u\|_{H^1_0}^2=\int_\Omega |\nabla u(x)|^2 dx\,. $$
Throughout the paper, $\Omega$ is assumed to be connected. Let $g$ be such that 
\begin{equation}\label{Propg}
\begin{split}
&g\in C^1(\mathbb{R})\,, \quad\lim_{s\to +\infty} g(s)=0\,,\quad  g(t)>-1\text{ and }g(t)=g(-t)\text{ for all }t\,.
\end{split}
\end{equation}
Then, we have that
 \begin{equation}\tag{$I_\alpha^g(\Omega)$}
C_{g,\alpha}(\Omega):=\sup_{u\in H^1_0:{\|u\|_{H^1_0}^2\le \alpha}} \int_{\Omega} (1+g(u))\exp(u^2) dx 
 \end{equation}
is finite for $0<\alpha\le 4\pi$ and equals $+\infty$ for $\alpha>4\pi$. This result was first obtained by Moser \cite{MoserIneq} in the unperturbed case $g\equiv 0$. Still by \cite{MoserIneq}, we easily extend the $g\equiv 0$ case to the case of $g$ as in \eqref{Propg}. At last, \cite{MoserIneq} gives also the existence of an extremal for $(I_\alpha^g(\Omega))$ if $0<\alpha<4\pi$ (see Lemma \ref{SubcriticalCase}). If now $\alpha=4\pi$, getting the existence of an extremal is more challenging; however Carleson-Chang \cite{CarlesonChang}, Struwe \cite{StruweCrit} and Flucher \cite{Flucher} were also able to prove that $(I_{4\pi}^0(\Omega))$ admits an extremal in the unperturbed case $g\equiv 0$. Yet, surprisingly, McLeod and Peletier \cite{McLeodPeletier} conjectured that there should exist a $g$ as in \eqref{Propg} such that $(I_{4\pi}^g(\Omega))$ does not admit any extremal function. Through a nice but very implicit procedure, Pruss \cite{Pruss} was able prove that such a $g$ does exist. Observe that, since $g(u)\to 0$ as $u\to +\infty$ in \eqref{Propg}, $(1+g(u))\exp(u^2)$ in $(I_\alpha^g(\Omega))$ sounds like a very mild perturbation of $\exp(u^2)$ as $u\to +\infty$ and then, this naturally raises the following question: 
\begin{quest}
To what extent does the existence of an extremal for the critical Moser-Trudinger inequality $(I_{4\pi}^0(\Omega))$ really depend on asymptotic properties of the function $t\mapsto\exp(t^2)$ as $t\to +\infty$ ? 
\end{quest}
 To investigate this question, we may rephrase it as follows: for what $g$ satisfying \eqref{Propg} does $(I_{4\pi}^g(\Omega))$ admit an extremal ? This is Open problem 2 in Mancini and Martinazzi \cite{MartMan}, stated in this paper for $\Omega=\mathbb{D}^2$, the unit disk of $\mathbb{R}^2$. In order to state our main general result, we introduce now some notations. For a first reading, one can go directly to Corollary \ref{CorD2}, which aims to give a less general but more readable statement. We let $H:(0,+\infty)\to \mathbb{R}$ be given by
\begin{equation}\label{DefH}
H(t)=1+g(t)+\frac{g'(t)}{2t}\,,
\end{equation}
so that we have
\begin{equation}\label{Croissance}
[(1+g(t))\exp(t^2)]'=2t H(t) \exp(t^2)\,.
\end{equation}
We set $tH(t)=0$ for $t=0$, so that $t\mapsto tH(t)$ is continuous at $0$ by \eqref{Propg}. This function $H$ comes into play, since the Euler-Lagrange associated to $(I^g_{\alpha}(\Omega))$ reads as 
\begin{equation}\label{ELEq1}
\begin{cases}
\Delta u=\lambda u H(u) \exp(u^2)\text{ in }\Omega\,,\\
u=0\text{ in }\partial\Omega\,,
\end{cases}
\end{equation}
where $\lambda\in\mathbb{R}$ is a Lagrange multiplier and $\Delta=-\partial_{xx}-\partial_{yy}$ (see also Lemma \ref{SubcriticalCase} below). Now, we make some further assumptions on the behavior of $g$ at $+\infty$ and at $0$. First, we assume that there exist $\delta_0\in (0,{1})$ and a sequence of real numbers $A=(A(\gamma))_{\gamma}$ such that
\begin{equation}\label{InftyBehavior}
\begin{split}
&\text{a) }{ H\left(\gamma-\frac{t}{\gamma}\right)}={H(\gamma)}\left(1+A(\gamma) t+o(|A(\gamma)|+\gamma^{-4})\right)\,,\\
& \text{ in } C^{{0}}_{loc}\left([0,+\infty)_t\right), \text{ as }  \gamma\to +\infty\,,\\
&\text{b) }\exists C>0\,, \left|H\left(\gamma-\frac{t}{\gamma}\right)-H(\gamma) \right|\le C|H(\gamma)|(|A(\gamma)|+\gamma^{-4}){\exp(\delta_0 t)}\,,\\
& \text{ for all }\gamma\gg 1 \text{ and  all }0\le t\le \gamma^2\,,\\
&\text{c) }\lim_{\gamma\to +\infty}  A(\gamma)= 0\,.
\end{split}
\end{equation}
We also assume that there exist $\delta_0'\in (0,{1})$, $\kappa\ge 0$, $\tilde{\varepsilon}_0\in\{-1,+1\}$, $F$ given by $F(t):=\tilde{\varepsilon}_0 t^{\kappa}$, and a sequence $B=(B(\gamma))_\gamma$ of positive real numbers such that
\begin{equation}\label{ZeroBehavior}
\begin{split}
&\text{a) }{\frac{t}{\gamma}H\left(\frac{t}{\gamma}\right)}=B(\gamma) F(t)+o(|B(\gamma)|+\gamma^{-1})\,,\\
&\text{ in }C^0_{loc}\left((0,+\infty)_t\right), \text{ as }  \gamma\to +\infty\,,\\
&\text{b) }\exists C>0\,, \left|\frac{t}{\gamma}H\left(\frac{t}{\gamma}\right)\right|\le C (|B(\gamma)|+\gamma^{-1}) {\exp(\delta_0' t)}\,,\\
&\text{ for all }\gamma\gg 1 \text{ and  all }0\le t\le {\gamma^2}\,.
\end{split}
\end{equation}
Observe that we may have $B(\gamma)=o(\gamma^{-1})$ as $\gamma \to +\infty$, in which case the precise formula for $F$ is not really significant. Since $t\mapsto (1+g(t)) \exp(t^2)$ is an even $C^1$ function, we have that
\begin{equation}\label{LimitB}
\lim_{\gamma\to +\infty}  B(\gamma)= 0\,,
\end{equation}
in view of \eqref{Croissance} and \eqref{ZeroBehavior}. Following rather standard notations, we may split the Green's function $G$ of $\Delta$, with zero Dirichlet boundary conditions in $\Omega$, according to 
\begin{equation}\label{GreenFun}
G_x(y)=\frac{1}{4\pi}\left(\log\frac{1}{|x-y|^2}+\mathcal{H}_x(y) \right)\,,
\end{equation}
for all $x\neq y$ in $\Omega$, where $\mathcal{H}_x$ is harmonic in $\Omega$ and coincides with $-\log\frac{1}{|x-\cdot|^2}$ in $\partial\Omega$. Then the Robin function $x\mapsto \mathcal{H}_x(x)$ is smooth in $\Omega$, and goes to $-\infty$ as $x\to \partial \Omega$, so that we may set
\begin{equation}\label{MaxRobFun}
\begin{split}
&M=\max_{x\in \Omega} \mathcal{H}_x(x)\,,\\
&K_\Omega=\{y\in\Omega\text{ s.t. }\mathcal{H}_y(y)=M\}\text{ and}\\
&S=\max_{z\in K_\Omega} \int_{\Omega} G_z(y) F(4\pi G_z(y)) dy\,,
\end{split}
\end{equation}
where $F$ is as in \eqref{ZeroBehavior}. For $N\ge 1$, we let $g_N$ be given by
\begin{equation}\label{gN}
(1+g_N(t))\exp(t^2)=(1+g(t))(1+t^2)+(1+g(t))\left(\sum_{k=N+1}^{+\infty}\frac{t^{2k}}{k!}\right)\,,
\end{equation}
so that $g_N\le g$, $g_N(0)=g(0)$ for all $N\ge 1$, while $g=g_N$ for $N=1$. We also set
\begin{equation}\label{NonlinEigen}
\Lambda_g(\Omega):=\max_{u\in H^1_0:\|u\|_{H^1_0}^2\le 4\pi} \int_\Omega\left((1+g(u))(1+u^2)-(1+g(0))\right) dx
\end{equation}
 We are now in position to state our main result, giving a new, very general and basically sharp picture about the existence of extremals for the perturbed Moser-Trudinger $(I_{4\pi}^g(\Omega))$. 
 
\begin{thm}[Existence and non-existence of an extremal]\label{MainThm}
Let $\Omega$ be a smooth bounded domain of $\mathbb{R}^2$. Let $g$ be such that \eqref{Propg} and \eqref{InftyBehavior}-\eqref{ZeroBehavior} hold true for $H$ as in \eqref{DefH}, and let $A$, $B$ and $F$ be thus given. Assume that 
\begin{equation}\label{Ratio}
l=\lim_{\gamma\to +\infty}  \frac{\gamma^{-4}+A(\gamma)/2+4\gamma^{-3}\exp(-1-M) B(\gamma) S}{\gamma^{-4}+|A(\gamma)|+\gamma^{-3}|B(\gamma)|}
\end{equation}
exists, where $M$ and $S$ are given by \eqref{MaxRobFun}. Then
\begin{enumerate}
\item if \fbox{$l>0$ or $\Lambda_g(\Omega)\ge \pi \exp(1+M)$}, $(I_{4\pi}^g(\Omega))$ admits an extremal, where $\Lambda_g(\Omega)$ is as in \eqref{NonlinEigen};
\item if \fbox{$l<0$ and $\Lambda_g(\Omega)<\pi \exp(1+M)$}, there exists $N_0\ge 1$ such that $(I_{4\pi}^{g_N}(\Omega))$ admits no extremal for all $N\ge N_0$, where $g_N$ is given by \eqref{gN}. 
\end{enumerate}
\end{thm}
Observe that, for all given $N\ge 1$, $g_N$ satisfies \eqref{Propg} and \eqref{InftyBehavior}-\eqref{ZeroBehavior}, with the same $A$, $B$ and $F$ as the original $g$. Moreover it is clear that $\Lambda_{g_N}(\Omega)\le \Lambda_g(\Omega)$. Then, this second assertion in Theorem \ref{MainThm} proves that the assumptions on $g$ in the first assertion are basically sharp to get the existence of an extremal for $(I_{4\pi}^g(\Omega))$. As a remark, Pruss concludes in \cite{Pruss} that the existence of an extremal for the critical Moser-Trudinger inequality is in some sense accidental and relies on non-asymptotic properties of $\exp(u^2)$. Theorem \ref{MainThm} clarifies this tricky situation: the existence or nonexistence of an extremal for $(I_{4\pi}^g(\Omega))$ may really depend on a balance of the asymptotic properties of $g$ both at infinity (given by $A(\gamma)$) and at zero (given by $B(\gamma)$). Yet, it may also depend on the non-asymptotic quantity $\Lambda_g(\Omega)$ (see Corollary \ref{NonAsymptCor}). Observe that $\Lambda_0(\Omega)=(4\pi)/\lambda_1(\Omega)$ in the unperturbed case $g\equiv 0$, where $\lambda_1(\Omega)$ is the first Dirichlet eigenvalue of $\Delta$ in $\Omega$. \\
\indent From now on, we illustrate Theorem \ref{MainThm} by two corollaries dealing with less general but more explicit situations. Let $c,c'\in\mathbb{R}$, $(a,b),(a',b')\in\mathcal{E}$, where
\begin{equation}\label{DefERonde}
\mathcal{E}=\left\{(a,b)\in [0,+\infty)\times \mathbb{R}~~\big| ~b>0\text{ if }a=0 \right\}\,.
\end{equation}
 Let $R'>0$ be a large positive constant. If one picks $g$ such that 
\begin{equation}\label{HypgDisk}
g(t)=
\begin{cases}
g_0(t):=g(0)+c t^{a+1}\log(1/t)^{-b}\text{ in }(0,1/R']\,,\\
g_\infty(t):=c' t^{-a'}(\log t)^{-b'} \text{ in }[R',+\infty)\,,
\end{cases}
\end{equation}
$l$ in \eqref{Ratio} of Theorem \ref{MainThm} can be made more explicit. Indeed, we can then set
\begin{equation}\label{Examples}
\begin{split}
&B(\gamma)= \frac{1+g(0)}{\gamma}+\frac{c(a+1)}{2}\gamma^{-a}\left(\log\gamma \right)^{-b}\text{ and }
\quad F(t)=t^{\min(a,1)}\,,
\\
&A(\gamma)=c'\times
\begin{cases}
a'\gamma^{-(a'+2)}(\log\gamma)^{-b'}\text{ if }a'>0\,,\\
b'\gamma^{-2}(\log\gamma)^{-(b'+1)}\text{ if }a'=0\,,
\end{cases}
\end{split}
\end{equation}
({see also Lemma \ref{HBehavior}}). Theorem \ref{MainThm} is even more explicit in the particular case $\Omega=\mathbb{D}^2$. Indeed, in this case we have that $K_{\mathbb{D}^2}=\{0\}$ in \eqref{MaxRobFun} and $G_0(x)=\frac{1}{2\pi}\log\frac{1}{|x|}$. Still on the unit disk $\mathbb{D}^2$, it is known that 
\begin{equation}\label{CondDisk}
\Lambda_0(\mathbb{D}^2)=\frac{4\pi}{\lambda_1(\mathbb{D}^2)}<\pi e\,,
\end{equation}
$(\lambda_1(\mathbb{D}^2)\simeq 5.78)$. Property \eqref{CondDisk} shows in particular that the second assertion $\Lambda_0(\mathbb{D}^2)\ge \pi e$ of Theorem \ref{MainThm}, Part (1), is not satisfied. In some sense, this is an additional motivation for the nice approach of \cite{CarlesonChang}, proving the existence of an extremal for $(I^0_{4\pi}(\mathbb{D}^2))$ via asymptotic analysis. As an illustration and a very particular case of Theorem \ref{MainThm}, we get the following corollary.

\begin{cor}[Case $\Omega=\mathbb{D}^2$]\label{CorD2} Assume that $\Omega=\mathbb{D}^2$. Let $c'\neq 0$ and $(a',b')\in\mathcal{E}$ be given, where $\mathcal{E}$ is as in \eqref{DefERonde}. Let $g_\infty$ be as in \eqref{HypgDisk}.
\begin{enumerate}  
\item If we assume \fbox{$a'>2$ or $c'>0$}, then for all even function $g\in C^2(\mathbb{R})$, zero in a neighborhood of $0$, such that $g>-1$ and 
\begin{equation}\label{Comparison}
g^{(i)}(t)=g_\infty^{(i)}(t)(1+o(1))
\end{equation}
 as $t\to +\infty$ for all $i\in\{0,1,2\}$, $(I_{4\pi}^g(\mathbb{D}^2))$ admits an extremal.
\item If we assume \fbox{$a'<2$ and $c'<0$}, there exists an even function $g\in C^2(\mathbb{R})$, zero in a neighborhood of $0$, such that $g>-1$ and such that \eqref{Comparison} holds true, while $(I_{4\pi}^g(\mathbb{D}^2))$ admits no extremal.
\end{enumerate}
\end{cor}
\noindent Our main concern in Corollary \ref{CorD2} is to write a readable statement. In this result, the existence of an extremal in the unperturbed case $g\equiv 0$ is recovered for quickly decaying $g$'s, namely if $a'>2$ (see \cite{MartMan}). But a threshold phenomenon appears (only if $c'<0$) and there are no more extremal for slowlier decaying $g$'s, namely for $a'<2$. Note that Theorem \ref{MainThm} also allows to point out the existence of a threshold $c'<0$ in the border case $a'=2, b'=0$. Indeed, proving Corollary \ref{CorD2} basically reduces to give an explicit formula for $l$ in \eqref{Ratio}, which only depends on $\Omega$ and on the asymptotics of $g$ at $+\infty$ and at $0$. On the contrary, we do not care about the precise asymptotics of $g$ in the following corollary, thus illustrating the role of $\Lambda_g(\Omega)$ in Theorem \ref{MainThm}.
\begin{cor}[Extremal for $\Lambda_g(\Omega)$ large]\label{NonAsymptCor}
Let $\Omega$ be a smooth bounded domain of $\mathbb{R}^2$. Let $\lambda_1(\Omega)>0$ be the first Dirichlet eigenvalue of $\Delta$ in $\Omega$ and $M$ be given as in \eqref{MaxRobFun}. Let $\bar{A}$ be such that $4(1+\bar{A})>\lambda_1(\Omega)\exp(1+M)$ and let $C>\bar{A}$ be given. Then there exists $R\gg 1$ such that $(I_{4\pi}^g(\Omega))$ admits an extremal for all $g$ satisfying \eqref{Propg} and 
\begin{equation}\label{NonAsymptCorEq}
\begin{split}
&g(0)=\bar{A},\quad g\ge g(0) \text{ in }  [1/R,R] \quad \text{and}\quad |g|\le C \text{ in }\mathbb{R}\,.
\end{split}
\end{equation}
\end{cor}
As a remark, in the process of the proof below (see Remark \ref{OpenPbm6}), we answer the very interesting Open problem 6 of \cite{MartMan}. \\ \\
\indent This paper is organized as follows. Theorem \ref{MainThm}, and Corollaries \ref{CorD2} and \ref{NonAsymptCor} are proved in Section \ref{SectProofs}. Theorem \ref{MainThm} follows from Propositions \ref{EnerExp} and \ref{NonExExtProp}, proved in Section \ref{SectEnerExp}. Both Propositions \ref{EnerExp} and \ref{NonExExtProp} are consequences of key Lemma \ref{LemBlowUpAnalysis}, which is proved in Section \ref{SectBlowUpAnalysis}, using some radial analysis results obtained in Appendix \ref{Appendix}.

\section{Proof of the main results}\label{SectProofs}
\noindent We begin by proving Corollary \ref{CorD2}, assuming that Theorem \ref{MainThm} holds true.
\begin{proof}[Proof of Corollary \ref{CorD2}] The first part of Corollary \ref{CorD2} is a straightforward consequence of the first part of Theorem \ref{MainThm}: plugging the formulas of \eqref{Examples} in \eqref{Ratio}, we get that $l>0$ for $g$ as in Case (1) of Corollary \ref{CorD2}. In order to prove the second part of Corollary \ref{CorD2}, we apply the second part of Theorem \ref{MainThm}. Let $\chi$ be a smooth nonnegative function in $\mathbb{R}$ such that $\chi(t)=0$ for all $t\le 1/2$ and $\chi(t)=1$ for all $t\ge 1$. By the Sobolev inequality and standard integration theory, we can check that $g_R:=g_\infty\times  \chi(\cdot/R)$ satisfies $\Lambda_{g_R}(\mathbb{D}^2)\to \Lambda_0(\mathbb{D}^2)$ as $R\to +\infty$. Then, by \eqref{Examples}, \eqref{CondDisk}, assuming $a'<2, c'<0$, the second part of Theorem \ref{MainThm} applies, starting from $g=g_R$, for $R\gg 1$ fixed sufficiently large. Observe that, for all given $N\gg 1$, $(g_R)_N$ (given by \eqref{gN} for $g=g_R$) satisfies \eqref{Comparison}. Corollary \ref{CorD2} is proved.  
\end{proof}
\begin{proof}[Proof of Corollary \ref{NonAsymptCor}] Let $\Omega, \bar{A}, \lambda_1(\Omega), C$ be as in the statement of the corollary. By Theorem \ref{NonAsymptCor}, it is sufficient to prove that there exists $R\gg 1$ such that for all $g$ satisfying \eqref{Propg} and \eqref{NonAsymptCorEq}, we have that $\Lambda_g(\Omega)\ge\pi \exp(1+M)$, where $\Lambda_g(\Omega)$ is as in \eqref{NonlinEigen}. Let $v>0$ in $\Omega$ be the first eigenvalue of $\Delta$ normalized according to $\|v\|_{H^1_0}^2=4\pi$. For all $g$  satisfying \eqref{NonAsymptCorEq}, we have that
\begin{equation*}
\begin{split}
\Lambda_g(\Omega)&\ge \int_\Omega \left((1+g(0)) v^2+(g(v)-g(0))(1+v^2) \right) dx\\
&\ge(1+\bar{A})\frac{4\pi}{\lambda_1(\Omega)}+\int_{\{v\not\in[1/R,R]\}}(g(v)-g(0))(1+v^2) dx\,,
\end{split}
\end{equation*}
and, since we have 
$$\left|\int_{\{v\not\in[1/R,R]\}}(g(v)-g(0))(1+v^2) dx \right|\le (|\bar{A}|+C)(1+\|v\|_{L^\infty}^2) \left|\{v\not\in[1/R,R]\}\right|\to 0 $$
as $R\to +\infty$, we get the result using that $4(1+\bar{A})>\lambda_1(\Omega)\exp(1+M)$. 
\end{proof}

The following proposition is the core of the argument to get the existence of an extremal in Theorem \ref{MainThm}, Part (1). Its proof is postponed in Section \ref{SectEnerExp}. It uses the tools developed in Druet-Thizy \cite{DruThiI} that allow us to push the asymptotic analysis of a concentrating sequence of extremals $(u_\varepsilon)_\varepsilon$ further than in previous works. In the process of the proof of Proposition \ref{EnerExp} (see Lemma \ref{Step3ExistExtremal}), we show first that a concentration point $\bar{x}$ of such $u_\varepsilon$'s realizes $M$ in \eqref{MaxRobFun}. But in the case where $|B(\gamma)|$ matters in \eqref{Ratio} or, in other words, where $\gamma^{3}|A(\gamma)|+\gamma^{-1}\lesssim |B(\gamma)|$ as $\gamma\to +\infty$, we also show that $S$ in \eqref{MaxRobFun} has to be attained at $\bar{x}$. 

\begin{prop}\label{EnerExp}
Let $\Omega$ be a smooth bounded domain of $\mathbb{R}^2$. Let $g$ be such that \eqref{Propg} and \eqref{InftyBehavior}-\eqref{ZeroBehavior} hold true, for $H$ as in \eqref{DefH}, and let $A$, $B$ and $F$ be thus given. Let $(u_\varepsilon)_\varepsilon$ be a sequence of nonnegative functions such that $u_\varepsilon$ is a maximizer for $(I_{4\pi(1-\varepsilon)}^g(\Omega))$, for all $0<\varepsilon\ll 1$. Assume that 
\begin{equation}\label{WeakConvToZero}
u_\varepsilon\rightharpoonup 0\text{ in }H^1_0\,,
\end{equation}
as $\varepsilon\to 0$.
 Then, $\|u_\varepsilon\|_{H^1_0}^2=4\pi(1-\varepsilon)$, there exists a sequence $(\lambda_\varepsilon)_\varepsilon$ of real numbers such that $u_\varepsilon$ solves in $H^1_0$
\begin{equation}\label{ELEq}
\begin{cases}
&\Delta u_\varepsilon=\lambda_\varepsilon u_\varepsilon H(u_\varepsilon) \exp(u_\varepsilon^2),\quad u_\varepsilon>0\text{ in }\Omega\,,\\
&u_\varepsilon=0\text{ on }\partial\Omega\,,
\end{cases}
\end{equation}
$u_\varepsilon\in C^{1,\theta}(\bar{\Omega})$ \textsc{(}$0<\theta<1$\textsc{)} and we have that
\begin{equation}\label{LossCompactness}
\gamma_\varepsilon:=\max_{y\in\Omega} u_\varepsilon \to +\infty\,.
\end{equation}
Moreover, we have that
\begin{equation}\label{Level}
\lim_{\varepsilon\to 0}\int_\Omega (1+g(u_\varepsilon)) \exp(u_\varepsilon^2) dx=|\Omega|(1+g(0))+\pi\exp(1+M)
\end{equation} 
 and that
\begin{equation}\label{EqEnerExp}
\|u_\varepsilon\|_{H^1_0}^2=4\pi\left(1+I(\gamma_\varepsilon)+o\left(\gamma_\varepsilon^{-4}+|A(\gamma_\varepsilon)|+\gamma_\varepsilon^{-3}|B(\gamma_\varepsilon)| \right)\right)
\end{equation}
as $\varepsilon\to 0$, where 
\begin{equation}\label{Eq2EnerExp}
I(\gamma_\varepsilon):=\gamma_\varepsilon^{-4}+A(\gamma_\varepsilon)/2+4\gamma_\varepsilon^{-3}\exp(-1-M) B(\gamma_\varepsilon) S\,,
\end{equation}
where $|\Omega|$ stands for the volume of the domain $\Omega$ and where $M$ and $S$ are as in \eqref{MaxRobFun}.
\end{prop}

\begin{rem}\label{OpenPbm6}
Let $g,H$ be such that \eqref{Propg}, \eqref{DefH}, \eqref{InftyBehavior}-\eqref{LimitB} hold true. Let $u_\varepsilon$ be a maximizer for $(I^g_{4\pi(1-\varepsilon)})$ such that \eqref{WeakConvToZero} holds true, as in Proposition \ref{EnerExp}. Then, for such a sequence $(u_\varepsilon)_\varepsilon$ satisfying in particular  \eqref{ELEq} and \eqref{LossCompactness}, we get in the process of the proof (see \eqref{AGammaSmall} below) that the term $I(\gamma_\varepsilon)$ in \eqref{EqEnerExp} is necessarily smaller than $o(\gamma_\varepsilon^{-2})$ as $\varepsilon\to 0$. Moreover this threshold $o(\gamma_\varepsilon^{-2})$ is sharp, in the sense that this term may be for instance of size $\gamma_\varepsilon^{-(2+a')}\,,$ for all given $a'\in(0,2]$. This can be seen by picking an appropriate $g$ such that $I^g_{4\pi}(\Omega)$ has no extremal, as in Corollary \ref{CorD2}, and by using Proposition \ref{EnerExp}. Observe that, for such a $g$, assumption \eqref{WeakConvToZero} is indeed automatically true. This gives an answer to Open Problem 6 in \cite{MartMan}.
\end{rem}

\begin{proof}[Proof of Theorem \ref{MainThm}, Part (1): existence of an extremal for $(I^g_{4\pi}(\Omega))$.] We first prove the existence of an extremal stated in Part (1) of Theorem \ref{MainThm}. Let $g$ be such that \eqref{Propg} and \eqref{InftyBehavior}-\eqref{ZeroBehavior} hold true, for $H$ as in \eqref{DefH}, and let $A$, $B$ and $F$ be thus given. Assume either that $l>0$ in \eqref{Ratio} or that $\Lambda_g(\Omega)\ge \pi \exp(1+M)$. Using Lemma \ref{SubcriticalCase}, let $(u_\varepsilon)_\varepsilon$ be a sequence of nonnegative functions such that $u_\varepsilon$ is a maximizer for $(I_{4\pi(1-\varepsilon)}^g(\Omega))$, for all $0<\varepsilon\ll 1$. Then, up to a subsequence, $(u_\varepsilon)_\varepsilon$ converges a.e. and weakly in $H^1_0$ to some $u_0$. Independently, we check that
\begin{equation}\label{PfThm1}
\lim_{\varepsilon\to 0}C_{g,4\pi(1-\varepsilon)}(\Omega)= C_{g,4\pi}(\Omega)\,,
\end{equation}
where $C_{g,\alpha}(\Omega)$ is as in $(I_\alpha^g(\Omega))$. Indeed, if one assumes by contradiction that the $C_{g,4\pi(1-\varepsilon)}(\Omega)$'s increase to some $\bar{l}<C_{g,4\pi}(\Omega)$ as $\varepsilon\to 0$, then we may choose some nonnegative $u$ such that $\|u\|_{H^1_0}^2 \le 4\pi$ and $\int_{\Omega} (1+g(u))\exp(u^2) dx>\bar{l}$. But, picking $v_\varepsilon=u\sqrt{1-\varepsilon}$, we have that $\|v_\varepsilon\|_{H^1_0}^2<4\pi$, and 
$$\lim_{\varepsilon\to 0} \int_{\Omega}(1+g(v_\varepsilon)) \exp(v_\varepsilon^2) dx=\int_\Omega (1+g(u))\exp(u^2) dx\,, $$
by the dominated convergence theorem, using \eqref{Propg}, $v_\varepsilon^2\le u^2$ and $\exp(u^2)\in L^1(\Omega)$. But this contradicts the definition of $\bar{l}$ and concludes the proof of \eqref{PfThm1}. Now, by \eqref{PfThm1} and since $\|u_0\|^2_{H^1_0}\le 4\pi$, in order to get that $u_0$ is the extremal for $(I^g_{4\pi}(\Omega))$ we look for, it is sufficient to prove that
\begin{equation}\label{PfThm2}
\lim_{\varepsilon\to 0}\int_\Omega (1+g(u_\varepsilon)) \exp(u_\varepsilon^2)~ dx=\int_\Omega (1+g(u_0)) \exp(u_0^2)~ dx\,. 
\end{equation}
  If $u_0=0$, then Proposition \ref{EnerExp} gives a contradiction: either by \eqref{Level} and \eqref{PfThm1} if $\Lambda_g(\Omega)\ge \pi \exp(1+M)$, since it is clear that
$$ C_{g,4\pi}(\Omega)>\Lambda_g(\Omega)+(1+g(0))|\Omega|\,,$$ 
or by \eqref{EqEnerExp}-\eqref{Eq2EnerExp} if $l>0$, since $\|u_\varepsilon\|_{H^1_0}\le 4\pi$. Thus, we necessarily have that $u_0\neq 0$. Then, noting that $\|u_\varepsilon-u_0\|_{H^1_0}^2\le 4\pi-\|u_0\|_{H^1_0}^2+o(1)$, the standard Moser-Trudinger inequality $(I_{4\pi}^0(\Omega))$ and some integration theory give that \eqref{PfThm2} still holds true, and Part (1) of Theorem \ref{MainThm} is proved in any case.
\end{proof}

The following proposition is the core of the argument to get the non-existence of an extremal in Theorem \ref{MainThm}, Part (2). Its proof is postponed in Section \ref{SectEnerExp}.

\begin{prop}\label{NonExExtProp}
Let $\Omega$ be a smooth bounded domain of $\mathbb{R}^2$. Let $g$ be such that \eqref{Propg} and \eqref{InftyBehavior}-\eqref{ZeroBehavior} hold true, for $H$ as in \eqref{DefH}, and let $A$, $B$ and $F$ be thus given. Assume that $\Lambda_g(\Omega)<\pi \exp(1+M)$, where $M$ is as in \eqref{MaxRobFun} and $\Lambda_g(\Omega)$ as in \eqref{NonlinEigen}. Assume that there exists a sequence of positive integers $(N_\varepsilon)_\varepsilon$ such that 
\begin{equation}\label{NToInfty}
\lim_{\varepsilon\to 0} N_\varepsilon=+\infty
\end{equation}
and such that $(I_{4\pi}^{g_{N_\varepsilon}}(\Omega))$ admits a nonnegative extremal $u_\varepsilon$ for all $\varepsilon>0$, where $g_{N_\varepsilon}$ is as in \eqref{gN}. Then we have \eqref{WeakConvToZero} and that $\|u_\varepsilon\|_{H^1_0}^2=4\pi$ for all $0<\varepsilon\ll 1$. Moreover, we have $u_\varepsilon\in C^{1,\theta}(\bar{\Omega})$ \textsc{(}$0<\theta<1$\textsc{)}, \eqref{LossCompactness} and that 
\begin{equation}\label{EqEnerExp2}
\|u_\varepsilon\|_{H^1_0}^2\le 4\pi\left(1+I(\gamma_\varepsilon)+o\left(\gamma_\varepsilon^{-4}+|A(\gamma_\varepsilon)|+\gamma_\varepsilon^{-3}|B(\gamma_\varepsilon)| \right)\right)
\end{equation}
as $\varepsilon\to 0$, where $I(\gamma_\varepsilon)$ is given by \eqref{Eq2EnerExp}.
\end{prop}

\begin{proof}[Proof of Theorem \ref{MainThm}, Part (2): non-existence of an extremal for $(I_{4\pi}^{g_N}(\Omega))$, $N\ge N_0$] Let $g$ be such that \eqref{Propg} and \eqref{InftyBehavior}-\eqref{ZeroBehavior} hold true, for $H$ as in \eqref{DefH}, and let $A$, $B$ and $F$ be thus given. Assume $l<0$ and $\Lambda_g(\Omega)<\pi \exp(1+M)$, where $l$ is as in \eqref{Ratio}, $\Lambda_g$ as in \eqref{NonlinEigen} and $M$ as in \eqref{MaxRobFun}. In order to prove Part (2) of Theorem \ref{MainThm}, we assume by contradiction that there exists a sequence $(N_\varepsilon)_\varepsilon$ of positive integers satisfying \eqref{NToInfty} and such that $(I_{4\pi}^{g_{N_\varepsilon}}(\Omega))$ admits an extremal, for $g_{N_\varepsilon}$ as in \eqref{gN}. We let $(u_\varepsilon)_\varepsilon$ be a sequence of nonnegative functions such that $u_\varepsilon$ is a maximizer for $(I_{4\pi}^{g_{N_\varepsilon}}(\Omega))$, for all $\varepsilon>0$. But this is not possible by Proposition \ref{NonExExtProp}, since $\|u_\varepsilon\|_{H^1_0}^2=4\pi$ contradicts \eqref{EqEnerExp2}, since we also assume now $l<0$. This concludes the proof of Part (2) of Theorem \ref{MainThm}.
\end{proof}

\section{Blow-up analysis in the strongly perturbed Moser-Trudinger regime}\label{SectBlowUpAnalysis}
In this section, we aim to prove the main blow-up analysis results that we need to get both Propositions \ref{EnerExp} and \ref{NonExExtProp}. The following preliminary lemma deals with the existence of an extremal for the perturbed Moser-Trudinger inequality $(I_\alpha^g(\Omega))$ in the subcritical case $0<\alpha<4\pi$. Its proof relies on integration theory combined with $(I_{4\pi}^0(\Omega))$, and on standard variational techniques. It is omitted here and the interested reader may find more details in the proof of Proposition 6 of \cite{MartMan}.
\begin{lem}\label{SubcriticalCase}
 Let $\Omega$ be a smooth bounded domain of $\mathbb{R}^2$. Let $g$ be such that \eqref{Propg} holds true. Then, $(I_\alpha^g(\Omega))$ admits a nonnegative extremal $u_\alpha$ for all $0<\alpha<4\pi$. Moreover, we have the following alternative
 \begin{enumerate}
\item $\text{either } \|u_\alpha\|_{H^1_0}^2<\alpha\text{ and } u_\alpha H(u_\alpha)=0\text{ a.e.}\,,$
\item $\text{or }\|u_\alpha\|_{H^1_0}^2=\alpha\text{ and there exists }\lambda\in\mathbb{R}\text{ such that } u_\alpha$ solves in $H^1_0$ the Euler-Lagrange equation \eqref{ELEq1}.
 \end{enumerate}
  
\end{lem}

\begin{rem}
The first alternative in Lemma \ref{SubcriticalCase} may occur in general, but does not if $t\mapsto (1+g(t))\exp(t^2)$ increases in $(0,+\infty)$.
\end{rem}
\noindent The following lemma investigates more precisely the behavior of $g$ and $H$, when we assume \eqref{Propg} together with \eqref{InftyBehavior}.
\begin{lem}\label{HBehavior} 
Let $\Omega$ be a smooth bounded domain of $\mathbb{R}^2$. Let $g$ be such that \eqref{Propg},  \eqref{InftyBehavior} and \eqref{ZeroBehavior} hold true, for $H$ as in \eqref{DefH}, and let $A, B$ and $\delta_0, \delta'_0, F, \kappa$ be thus given. Then we have that
\begin{equation}\label{ZeroBehaviorG}
\begin{split}
&
\begin{split}
\text{a) }\left(1+g\left(\frac{t}{\gamma} \right)\right)\exp\left(\frac{t^2}{\gamma^2} \right)=~ &(1+g(0))+\frac{2 B(\gamma) F(t)t}{\gamma(\kappa+1)}\\
&\quad\quad\quad+o\left(\frac{|B(\gamma)|}{\gamma}+\frac{1}{\gamma^2}\right)\,,
\end{split}
\\
&\text{ in }C^0_{loc}\left((0,+\infty)_t\right), \text{ as }  \gamma\to +\infty\,,\\
&\text{b) }\exists C>0\,,\\
&~~ \left|\left(1+g\left(\frac{t}{\gamma} \right)\right)\exp\left(\frac{t^2}{\gamma^2} \right)-(1+g(0))\right|\le C \left(\frac{|B(\gamma)|}{\gamma}+\frac{1}{\gamma^2}\right) {t\exp(\delta_0' t)}\,,\\
&\text{ for all }\gamma\gg 1 \text{ and  all }0\le t\le {\gamma}\,,\\
&\text{c) }\|g\|_{L^\infty(\mathbb{R})}<+\infty\,,
\end{split}
\end{equation}
and that
\begin{equation}\label{AsymptG}
\begin{split}
&\text{a) }{1+g\left(\gamma-\frac{t}{\gamma}\right)}={H(\gamma)}\left(1+A(\gamma) \left(t+\frac{1}{2}\right)+o(|A(\gamma)|+\gamma^{-4})\right)\,,\\
&\text{ in }C^0_{loc}\left((0,+\infty)_t\right), \text{ as }  \gamma\to +\infty\,,\\
&\text{b) }\exists C>0\,, \left|1+g\left(\gamma-\frac{t}{\gamma}\right)-H(\gamma) \right|\le C|H(\gamma)|(|A(\gamma)|+\gamma^{-4}) {\exp(\delta_0 t)}\,,\\
&\text{ for all }\gamma\gg 1 \text{ and  all }0\le t\le \gamma\,.
\end{split}
\end{equation}
In particular, we have that
\begin{equation}\label{EqHBehavior}
H(\gamma)\to 1\text{ as }\gamma\to +\infty.
\end{equation}
\end{lem}

\begin{proof}[Proof of Lemma \ref{HBehavior}] We first prove \eqref{EqHBehavior}. Using \eqref{Croissance}, we write
\begin{equation}\label{Interm000}
(1+g(r))\exp(r^2)-(1+g(0))=2\int_0^r s H(s) \exp(s^2) ds\,,
\end{equation}
for all $r\ge 0$. Then, as $\gamma\to +\infty$, setting $r=\gamma$, we can write 
\begin{equation*}
\begin{split}
&1+g(\gamma)\\
&= \exp(-\gamma^2)\left(1+g(0) \right)+2\int_0^{\gamma^2} \left(1-\frac{u}{\gamma^2}\right)H\left(\gamma-\frac{u}{\gamma}\right) \exp\left(-2u+\frac{u^2}{\gamma^2} \right) du\,,\\
&=O\left(\exp(-\gamma^2)\right)+2H(\gamma)\int_0^{\gamma^2} \left(1-\frac{u}{\gamma^2}\right)\exp\left(-2u+\frac{u^2}{\gamma^2} \right) du\,, \\
&\quad +O\left(|H(\gamma)|(|A(\gamma)|+\gamma^{-4})\int_0^{\gamma^2} \exp(-(1-\delta_0)u)\exp\left(-u\left(1-\frac{u}{\gamma^2} \right)\right)du \right)\,,\\
&=O\left(\exp(-\gamma^2)\right)+H(\gamma)\left(1+\exp(-\gamma^2) \right)+o(H(\gamma))\,,
\end{split}
\end{equation*}
using \eqref{InftyBehavior}. This proves \eqref{EqHBehavior} since $g$ satisfies \eqref{Propg}.  Observe that parts $a)$ and $b)$ of \eqref{ZeroBehaviorG} follow from \eqref{ZeroBehavior} and \eqref{Interm000} with $r=t/\gamma$, while part $c)$ of \eqref{ZeroBehaviorG} is a straightforward consequence of \eqref{Propg}. We prove now part $b)$ of \eqref{AsymptG}. As $\gamma\to +\infty$, we write for all $0\le t\le \gamma$ 
\begin{equation*}
\begin{split}
&\left(1+g\left(\gamma-\frac{t}{\gamma} \right) \right)\exp\left(\left(\gamma-\frac{t}{\gamma} \right)^2\right)-\left(1+g(\gamma-1)\right) \exp((\gamma-1)^2)\,,\\
&=2\int_{\gamma-1}^{\gamma-\frac{t}{\gamma}} r H(r) \exp(r^2) dr\,,\\
&=2\int_{t}^{\gamma}\left(1-\frac{u}{\gamma^2}\right)H\left(\gamma-\frac{u}{\gamma}\right) \exp\left(\gamma^2-2u+\frac{u^2}{\gamma^2} \right) du\,,\\
&=H(\gamma)\left( \exp\left(\left(\gamma-\frac{t}{\gamma} \right)^2\right)-\exp((\gamma-1)^2)\right)\\
&\quad +O\left(|H(\gamma)|(|A(\gamma)|+\gamma^{-4})\int_t^\gamma\exp\left(\gamma^2-(2-\delta_0)u\right) du\right) \,,
\end{split}
\end{equation*}
using $b)$ in \eqref{InftyBehavior}. Multiplying the above identity by $\exp(-(\gamma-(t/\gamma))^2)$, using $t\le \gamma$, \eqref{Propg} and \eqref{EqHBehavior}, part $b)$ of \eqref{AsymptG} easily follows. Using now $a)$ of \eqref{InftyBehavior} in the above before last inequality, we also get part $a)$ of \eqref{AsymptG}.
\end{proof}

\noindent In the sequel, for all integer $N\ge 1$, we let $\varphi_{N}$ be given by (see also \eqref{FormulaPhi} below)
\begin{equation}\label{DefPhiN}
\varphi_N(t)=\sum_{k=N+1}^{+\infty}\frac{t^k}{k!}\,.
\end{equation}
The main results of this section are stated in the following lemma.

\begin{lem}\label{LemBlowUpAnalysis}
Let $\Omega$ be a smooth bounded domain of $\mathbb{R}^2$. Let $g$ be such that \eqref{Propg} and \eqref{InftyBehavior}-\eqref{ZeroBehavior} hold true, for $H$ as in \eqref{DefH}, and let $A$, $B$ and $F$ be thus given. Let $(\alpha_\varepsilon)_\varepsilon$ be a sequence of numbers in $(0,4\pi]$. Let $(N_\varepsilon)_\varepsilon$ be a sequence of positive integers. Assume that 
\begin{equation}\label{ExtremUEps}
\lim_{\varepsilon\to 0}\alpha_\varepsilon=4\pi\text{ and that }u_\varepsilon\text{ is an extremal for }(I_{\alpha_\varepsilon}^{g_{N_\varepsilon}}(\Omega))\,,
\end{equation}
 for all $0<\varepsilon\ll 1$, where $g_{N_\varepsilon}$ is as in \eqref{gN}. Assume in addition that we are in one of the following two cases:
 \begin{eqnarray}
 &\text{{\bf (Case 1)}}\quad\quad \quad \quad &\lim_{\varepsilon\to 0}N_\varepsilon= +\infty\,, \alpha_\varepsilon=4\pi\text{ for all }\varepsilon\,,\text{ and}\nonumber  \\
&~&\quad\quad\Lambda_g(\Omega)<\pi \exp(1+M)\,, \label{AssumptCase1}
\end{eqnarray}
where $\Lambda_g(\Omega)$ is as in \eqref{NonlinEigen} and $M$ as in \eqref{MaxRobFun}, or
\begin{equation*}
\text{{\bf (Case 2)}}\quad\quad\quad N_\varepsilon=1\text{ for all }\varepsilon\text{ and \eqref{WeakConvToZero} holds true}\,.
\end{equation*} 
Then, up to a subsequence, 
\begin{equation}\label{SaturCond}
\|u_\varepsilon\|^2_{H^1_0}=\alpha_\varepsilon\,,
\end{equation}
 $u_\varepsilon\in C^{1,\theta}(\bar{\Omega})$ $(0<\theta<1)$ solves
\begin{equation} \label{ELEqNEps}
\begin{cases}
&\Delta u_\varepsilon=\lambda_\varepsilon u_\varepsilon H_{N_\varepsilon}(u_\varepsilon) \exp(u_\varepsilon^2),\quad u_\varepsilon>0\text{ in }\Omega\,,\\
&u_\varepsilon=0\text{ on }\partial\Omega\,,
\end{cases}
\end{equation}
where $H_{N}(t)=1+g_{N}(t)+\frac{g_{N}'(t)}{2t}$. Moreover, we have \eqref{Level}, that
\begin{equation}\label{EqualLambdaEps}
\lambda_\varepsilon=\frac{4+o(1)}{\gamma_\varepsilon^2\exp(1+M)}\,,
\end{equation}
that
\begin{equation}\label{AGammaSmall1}
{A(\gamma_\varepsilon)}-2\xi_\varepsilon=o\left(\tilde{\zeta}_\varepsilon \right)\,,
\end{equation}
and that
\begin{equation}\label{XNotToBDry}
x_\varepsilon\to\bar{x}, \quad (\bar{x}\in K_\Omega)
\end{equation}
as $\varepsilon\to 0$, where $x_\varepsilon, \gamma_\varepsilon$ satisfy
\begin{equation}\label{XEps}
u_\varepsilon(x_\varepsilon)=\max_\Omega u_\varepsilon=\gamma_\varepsilon\to +\infty\,,
\end{equation}
as $\varepsilon\to 0$, where $\xi_\varepsilon$ is given by \eqref{DefXi}
\begin{equation}\label{DefXi}
\xi_\varepsilon=\frac{\gamma_\varepsilon^{2(N_\varepsilon-1)}}{\varphi_{N_\varepsilon-1}(\gamma_\varepsilon^2)(N_\varepsilon-1) !}\,,
\end{equation}
and where $\tilde{\zeta}_\varepsilon$ is given by
\begin{equation}\label{TildeZeta}
\tilde{\zeta}_\varepsilon=\max\left(\frac{1}{\gamma_\varepsilon^2}, |A(\gamma_\varepsilon)|, \xi_\varepsilon \right) \,.
\end{equation}
At last, \eqref{AFirstPointwEst}-\eqref{ExtEstimate} below hold true, for $\mu_\varepsilon$ as in \eqref{ScalRel} and $t_\varepsilon$ as in \eqref{TEps}. 
\end{lem}

Observe that $N_\varepsilon=1$ in ({\bf Case 2}) reduces to say that $g_{N_\varepsilon}=g$. From \eqref{MinorPhiNEps} obtained in the process of the proof below, we get that $\xi_\varepsilon=o(1/\gamma_\varepsilon^2)$ in $(\text{\bf Case 2})$, so that \eqref{AGammaSmall1} is then equivalent to
\begin{equation}\label{AGammaSmall}
{A(\gamma_\varepsilon)}=o\left(\frac{1}{\gamma_\varepsilon^2} \right)\,,
\end{equation}
as discussed in Remark \ref{OpenPbm6}. 

\begin{proof}[Proof of Lemma \ref{LemBlowUpAnalysis}] We start by several basic steps. First, a test function computation gives the following result.
\begin{Step}\label{Step1}
For all $g$ such that \eqref{Propg} holds true, we have that
\begin{equation}\label{TestFunPrel}
C_{g,4\pi}(\Omega)\ge |\Omega|(1+g(0))+\pi \exp(1+M)\,,
\end{equation}
where $C_{g,4\pi}(\Omega)$ is as in $(I^g_{\alpha}(\Omega))$ $(\alpha=4\pi)$ and where $M$ is as in \eqref{MaxRobFun}.
\end{Step}

\begin{proof}[Proof of Step \ref{Step1}]
In order to get \eqref{TestFunPrel}, it is sufficient to prove that there exists functions ${f}_\varepsilon\in H^1_0$ such that $\|{f}_\varepsilon\|_{H^1_0}^2=4\pi$ and such that
\begin{equation}\label{TestFunToProve}
\int_\Omega \left(1+g({f}_\varepsilon) \right) \exp\left({f}_\varepsilon^2\right) ~dy\ge |\Omega|(1+g(0))+\pi \exp(1+M)+o(1)\,,
\end{equation}
as $\varepsilon\to 0$. In order to reuse these computations later, we fix any sequence $(z_\varepsilon)_\varepsilon$ of poins in $\Omega$ such that 
\begin{equation}\label{TechnInter}
\frac{\varepsilon^2}{d(z_\varepsilon,\partial\Omega)^2}=o\left({\left(\log\frac{1}{\varepsilon}\right)^{-1}} \right)\,.
\end{equation}
  For $0<\varepsilon<1$, we let $v_\varepsilon$ be given by $v_\varepsilon(y)=\log \frac{1}{\varepsilon^2+|y-z_\varepsilon|^2}+\mathcal{H}_{z_\varepsilon,\varepsilon}$, where $\mathcal{H}_{z_\varepsilon,\varepsilon}$ is harmonic in $\Omega$ and such that $v_\varepsilon$ is zero on $\partial \Omega$. Then, by the maximum principle and \eqref{GreenFun}, we have that
\begin{equation}\label{AFirstEstTestFun}
\mathcal{H}_{z_\varepsilon,\varepsilon}(y)=\mathcal{H}_{z_\varepsilon}(y)+O\left(\frac{\varepsilon^2}{d(z_\varepsilon,\partial\Omega)^2} \right)\text{ for all }y\in\Omega\,,
\end{equation}
where $\mathcal{H}_{z_\varepsilon}$ is as in \eqref{GreenFun}. Then, integrating by parts, we compute
\begin{equation}\label{ASecondTestFun}
\begin{split}
\|v_\varepsilon\|_{H^1_0}^2&=\int_\Omega v_\varepsilon \Delta v_\varepsilon~dy\,,\\
&=\int_\Omega \frac{4}{\varepsilon^2\left(1+\frac{|z_\varepsilon-y|^2}{\varepsilon^2} \right)^2}\left(\log\frac{1}{\varepsilon^2}+\log\frac{1}{1+\frac{|y-z_\varepsilon|^2}{\varepsilon^2}}+\mathcal{H}_{z_\varepsilon,\varepsilon}(y) \right)~dy\,,\\
&=4\pi\left(\log \frac{1}{\varepsilon^2}+o(1)\right)-4\pi\left(1+o\left(1 \right) \right)\\
&\quad\quad \quad +4\pi\left(\mathcal{H}_{z_\varepsilon}(z_\varepsilon)+o(1) \right)\,,\\
&=4\pi \left(\log \frac{1}{\varepsilon^2}-1+\mathcal{H}_{z_\varepsilon}(z_\varepsilon) \right)+o(1)\,,
\end{split}
\end{equation}
where the change of variable $z=|y-z_\varepsilon|/\varepsilon$, \eqref{TechnInter}, \eqref{AFirstEstTestFun} and 
\begin{equation}\label{Compl1}
\mathcal{H}_{z_\varepsilon}(z_\varepsilon+\varepsilon z)=\mathcal{H}_{z_\varepsilon}(z_\varepsilon)+O\left(\frac{\varepsilon|z|}{d(z_\varepsilon,\partial\Omega)}\right)\,,
\end{equation}
(see for instance Appendix B in \cite{DruThiI}) are used. Let ${f}_\varepsilon$ be given by $4\pi v_\varepsilon^2={f}_\varepsilon^2 \|v_\varepsilon\|_{H^1_0}^2$. We can write 
\begin{equation*}
f_\varepsilon(y)^2=\frac{\left(\log\frac{1}{|z_\varepsilon-y|^2+\varepsilon^2} \right)^2+2\mathcal{H}_{z_\varepsilon,\varepsilon}(y)\log\frac{1}{|z_\varepsilon-y|^2+\varepsilon^2}+\mathcal{H}_{z_\varepsilon,\varepsilon}(y)}{\log \frac{1}{\varepsilon^2}\left(1+\frac{\mathcal{H}_{z_\varepsilon}(z_\varepsilon)-1}{\log\frac{1}{\varepsilon^2}}+o\left(\frac{1}{\log \frac{1}{\varepsilon}}\right) \right)}
\end{equation*}
using \eqref{ASecondTestFun}. Then, writing $\log\frac{1}{|z_\varepsilon-y|^2+\varepsilon^2}=\log\frac{1}{\varepsilon^2}+\log
\frac{1}{1+\frac{|z_\varepsilon-y|^2}{\varepsilon^2}}$, we get
\begin{equation}\label{AThirdTestFunEst}
\begin{split}
&\int_{B_{z_\varepsilon}(\check{r}_\varepsilon)\cap \Omega} (1+g(f_\varepsilon))\exp(f_\varepsilon^2)~dy \\
&=\int_{B_{z_\varepsilon}(\check{r}_\varepsilon)\cap\Omega}(1+o(1))\frac{\exp\left(-2\check{t}_\varepsilon(y)+2\mathcal{H}_{z_\varepsilon,\varepsilon}(y)-\mathcal{H}_{z_\varepsilon}(z_\varepsilon)+1 \right)}{\varepsilon^2}\times\\
&~\exp\left(\frac{\check{t}_\varepsilon^2}{\log\frac{1}{\varepsilon^2}}+O\left(\frac{1+\check{t}_\varepsilon}{\log\frac{1}{\varepsilon^2}}+\frac{1+\check{t}_\varepsilon^2}{\left(\log\frac{1}{\varepsilon^2}\right)^2} \right) \right)~dy\\
&= \pi\exp(\mathcal{H}_{z_\varepsilon}(z_\varepsilon)+1)(1+o(1))
\end{split}
\end{equation}
as $\varepsilon\to 0$, using \eqref{Propg}, \eqref{AFirstEstTestFun} and \eqref{Compl1}, where $\check{t}_\varepsilon(y)=\log\left(1+\frac{|z_\varepsilon-y|^2}{\varepsilon^2}\right)$ and where $\check{r}_\varepsilon$ is given by $\check{t}_\varepsilon(\check{r}_\varepsilon)=\frac{1}{2}\log\frac{1}{\varepsilon^2}$. Now, we can check that 
\begin{equation*}
\begin{split}
f_\varepsilon(y)^2 &\le \left(\log \frac{1}{\varepsilon^2}+O(1)\right)^{-1}\left(\log \frac{1}{|z_\varepsilon-y|^2}+O(1)\right)^2\,,\\
&\le \left(\log\frac{1}{|z_\varepsilon-y|^2}+O(1)\right)\times \left(\frac{1}{2}+o(1)\right) \text{ for all }y\in \Omega\backslash B_x(\check{r}_\varepsilon)\,,
\end{split}
\end{equation*}
using \eqref{GreenFun}, \eqref{AFirstEstTestFun} and our definition of $\check{r}_\varepsilon$, so that we also get
\begin{equation}\label{AFourthTestFun}
\int_{B_{z_\varepsilon}(\Omega\backslash \check{r}_\varepsilon)} (1+g(f_\varepsilon))\exp(f_\varepsilon^2)~dy\to (1+g(0))|\Omega|
\end{equation}
as $\varepsilon\to 0$, by the dominated convergence theorem, using \eqref{Propg}. Property \eqref{TestFunToProve} and then Step \ref{Step1} follow from \eqref{AThirdTestFunEst} and \eqref{AFourthTestFun}, choosing $z_\varepsilon\in K_\Omega$ as in \eqref{MaxRobFun}.
\end{proof}
 From now on, we make the assumptions of Lemma \ref{LemBlowUpAnalysis}. In particular, we assume that either ({\bf Case 1}), or ({\bf Case 2}) holds true. Given an integer $N\ge 1$, observe that Step \ref{Step1} applies to $g_N$, since $g_N$ satisfies \eqref{Propg}, if $g$ does. {Then, using $\alpha_\varepsilon=4\pi$ in ({\bf Case 1}), or \eqref{PfThm1} and $g_{N_\varepsilon}=g$ in ({\bf Case 2}), we get that
\begin{equation}\label{TestFunPrelBis}
|\Omega|(1+g(0))+\pi \exp(1+M)\le 
\begin{cases}
&C_{g_{N_\varepsilon},4\pi}\quad\quad\quad\text{ in ({\bf Case 1})}\,,\\
&C_{g_{N_\varepsilon},\alpha_\varepsilon}+o(1)\text{ in ({\bf Case 2})}\,,
\end{cases}
\end{equation}
as $\varepsilon\to 0^+$, where $C_{g,\alpha}(\Omega)$ is as in formula $(I^g_{\alpha}(\Omega))$ and where $M$ is as in \eqref{MaxRobFun}.} Let us rewrite now \eqref{ELEqNEps} in a more convenient way. Let $\Psi_N$ be given by
\begin{equation}\label{DefPsiN}
\Psi_N(t)=(1+g_N(t))\exp(t^2)\,.
\end{equation}
Observe in particular that 
\begin{equation*}
(1+g(t))(1+t^2)\le \Psi_N(t)\le (1+g(t))\exp(t^2)\,,
\end{equation*}
for all $t$ and all $N$, by \eqref{Propg}. Using \eqref{DefH}, \eqref{Croissance} and \eqref{gN}, we may rewrite \eqref{ELEqNEps} as 
\begin{equation}\label{ELEqNEpsBis}
\begin{cases}
&\Delta u_\varepsilon=\frac{\lambda_\varepsilon}{2}\Psi'_{N_\varepsilon}(u_\varepsilon),\quad u_\varepsilon>0\text{ in }\Omega\,,\\
&u_\varepsilon=0\text{ on }\partial\Omega\,,
\end{cases}
\end{equation}
with 
\begin{equation}\label{PsiPrime}
\begin{split}
\Psi'_N(t) ~&=2t H(t)\left(1+t^2+\varphi_N(t^2) \right)+2t(1+g(t))\left(\frac{t^{2N}}{N!}-t^2 \right)\\
&=2t H(t)\varphi_N(t^2)+2t\left(1+\frac{t^{2N}}{N!}\right)(1+g(t))+g'(t)(1+t^2)\,.
\end{split}
\end{equation}
Indeed, in \eqref{ELEqNEps}, it turns out that
\begin{equation}\label{HNEpsFormula}
H_{N}(t)=\frac{\Psi'_{N}(t)\exp(-t^2)}{2t}\,.
\end{equation}
Observe that by \eqref{Propg} and \eqref{EqHBehavior}, using the first line of \eqref{PsiPrime}, we clearly have that there exists $C>0$ such that
\begin{equation}\label{RoughEstimPsiPrime}
|\Psi'_{N_\varepsilon}(t)|\le C t\exp(t^2)
\end{equation}
for all $t\ge 0$ and all $\varepsilon$. In ({\bf Case 2}), \eqref{WeakConvToZero} is assumed to be true. We prove now that \eqref{WeakConvToZero} also holds true in ({\bf Case 1}).
\begin{Step}\label{Step2}
Assume that we are in \textsc{(}{\bf Case 1}\textsc{)}. Then \eqref{WeakConvToZero} holds true. Moreover, we have that
\begin{equation}\label{MinorPhiNEps}
\liminf_{\varepsilon\to 0}\underset{:=\delta_\varepsilon\in(0,1)}{\underbrace{\frac{\varphi_{N_\varepsilon}\left(\gamma_\varepsilon^2\right)}{\exp\left(\gamma_\varepsilon^2\right)}}}>0 \,,
\end{equation}
and, in other words, that 
\begin{equation}\label{MinorGammaByNEps}
\liminf_{\varepsilon\to 0} \frac{\gamma_\varepsilon^2-N_\varepsilon}{\sqrt{N_\varepsilon}}>-\infty\,,
\end{equation}
where $\gamma_\varepsilon=\esssup u_\varepsilon$ and $\varphi_N$ is as in \eqref{DefPhiN}.
\end{Step}

\begin{proof}[Proof of Step \ref{Step2}] By \eqref{ExtremUEps} and \eqref{TestFunPrelBis}, we get that
\begin{equation}\label{1ST2}
\int_\Omega \Psi_{N_\varepsilon}(u_\varepsilon) dy\ge  (1+g(0))|\Omega|+\pi\exp(1+M)\,.
\end{equation}
Writing now $$\Psi_{N}(t)=(1+g(0))+\left((1+g(t))(1+t^2)-(1+g(0)) \right)+(1+g(t))\varphi_N(t^2)$$
and using \eqref{Propg}, we also get
\begin{equation}\label{2ST2}
\int_\Omega \Psi_{N_\varepsilon}(u_\varepsilon) dy\le (1+g(0))|\Omega|+\Lambda_g(\Omega)+\int_\Omega(1+g(u_\varepsilon))\varphi_{N_\varepsilon}(u_\varepsilon^2) dy
\end{equation}
where $\Lambda_g$ is as in \eqref{NonlinEigen}. Then by \eqref{Propg} and \eqref{AssumptCase1}, we get from \eqref{1ST2} and \eqref{2ST2} that
\begin{equation}\label{3ST2}
\liminf_{\varepsilon\to 0} \int_\Omega \varphi_{N_\varepsilon}(u_\varepsilon^2) dy>0\,.
\end{equation}
Up to a subsequence, $u_\varepsilon\rightharpoonup u_0$ in $H^1_0$, for some $u_0\in H^1_0$ such that $\|u_0\|_{H^1_0}^2\le 4\pi$. Let $0<\beta\ll 1$ be given. We have that
\begin{equation}\label{4ST2}
u_\varepsilon^2\le (1+\beta)(u_\varepsilon-u_0)^2+\left(1+\frac{1}{\beta} \right)u_0^2\,.
\end{equation}
Independently, by Moser-Trudinger's inequality, we have that
\begin{equation}\label{MTCons}
u\in H^1_0\implies \forall p\in[1,+\infty),\quad \exp(u^2)\in L^p\,.
\end{equation}
If $u_0\not \equiv 0$, $\lim_{\varepsilon\to 0}\|u_\varepsilon-u_0\|_{H^1_0}^2<4\pi$ and, by \eqref{4ST2}, \eqref{MTCons}, Moser's and Hölder's inequalities, there exists $p_0>1$ such that $(\exp(u_\varepsilon^2))_\varepsilon$ is bounded in $L^{p_0}$. Then, by standard integration theory, since $\varphi_{N_\varepsilon}\le \exp$ in $[0,+\infty)$ and since $N_\varepsilon\to +\infty$ in ({\bf Case 1}), we get
$$u_0\not \equiv 0\implies \int_\Omega \varphi_{N_\varepsilon}(u_\varepsilon^2) dy=o(1) $$
as $\varepsilon\to 0$, which proves \eqref{WeakConvToZero}, in view of \eqref{3ST2}. Noting that the function $t\mapsto\varphi_N(t)\exp(-t)$ increases in $[0,+\infty)$, we can write
$$\int_\Omega \varphi_{N_\varepsilon}(u_\varepsilon^2) dy\le \frac{\varphi_{N_\varepsilon}(\gamma_\varepsilon^2)}{\exp(\gamma_\varepsilon^2)}\int_\Omega \exp(u_\varepsilon^2) dy\,, $$
and conclude that \eqref{MinorPhiNEps} holds true by \eqref{3ST2} and Moser's inequality. Observe that 
\begin{equation}\label{FormulaPhi}
\varphi_N(\Gamma)=\exp(\Gamma) \int_0^{\Gamma} \exp(-s) \frac{s^N}{N!} ds\,.
\end{equation}
Setting $\Gamma=\gamma_\varepsilon^2, N=N_\varepsilon$ and $s=N_\varepsilon+u\sqrt{N_\varepsilon}$, we clearly get \eqref{MinorGammaByNEps} from \eqref{MinorPhiNEps}.
\end{proof}

The next steps applies in both ({\bf Case 1}) and ({\bf Case 2}).
\begin{Step}\label{StepELSatisfied}
We have that \eqref{SaturCond}, \eqref{ELEqNEps} hold true, and that $u_\varepsilon$ is in $C^{1,\theta}(\bar{\Omega})$.
\end{Step}
\begin{proof}[Proof of Step \ref{StepELSatisfied}]
Since $u_\varepsilon\in L^1$, note that $\tilde{\mu}_\varepsilon$ given by $$\tilde{\mu}_\varepsilon(t):=|\{x\in\Omega\text{ s.t. }u_\varepsilon(x)>t\}|$$ is continuous in $[0,\gamma_\varepsilon]$. By \eqref{ExtremUEps} and the considerations as in Lemma \ref{SubcriticalCase}, either \eqref{SaturCond} and \eqref{ELEqNEps} hold true, or $\Psi'_{N_\varepsilon}(u_\varepsilon)=0$ almost everywhere in $\Omega$. Then, if we assume by contradiction that this second alternative holds true, since $\Psi'_{N_\varepsilon}$ is continuous, we get that $\Psi'_{N_\varepsilon}=0$ in $[0,\gamma_\varepsilon]$.  Then, since $\Psi_{N_\varepsilon}(0)=1$, there must be the case that
\begin{equation}\label{InterStepSatisfied}
(1+g(t))=\frac{1}{1+t^2+\varphi_{N_\varepsilon}(t^2)}
\end{equation}
for all $t\in [0,\gamma_\varepsilon]$. Now we prove that 
\begin{equation}\label{GammaToInfty}
\gamma_\varepsilon\to +\infty\,.
\end{equation}  as $\varepsilon\to 0$. This is merely a consequence of Step \ref{Step2} in ({\bf Case 1}). In ({\bf Case 2}), \eqref{WeakConvToZero} is assumed. Thus, up to a subsequence, $u_\varepsilon\to 0$ a.e. and if we assume by contradiction that $\gamma_\varepsilon=O(1)$, we contradict \eqref{ExtremUEps} and \eqref{TestFunPrelBis} by the dominated convergence theorem. This concludes the proof of \eqref{GammaToInfty}. Then \eqref{InterStepSatisfied} contradicts that $g(t)\to 0$ as $t\to +\infty$ in \eqref{Propg}, which proves \eqref{SaturCond} and \eqref{ELEqNEps}. By \eqref{MTCons}, the regularity of $u_\varepsilon$ comes from \eqref{ELEqNEps} under its form \eqref{ELEqNEpsBis}, by \eqref{Propg}, \eqref{EqHBehavior}, \eqref{PsiPrime} and standard elliptic theory {(see for instance Gilbarg-Trudinger \cite{Gilbarg}).}
\end{proof}

\noindent The previous steps give in particular that \eqref{XEps} makes sense and holds true.

\begin{Step}\label{StLambdaPositive}
There holds that $\lambda_\varepsilon>0$ for all $0<\varepsilon\ll 1$. Moreover
\begin{equation}\label{LambdaTo0}
\lambda_\varepsilon\to 0\,,
\end{equation}
as $\varepsilon\to 0$, where $\lambda_\varepsilon$ is as in \eqref{ELEqNEps}.
\end{Step}

\begin{proof}[Proof of Step \ref{StLambdaPositive}]
By \eqref{ExtremUEps} and \eqref{TestFunPrelBis}, we have that
$$\liminf_{\varepsilon\to 0} \int_\Omega \Psi_{N_\varepsilon}(u_\varepsilon) dx>0\,, $$
so that, by \eqref{Propg}, \eqref{WeakConvToZero}, \eqref{EqHBehavior}, \eqref{DefPsiN}, \eqref{PsiPrime} and integration theory
\begin{equation*}\label{StLambdaPos1}
\liminf_{\varepsilon\to 0}\int_\Omega\left(\Psi'_{N_\varepsilon}(u_\varepsilon)+2(1+g(u_\varepsilon))u_\varepsilon^3 \right) u_\varepsilon dx=+\infty \,.
\end{equation*}
But by \eqref{Propg}, \eqref{WeakConvToZero} and Rellich-Kondrachov's theorem, we get that
\begin{equation*}\label{StLambdaPos2}
\int_\Omega(1+g(u_\varepsilon)u_\varepsilon^4 dx=o(1)\,.
\end{equation*}
Then, multiplying \eqref{ELEqNEpsBis} by $u_\varepsilon$ and integrating by parts, we get that $\lambda_\varepsilon>0$ and
$$4\pi+o(1)=\int_\Omega |\nabla u_\varepsilon|^2 dx\gg \lambda_\varepsilon\,, $$
which proves \eqref{LambdaTo0}.
\end{proof}

 Then, using \eqref{EqHBehavior}, we may let $\mu_\varepsilon>0$ be given by 
\begin{equation}\label{ScalRel}
\lambda_\varepsilon H(\gamma_\varepsilon) \mu_\varepsilon^2 \gamma_\varepsilon^2 \varphi_{N_\varepsilon-1}(\gamma_\varepsilon^2)=4\,,
\end{equation}
where $\varphi_N$ is as in \eqref{DefPhiN}. Before starting the core of the proof, we would like to make a parenthetical remark.

\begin{rem}\label{RemLaneEmden} Note that \textsc{(}{\bf Case 1}\textsc{)} is particularly delicate to handle, since the nonlinearities $(\Psi'_{N_\varepsilon})_\varepsilon$ are not of \textit{uniform critical growth}, even in the very general framework of \cite[Definition 1]{DruetDuke}. A more intuitive way to see this is the following: if $(\tilde{\gamma}_\varepsilon)_\varepsilon$ is a sequence of positive real numbers such that $\tilde{\gamma}_\varepsilon\to +\infty$, but not too fast, in the sense that $\tilde{\gamma}_\varepsilon^2\ll N_\varepsilon$, then it can be checked with \eqref{Propg} and \eqref{EqHBehavior} that
$$\frac{\lambda_\varepsilon}{2} \Psi'_{N_\varepsilon}(\tilde{\gamma}_\varepsilon)=\tilde{\lambda}_\varepsilon(1+o(1)){\tilde{\gamma}_\varepsilon^{2N_\varepsilon+1}}\,, $$
as $\varepsilon\to 0$, where $\tilde{\lambda}_\varepsilon=\lambda_\varepsilon/(N_\varepsilon!)$. Then, in the regime $0\le u_\varepsilon\le \tilde{\gamma}_\varepsilon$, at least formally, \eqref{ELEqNEpsBis} looks at first order like the Lane-Emden problem, namely
\begin{equation}\tag{Lane-Emden problem}
\begin{cases}
&\Delta u_\varepsilon=\tilde{\lambda}_\varepsilon u_\varepsilon^{2N_\varepsilon+1}, \quad u_\varepsilon>0\text{ in }\Omega\,,\\
& u_\varepsilon=0\text{ on }\partial\Omega\,,\\
& N_\varepsilon\to +\infty\,,
\end{cases}
\end{equation}
for which very interesting, but very different concentration phenomena were pointed out (see for instance \cites{AdimGrossi,DeMarchisIanniPacella,DeMarchisIanniPacellaSurvey,LaneEmdenEspMusPis,RenWei,RenWei2}). A real difficulty to conclude the subsequent proofs is to extend the analysis developed in \cite{AdimurthiDruet,DruetDuke,DruThiI} for the Moser-Trudinger "purely critical" regime, in order to deal also with such other intermediate regimes. As a last remark, a much simpler version of the techniques developed here permits also to answer some open questions about the Lane-Emden problem, as performed in \cite{DupThi}.
\end{rem}

\noindent We let $t_\varepsilon$ be given by
\begin{equation}\label{TEps}
t_\varepsilon(x)=\log\left(1+\frac{|x-x_\varepsilon|^2}{\mu_\varepsilon^2} \right)\,.
\end{equation}
Here and in the sequel, for a radially symmetric function $f$ around of $x_\varepsilon$ (resp. around $0$), we will often write $f(r)$ instead of $f(x)$ for $|x-x_\varepsilon|=r$ (resp. $|x|=r$).
\begin{Step}\label{StMinor}
We have that
\begin{equation}\label{FirstResc}
\gamma_\varepsilon\left(\gamma_\varepsilon-u_\varepsilon(x_\varepsilon-\mu_\varepsilon\cdot) \right)\to T_0:=\log\left(1+|\cdot|^2 \right)\text{ in }C^{1,\theta}_{loc}(\mathbb{R}^2)\,,
\end{equation}
where $\gamma_\varepsilon, x_\varepsilon$ are as in \eqref{XEps} and $\mu_\varepsilon$ is as in \eqref{ScalRel}. Moreover, we have that
\begin{equation}\label{MinorLambdaEps}
\liminf_{\varepsilon\to 0}\lambda_\varepsilon \gamma_\varepsilon^2>0\,.
\end{equation}
\end{Step}
\noindent At this stage, we can observe that 
\begin{equation}\label{EquivLog1surMu2}
\log\frac{1}{\mu_\varepsilon^2}=\gamma_\varepsilon^2(1+o(1))\,,
\end{equation}
as $\varepsilon\to 0$, by \eqref{EqHBehavior}, \eqref{MinorPhiNEps}, \eqref{LambdaTo0}, \eqref{ScalRel}, \eqref{MinorLambdaEps}.
\begin{proof}[Proof of Step \ref{StMinor}]
We first sketch the proof of \eqref{FirstResc}. In ({\bf Case 2}), \eqref{FirstResc} follows closely Step 1 of the proof of \cite[Proposition 1]{DruetDuke}. Thus, we focus now on the the proof of \eqref{FirstResc} in ({\bf Case 1}). 
Observe that
\begin{equation}\label{TechnicalRk}
\sup_{t\in\mathbb{R}} \frac{t^{2N}}{N!}\exp(-t^2)=\frac{N^N}{N!}\exp(-N)\underset{N\to +\infty}{=}\frac{1+o(1)}{\sqrt{2\pi N}}
\end{equation}
by Stirling's formula. Then, by \eqref{Propg}, \eqref{EqHBehavior}, \eqref{XEps}, \eqref{PsiPrime}, \eqref{MinorPhiNEps} and \eqref{GammaToInfty}, we have that
\begin{equation}\label{ContrSecdMbrTau}
\begin{split}
\frac{\Psi_{N_\varepsilon}'(u_\varepsilon)}{2}~&=u_\varepsilon H(u_\varepsilon) \varphi_{N_\varepsilon}(u_\varepsilon^2)+u_\varepsilon (1+g(u_\varepsilon))\frac{u_\varepsilon^{2N_\varepsilon}}{N_\varepsilon !}+O\left(\gamma_\varepsilon^3 \right)\\
&\le (1+o(1))\gamma_\varepsilon\varphi_{N_\varepsilon-1}(\gamma_\varepsilon^2)\,.
\end{split}
\end{equation}
Let $\tau_\varepsilon$ be given in $(\Omega-x_\varepsilon)/\mu_\varepsilon$ by
\begin{equation}\label{DefTau}
u_\varepsilon(x_\varepsilon+\mu_\varepsilon\cdot)=\gamma_\varepsilon-\frac{\tau_\varepsilon}{\gamma_\varepsilon}\,.
\end{equation}
Then, since $\Delta \tau_\varepsilon=-\mu_\varepsilon^2\gamma_\varepsilon (\Delta u_\varepsilon)(x_\varepsilon+\mu_\varepsilon\cdot)$, we get from \eqref{ELEqNEpsBis}, \eqref{ScalRel} and \eqref{ContrSecdMbrTau}, that there exists $C>0$ such that $|\Delta \tau_\varepsilon|\le C$, while $\tau_\varepsilon\ge 0$, $\tau_\varepsilon(0)=0$. As in \cite[p.231]{DruetDuke}, we have that $\mu_\varepsilon=o(d(x_\varepsilon,\partial\Omega))$. Then, by standard elliptic theory, there exists $\tau_0$ such that
\begin{equation}\label{ConvLocTau}
\tau_\varepsilon\to \tau_0\text{ in }C^{1,\theta}_{loc}(\mathbb{R}^2)\,,
\end{equation}
as $\varepsilon\to 0$. Note that for all $\Gamma,T>0$ and all $N$, we have that
\begin{equation}\label{AlgRelat}
\varphi_N(T)=\varphi_N(\Gamma) \exp\left(-(\Gamma-T) \right)-\exp(T)\int_T^\Gamma \exp(-s)\frac{s^N}{N!} ds\,.
\end{equation}
Writing the previous identity for $N=N_\varepsilon-1$, $\Gamma=\gamma_\varepsilon^2$ and $T=u_\varepsilon^2=\gamma_\varepsilon^2-2\tau_\varepsilon+\frac{\tau_\varepsilon^2}{\gamma_\varepsilon^2}$, noting from \eqref{TechnicalRk} and \eqref{ConvLocTau} that
$$\int_{u_\varepsilon^2}^{\gamma_\varepsilon^2} \exp(-s)\frac{s^{N_\varepsilon-1}}{(N_\varepsilon-1)!} ds =O\left(\frac{1}{\sqrt{N_\varepsilon}}\right) $$
in $\mathbb{R}^2_{loc}$ and resuming the arguments to get \eqref{ContrSecdMbrTau}, we get that
\begin{equation}\label{LimitEqTau}
\Delta(-\tau_0)=4\exp(-\tau_0)
\end{equation}
using also \eqref{ELEqNEpsBis}, \eqref{MinorPhiNEps}  and \eqref{ScalRel}. Now, choosing $R\gg 1$ such that $|g(t)|<1$ and $H(t)>0$ for all $t\ge R$, we easily see that
\begin{equation}\label{PropTau1}
u_\varepsilon \left[\Psi_{N_\varepsilon}'(u_\varepsilon)\right]^-\le 2 \|t\mapsto tH(t)\|_{L^\infty(0,R)}\exp(R^2)u_\varepsilon+4u_\varepsilon^4
\end{equation}
 by \eqref{Propg}, \eqref{EqHBehavior} and \eqref{PsiPrime}, where $t^-=-\min(t,0)$. Then, we have that
 \begin{equation}\label{PropTau2}
 \frac{\lambda_\varepsilon}{2} \int_\Omega u_\varepsilon \left[\Psi_{N_\varepsilon}'(u_\varepsilon)\right]^+ dy=4\pi +o(1)\,,
\end{equation} 
by \eqref{ExtremUEps}, \eqref{ELEqNEpsBis}, \eqref{LambdaTo0} and \eqref{PropTau1}, where $t^+=\max(t,0)$. Then, for all $A\gg 1$, we get that
$$ 4\int_{B_0(A)} \exp(-\tau_0) dy\le \liminf_{\varepsilon\to 0}\lambda_\varepsilon \int_\Omega u_\varepsilon \left[\Psi_{N_\varepsilon}'(u_\varepsilon)\right]^+ dy\,,  $$
by \eqref{ConvLocTau} and, since $A$ is arbitrary, we get from \eqref{PropTau2} that $\int_{\mathbb{R}^2} \exp(-\tau_0) dy<+\infty$. Then, by the classification result Chen-Li \cite{ChenLi}, since $\tau_0\ge 0$ and $\tau_0(0)=0$, we get that $\tau_0(y)=\log(1+|y|^2)$. Thus \eqref{FirstResc} is proved by \eqref{ConvLocTau}. Similarly, we may also choose some $A_\varepsilon$'s, such that $A_\varepsilon\to +\infty$ and such that
\begin{equation}\label{InterTau1}
\frac{\lambda_\varepsilon}{2} \int_{B_{x_\varepsilon}(A_\varepsilon\mu_\varepsilon)} \Psi_{N_\varepsilon}(u_\varepsilon) dy=\frac{4\pi+o(1)}{\gamma_\varepsilon^2}\,.
\end{equation}
Since $0<\Psi_{N_\varepsilon}(t)\le(1+g(t))\exp(t^2)$ for all $t\ge 0$, and since $C_{g,4\pi}(\Omega)<+\infty$, we get \eqref{MinorLambdaEps} from \eqref{Propg} and \eqref{InterTau1}. This concludes the proof of Step \ref{StMinor}.
\end{proof}

By Step \ref{StMinor} and estimates in its proof, since we assume $\|u_\varepsilon\|_{H^1_0}^2\le 4\pi$, we get that
\begin{equation}\label{NoOtherConcentrationPt}
\lim_{R\to +\infty} \lim_{\varepsilon\to 0} \int_{\Omega\backslash B_{x_\varepsilon}(R\mu_\varepsilon)} (\Delta u_\varepsilon(y))^+ u_\varepsilon ~dy=0\,.
\end{equation}
We let $\Omega_\varepsilon$ be given by
\begin{equation}\label{DefOmegaEps}
\Omega_\varepsilon=
\begin{cases}
&\left\{y\in\Omega\text{ s.t. }\varphi_{N_\varepsilon-1}(u_\varepsilon(y)^2)\ge u_\varepsilon(y)^2+1 \right\}\text{ in ({\bf Case 1})}\,,\\
&\Omega\text{ in ({\bf Case 2})}\,.
\end{cases}
\end{equation}
 Now, despite the difficulty pointed out in Remark \ref{RemLaneEmden}, we are able to get the following weak, but global pointwise estimates.  

\begin{Step}\label{StPointwiseGlobalControl}
There exists $C>0$ such that
\begin{equation}\label{WeakPointwiseEst}
|\cdot-x_\varepsilon|^2 |\Delta u_\varepsilon| u_\varepsilon\le C \text{ in }\Omega_\varepsilon
\end{equation}
and such that
\begin{equation}\label{WeakGradEst}
|\cdot-x_\varepsilon| |\nabla u_\varepsilon| u_\varepsilon\le C \text{ in }\Omega_\varepsilon
\end{equation}
for all $\varepsilon$, where $\Omega_\varepsilon$ is as in \eqref{DefOmegaEps}.
\end{Step}
In ({\bf Case 2}), it is not so difficult to adapt the arguments of \cite[§3,4]{DruetDuke} to get Step \ref{StPointwiseGlobalControl}. Thus, in the proof of Step \ref{StPointwiseGlobalControl} just below, we assume that we are in ({\bf Case 1}). Then observe that $\Omega_\varepsilon\neq \emptyset$ by Step \ref{Step2}. Given $\eta_0\in(0,1)$, writing $$\varphi_{N_\varepsilon-1}(t N_\varepsilon)=\frac{t^{N_\varepsilon}N_\varepsilon^{N_\varepsilon}(1+o(1))}{N_\varepsilon!}$$ for all $0<\varepsilon\ll 1$, uniformly in $|t|\le \eta_0$, the unique positive solution $\Gamma_\varepsilon$ of $\varphi_{N_\varepsilon-1}(\Gamma_\varepsilon)=\Gamma_\varepsilon+1$ satisfies $\Gamma_\varepsilon=(1+o(1))\frac{N_\varepsilon}{e}$. Then, since $\varphi_{N_\varepsilon-1}/(1+\cdot)$ increases in $(0,+\infty)$, we clearly get that
\begin{equation}\label{UInOmegaEps}
(1+o(1))\frac{N_\varepsilon}{e}\le \min_{\Omega_\varepsilon} u_\varepsilon^2\,.
\end{equation}
\begin{proof}[Proof of Step \ref{StPointwiseGlobalControl}, Formula \eqref{WeakPointwiseEst}] As aforementioned, we still assume that we are in ({\bf Case 1}). Thus, in particular, we assume that $N_\varepsilon\to +\infty$ as $\varepsilon\to 0$. Assume now by contradiction that 
\begin{equation}\label{1WP}
\max_{y\in \Omega_\varepsilon} |y-x_\varepsilon|^2 |\Delta u_\varepsilon(y)| u_\varepsilon(y)=|y_\varepsilon-x_\varepsilon|^2 |\Delta u_\varepsilon(y_\varepsilon)| u_\varepsilon(y_\varepsilon)\to +\infty
\end{equation}
as $\varepsilon\to 0$, for some $y_\varepsilon$'s such that $y_\varepsilon\in\Omega_\varepsilon$. First for all sequence $(\check{z}_\varepsilon)_\varepsilon$ such that $\check{z}_\varepsilon\in \Omega_\varepsilon$, we have that $\Delta u_\varepsilon(\check{z}_\varepsilon)>0$, that $g'(u_\varepsilon(\check{z}_\varepsilon))=o(u_\varepsilon(\check{z}_\varepsilon))$ and that
\begin{equation}\label{EquivPsiN}
\Psi'_{N_\varepsilon}(u_\varepsilon(\check{z}_\varepsilon))=(1+o(1)) 2 u_\varepsilon(\check{z}_\varepsilon) \varphi_{N_\varepsilon-1}(u_\varepsilon(\check{z}_\varepsilon)^2)\,,
\end{equation}
as $\varepsilon\to 0$, using \eqref{Propg}, \eqref{InftyBehavior}, \eqref{EqHBehavior}, \eqref{PsiPrime} and \eqref{UInOmegaEps}. Besides, we have that
\begin{equation}\label{2WP}
u_\varepsilon(y_\varepsilon)\to +\infty\,,
\end{equation}
as $\varepsilon\to 0$. Let $\nu_\varepsilon>0$ be given by
\begin{equation}\label{5WP}
\nu_\varepsilon^2 |\Delta u_\varepsilon(y_\varepsilon)|u_\varepsilon(y_\varepsilon)=1\,.
\end{equation}
Then, in view of \eqref{1WP} and \eqref{5WP}, we have that
\begin{equation}\label{4WP}
\lim_{\varepsilon\to 0} \frac{|y_\varepsilon-x_\varepsilon|}{\nu_\varepsilon}=+\infty\,,
\end{equation}
and, in view of Step \ref{StMinor}, that
\begin{equation}\label{3WP}
\lim_{\varepsilon\to 0} \frac{|y_\varepsilon-x_\varepsilon|}{\mu_\varepsilon}=+\infty\,.
\end{equation}
For $R>0$, we set $\Omega_{R,\varepsilon}=B_{y_\varepsilon}(R\nu_\varepsilon)\cap \Omega$ and $\tilde{\Omega}_{R,\varepsilon}=(\Omega_{R,\varepsilon}-y_\varepsilon)/\nu_\varepsilon$. Up to harmless rotations and since $\Omega$ is smooth, we may assume that there exists $B\in [0,+\infty]$ such that $\tilde{\Omega}_{0,R} \to (-\infty,B)\times \mathbb{R}$ as $R\to+\infty$, where $\tilde{\Omega}_{\varepsilon,R}\to \tilde{\Omega}_{0,R}$ as $\varepsilon\to 0$. In this proof, for $z\in \tilde{\Omega}_{R,\varepsilon}$, we write $z_\varepsilon=y_\varepsilon+\nu_\varepsilon z\in \Omega_{R,\varepsilon}$. Let $\tilde{u}_\varepsilon$ be given by 
\begin{equation}\label{6WP}
\tilde{u}_\varepsilon(z)=u_\varepsilon(y_\varepsilon)\left(u_\varepsilon(z_\varepsilon)-u_\varepsilon(y_\varepsilon) \right)\,,
\end{equation}
so that we get
\begin{equation}\label{7WP}
\left(\Delta \tilde{u}_\varepsilon\right)(z)=\frac{(\Delta u_\varepsilon)(z_\varepsilon)}{(\Delta u_\varepsilon)(y_\varepsilon)}=\frac{\Psi'_{N_\varepsilon}(z_\varepsilon)}{\Psi'_{N_\varepsilon}(y_\varepsilon)}\,.
\end{equation}
First, we prove that for all $R>0$, there exists $C_R>0$ such that
\begin{equation}\label{BdLapl}
|\Delta \tilde{u}_\varepsilon|\le C_R\text{ in }\tilde{\Omega}_{R,\varepsilon}\,,
\end{equation}
for all $0<\varepsilon\ll 1$. Otherwise, by \eqref{7WP}, assume by contradiction that there exists $z_\varepsilon\in {\Omega}_{R,\varepsilon}$ such that 
\begin{equation}\label{ContradWP}
|\Psi'_{N_\varepsilon}(z_\varepsilon)|\gg \Psi'_{N_\varepsilon}(y_\varepsilon)
\end{equation}
 as $\varepsilon\to 0$. If, still by contradiction, $z_\varepsilon\not\in \Omega_\varepsilon$, we have that $u_\varepsilon(z_\varepsilon)<u_\varepsilon(y_\varepsilon)$, that
$$\varphi_{N_\varepsilon-1}(u_\varepsilon(z_\varepsilon)^2)< \varphi_{N_\varepsilon-1}(u_\varepsilon(y_\varepsilon)^2)\,, $$
by definition of $\Omega_\varepsilon$ and since $\varphi_N/(1+\cdot)$ increases in $[0,+\infty)$, and then that
$$|\Psi'_{N_\varepsilon}(u_\varepsilon(z_\varepsilon))|\lesssim u_\varepsilon(z_\varepsilon)\left(1+u_\varepsilon(z_\varepsilon)^2+\varphi_{N_\varepsilon-1}(u_\varepsilon(z_\varepsilon)^2)  \right)\lesssim \Psi'_{N_\varepsilon}(u_\varepsilon(y_\varepsilon)) \,,$$
using \eqref{Propg}, \eqref{EqHBehavior}, \eqref{PsiPrime}, \eqref{EquivPsiN} and $y_\varepsilon\in\Omega_\varepsilon$ again. This contradicts \eqref{ContradWP} and then it must be the case that $z_\varepsilon\in \Omega_\varepsilon$. Thus, since $y_\varepsilon$ is a maximizer on $\Omega_\varepsilon$ in \eqref{1WP}, we get from \eqref{4WP} and \eqref{ContradWP} that $u_\varepsilon(z_\varepsilon)\ll u_\varepsilon(y_\varepsilon)$. But this is not possible by \eqref{EquivPsiN} and \eqref{ContradWP}, which proves \eqref{BdLapl}. Now we prove that, for all $R>0$,
\begin{equation}\label{TildeUNeg}
\limsup_{\varepsilon\to 0} \sup_{z\in \tilde{\Omega}_{R,\varepsilon}} \tilde{u}_\varepsilon(z)\le 0\,.
\end{equation}
Until the end of this proof, we set $\tilde{\gamma}_\varepsilon:=u_\varepsilon(y_\varepsilon)$. If \eqref{TildeUNeg} does not hold true, since $\tilde{u}_\varepsilon(0)=0$ and by continuity, we may assume that there exist $z_\varepsilon\in \Omega_{R,\varepsilon}$ such that
\begin{equation}\label{ContradWP2}
\beta_\varepsilon:= \left[\tilde{\gamma}_\varepsilon\left(u_\varepsilon(z_\varepsilon)-\tilde{\gamma}_\varepsilon \right) \right]\to\beta_0\in(0,+\infty)\,,
\end{equation}
as $\varepsilon\to 0$. Since $u_\varepsilon(z_\varepsilon)>u_\varepsilon(y_\varepsilon)$ for $0<\varepsilon\ll 1$ by \eqref{ContradWP2}, we have that $z_\varepsilon\in \Omega_\varepsilon$. Moreover, since $y_\varepsilon$ is maximizing in \eqref{1WP}, we then get from \eqref{EquivPsiN}, \eqref{2WP} and \eqref{4WP} that
\begin{equation}\label{InterWP}
\varphi_{N_\varepsilon-1}(u_\varepsilon(z_\varepsilon)^2)\le (1+o(1))~\varphi_{N_\varepsilon-1}(\tilde{\gamma}_\varepsilon^2)\,.
\end{equation}
Independently, since $\varphi_N$ is convex, we get that
\begin{equation}\label{SuiteWP}
\begin{split}
\varphi_{N_\varepsilon-1}(u_\varepsilon(z_\varepsilon)^2)&\ge \varphi_{N_\varepsilon-1}(\tilde{\gamma}_\varepsilon^2)+\varphi_{N_\varepsilon-1}'(\tilde{\gamma}_\varepsilon^2)\left(u_\varepsilon(z_\varepsilon)^2-\tilde{\gamma}_\varepsilon^2 \right)\,,\\
&\ge \left(1+2\beta_0(1+o(1))\right) \varphi_{N_\varepsilon-1}(\tilde{\gamma}_\varepsilon^2)\,,
\end{split}
\end{equation}
using \eqref{ContradWP2} and $\varphi_N'(t)\ge \varphi_N(t)$ for $t\ge 0$. But \eqref{ContradWP2}-\eqref{SuiteWP} cannot hold true simultaneously, which proves \eqref{TildeUNeg}. As in \cite[p.231]{DruetDuke}, $\tilde{u}_\varepsilon(0)=0$, $u_\varepsilon=0$ on $\partial\Omega$, \eqref{BdLapl} and \eqref{TildeUNeg} imply that 
\begin{equation}\label{WPFarBdry}
\lim_{\varepsilon\to 0}\frac{d(y_\varepsilon,\partial\Omega)}{\nu_\varepsilon}=+\infty.
\end{equation}
Moreover, by standard elliptic theory, $\tilde{u}_\varepsilon(0)=0$, \eqref{BdLapl}, \eqref{TildeUNeg} and \eqref{WPFarBdry} give that 
\begin{equation}\label{LocConvWP}
\tilde{u}_\varepsilon\to u_0\text{ in }C^1_{loc}(\mathbb{R}^2)\,,
\end{equation}
as $\varepsilon\to 0$, for some $u_0\in C^1(\mathbb{R}^2)$. Given $R>0$, we prove now that
\begin{equation}\label{MinorLaplWP}
\liminf_{\varepsilon\to 0}\inf_{z\in\tilde{\Omega}_{R,\varepsilon}} (\Delta\tilde{u}_\varepsilon)(z) >0\,.
\end{equation}
Using \eqref{PsiPrime}, \eqref{2WP} and \eqref{LocConvWP}, we have that 
$$\Psi'_{N_\varepsilon}(u_\varepsilon)=2\tilde{\gamma}_\varepsilon \varphi_{N_\varepsilon-1}(u_\varepsilon^2)(1+o(1))+o(\tilde{\gamma}_\varepsilon^3)\,, $$
uniformly in $\Omega_{R,\varepsilon}$. Then, coming back to \eqref{7WP}, using \eqref{EquivPsiN} and $y_\varepsilon\in\Omega_\varepsilon$, we get that
$$(\Delta \tilde{u}_\varepsilon)(z)=(1+o(1))\frac{\varphi_{N_\varepsilon-1}(u_\varepsilon(z_\varepsilon)^2)}{\varphi_{N_\varepsilon-1}(\tilde{\gamma}_\varepsilon^2)}+o(1)\,, $$
uniformly in $z\in \tilde{\Omega}_{R,\varepsilon}$. Now, we write \eqref{AlgRelat} with $\Gamma=\tilde{\gamma}_\varepsilon^2$ and $T=u_\varepsilon^2$. Then, in order to conclude the proof of \eqref{MinorLaplWP}, using also \eqref{FormulaPhi}, it is sufficient to check that there exists $\eta_R<1$ such that
\begin{equation}\label{TechnicalRk2}
\begin{split}
I_\varepsilon:=\frac{\exp(u_\varepsilon^2)}{\varphi_{\tilde{N}_\varepsilon}(\tilde{\gamma}_\varepsilon^2)\exp\left(-\left(\tilde{\gamma}_\varepsilon^2-u_\varepsilon^2 \right) \right)}\int_{u_\varepsilon^2}^{\tilde{\gamma}_\varepsilon^2} \exp(-s) \frac{s^{\tilde{N}_\varepsilon}}{\tilde{N}_\varepsilon !} ds&=\frac{\int_{u_\varepsilon^2}^{\tilde{\gamma}_\varepsilon^2} \exp(-s) \frac{s^{\tilde{N}_\varepsilon}}{\tilde{N}_\varepsilon !} ds}{\int_{0}^{\tilde{\gamma}_\varepsilon^2} \exp(-s) \frac{s^{\tilde{N}_\varepsilon}}{\tilde{N}_\varepsilon !} ds}\,,\\
&\le \eta_R\,,
\end{split}
\end{equation}
for all $0<\varepsilon\ll 1$, uniformly in ${\Omega}_{R,\varepsilon}$, where $\tilde{N}_\varepsilon=N_\varepsilon-1$. If $u_\varepsilon\ge \tilde{\gamma}_\varepsilon$, the last inequality in \eqref{TechnicalRk2} is obvious. If now $u_\varepsilon< \tilde{\gamma}_\varepsilon$, we write
\begin{equation}\label{TechnicalRk3}
\begin{split}
I_\varepsilon &\le\frac{\int_{u_\varepsilon^2-\tilde{\gamma}_\varepsilon^2}^0 \exp(-t) \left(1+\frac{t}{\tilde{\gamma}_\varepsilon^2} \right)^{\tilde{N}_\varepsilon} dt}{\int_{2(u_\varepsilon^2-\tilde{\gamma}_\varepsilon^2)}^0 \exp(-t) \left(1+\frac{t}{\tilde{\gamma}_\varepsilon^2} \right)^{\tilde{N}_\varepsilon} dt}\\
&\le \frac{\int_{u_\varepsilon^2-\tilde{\gamma}_\varepsilon^2}^0 \exp\left(t\left(\frac{\tilde{N}_\varepsilon}{\tilde{\gamma}_\varepsilon^2}-1+O\left(\frac{\tilde{N}_\varepsilon t^2}{\tilde{\gamma}_\varepsilon^4} \right) \right)\right) dt}{\int_{2(u_\varepsilon^2-\tilde{\gamma}_\varepsilon^2)}^0 \exp\left(t\left(\frac{\tilde{N}_\varepsilon}{\tilde{\gamma}_\varepsilon^2}-1+O\left(\frac{\tilde{N}_\varepsilon t^2}{\tilde{\gamma}_\varepsilon^4} \right) \right)\right) dt}\\
&\le \eta_R
\end{split}
\end{equation}
using \eqref{LocConvWP}, where $I_\varepsilon$ is as in \eqref{TechnicalRk2}. We get the last inequality using \eqref{UInOmegaEps} and $y_\varepsilon\in\Omega_\varepsilon$: \eqref{TechnicalRk2} and then \eqref{MinorLaplWP} are proved in any case. But \eqref{2WP}, \eqref{4WP}, \eqref{3WP}, \eqref{LocConvWP} and \eqref{MinorLaplWP} clearly contradict \eqref{NoOtherConcentrationPt}, which concludes the proof of \eqref{WeakPointwiseEst}.
\end{proof}

\begin{proof}[Proof of Step \ref{StPointwiseGlobalControl}, Formula \eqref{WeakGradEst}]
Remember that we assume that $(\text{\bf Case 1})$ holds true. Assume then by contradiction that there exists $(y_\varepsilon)_\varepsilon$ such that $y_\varepsilon\in \Omega_\varepsilon$ and 
\begin{equation}\label{GW1}
\max_{y\in \Omega_\varepsilon} |y-x_\varepsilon| |\nabla u_\varepsilon(y)| u_\varepsilon(y)=|y_\varepsilon-x_\varepsilon| |\nabla u_\varepsilon(y_\varepsilon)| u_\varepsilon(y_\varepsilon):=C_\varepsilon\to +\infty
\end{equation}
as $\varepsilon\to 0$. Then, by \eqref{UInOmegaEps}, \eqref{2WP} holds true. Let $\nu_\varepsilon>0$ be given by
\begin{equation}\label{GW2}
\nu_\varepsilon =\min\left(|x_\varepsilon-y_\varepsilon|,d(y_\varepsilon,\partial\Omega)\right)\,.
\end{equation}
For all $R>1$ and all $\varepsilon$, we let $\Omega_{R,\varepsilon}$ and $\tilde{\Omega}_{R,\varepsilon}$ be given by the formulas above \eqref{6WP}. Let $w_\varepsilon$ be given by
\begin{equation}\label{DefWGW}
w_\varepsilon(z)=u_\varepsilon(y_\varepsilon+\nu_\varepsilon z).
\end{equation}
Since $\|u_\varepsilon\|_{H^1_0}^2\le 4\pi$, we get from Moser's inequality that $\int_\Omega \exp(u_\varepsilon^2) dy=O(1)$ and then that, for all given $p\ge 1$,
\begin{equation}\label{GW5}
\|\nu_\varepsilon^{2/p}w_\varepsilon\|_{L^p(\tilde{\Omega}_{R,\varepsilon})}=O(1)\,,
\end{equation}
for all $\varepsilon$. Now, for any given $R>1$ and all sequence $(z_\varepsilon)_\varepsilon$ such that $z_\varepsilon\in \Omega_{R,\varepsilon}\backslash\{x_\varepsilon\}$ (i.e. $\tilde{z}_\varepsilon:=(z_\varepsilon-y_\varepsilon)/\nu_\varepsilon\in \tilde{\Omega}_{R,\varepsilon}\backslash\{\tilde{x}_\varepsilon\}$), we get that
$$|\Delta w_\varepsilon(\tilde{z}_\varepsilon)|=\nu_\varepsilon^2 |\Delta u_\varepsilon(z_\varepsilon)|\lesssim
\begin{cases}
\frac{1}{u_\varepsilon(z_\varepsilon)|\tilde{z}_\varepsilon-\tilde{x}_\varepsilon|^2}\text{ if }z_\varepsilon\in \Omega_\varepsilon\,,\\
\lambda_\varepsilon \nu_\varepsilon^2 |\Psi'_{N_\varepsilon}(u_\varepsilon(z_\varepsilon))|=O\left(\lambda_\varepsilon \nu_\varepsilon^2 (1+u_\varepsilon(z_\varepsilon)^3) \right)\text{ if }z_\varepsilon\not \in \Omega_\varepsilon\,,
\end{cases}
 $$
 using \eqref{WeakPointwiseEst} for the first line, and \eqref{PsiPrime} for the second one. Then, using either \eqref{UInOmegaEps} or \eqref{LambdaTo0} with \eqref{GW5}, we get that
 \begin{equation}\label{LaplTo0GW}
 \|\Delta w_\varepsilon\|_{L^p(\tilde{\Omega}_{R,\varepsilon}\backslash B_{\tilde{x}_\varepsilon}(1/R))}\to 0
 \end{equation}
 as $\varepsilon\to 0$. Independently, since $\|u_\varepsilon\|_{H^1_0}=O(1)$, we easy get that
 \begin{equation}\label{H1LOcBdGW}
 \int_{\tilde{\Omega}_{R,\varepsilon}}|\nabla w_\varepsilon|^2 dz=O(1)\,.
 \end{equation}
 Set $\tilde{x}_\varepsilon=\frac{x_\varepsilon-y_\varepsilon}{\nu_\varepsilon}$. Observe that $|\tilde{x}_\varepsilon|\ge 1$. Now, we claim that up to a subsequence,
 \begin{equation}\label{EstimNuGW}
 \nu_\varepsilon\to 0\text{ and }\frac{d(y_\varepsilon,\partial\Omega)}{|x_\varepsilon-y_\varepsilon|}\to +\infty\,,
 \end{equation}
 as $\varepsilon\to 0$. In particular, by \eqref{GW2}, this implies that $\nu_\varepsilon=|x_\varepsilon-y_\varepsilon|$. Now we prove \eqref{EstimNuGW}. Indeed, if we assume by contradiction that \eqref{EstimNuGW} does not hold, for all $R\gg 1$ sufficiently large, we get that the $(w_\varepsilon/u_\varepsilon(y_\varepsilon))$'s converge locally out of $B_{\tilde{x}_\varepsilon}(1/2)$ to some $C^1$ function which is $1$ at $0$ and $0$ on the non-empty and smooth boundary of $\lim_{R\to+\infty}\lim_{\varepsilon\to 0} \tilde{\Omega}_{R,\varepsilon}$ (maybe after a harmless rotation). We use here the Harnack inequality and elliptic theory with \eqref{2WP}, \eqref{LaplTo0GW} (with $p>2$) and \eqref{H1LOcBdGW}, since $u_\varepsilon=0$ in $\partial\Omega$. This clearly contradicts \eqref{H1LOcBdGW} and \eqref{EstimNuGW} is proved. Up to a subsequence, we may now assume that
\begin{equation}\label{GW2Bis}
\tilde{x}_\varepsilon\to \tilde{x},~~|\tilde{x}|=1\,,
\end{equation}
as $\varepsilon\to 0$. By \eqref{2WP}, \eqref{LaplTo0GW}, \eqref{H1LOcBdGW}, and similar arguments including again Harnack's principle, we get that 
 \begin{equation}\label{ComplHarnGW}
 \frac{w_\varepsilon}{u_\varepsilon(y_\varepsilon)}\to 1\text{ in }C^1_{loc}(\mathbb{R}^2\backslash \{\tilde{x}\})\,,
 \end{equation}
 using also \eqref{EstimNuGW}. By \eqref{GW5} and \eqref{ComplHarnGW}, we get that for all $p\ge 1$
 \begin{equation}\label{UnifOutOfOEpsGW}
 \nu_\varepsilon^{2/p} u_\varepsilon(y_\varepsilon)=O(1)\,,
 \end{equation}
 as $\varepsilon\to 0$. Let now $\tilde{w}_\varepsilon$ be given by $\tilde{w}_\varepsilon=\frac{w_\varepsilon-w_\varepsilon(0)}{\nu_\varepsilon |\nabla u_\varepsilon(y_\varepsilon)|}$, so that $|\nabla \tilde{w}_\varepsilon(0)|=1$. For any given $R>1$ and all sequence $(z_\varepsilon)_\varepsilon$ such that $\tilde{z}_\varepsilon:=(z_\varepsilon-y_\varepsilon)/\nu_\varepsilon\in \tilde{\Omega}_{R,\varepsilon}\backslash B_{\tilde{x}}(1/R)$, we get that
\begin{equation*} 
 |\Delta\tilde{w}_\varepsilon(\tilde{z}_\varepsilon)|=\frac{u_\varepsilon(y_\varepsilon)}{C_\varepsilon}|\Delta w_\varepsilon(\tilde{z}_\varepsilon)|\lesssim 
 \begin{cases}
 \frac{1}{C_\varepsilon |\tilde{z}_\varepsilon-\tilde{x}_\varepsilon|^2}\text{ if }z_\varepsilon\in \Omega_\varepsilon\,,\\
\frac{\lambda_\varepsilon}{C_\varepsilon} \nu_\varepsilon^2 u_\varepsilon(y_\varepsilon)^4 \text{ if }z_\varepsilon\not \in \Omega_\varepsilon\,,
 \end{cases}
\end{equation*}
for all $\varepsilon$, using \eqref{WeakPointwiseEst}, \eqref{GW1} and \eqref{ComplHarnGW}. Then, by \eqref{LambdaTo0}, \eqref{GW1}, \eqref{EstimNuGW} and \eqref{UnifOutOfOEpsGW} (with $p\ge 4$), we get that
 \begin{equation}\label{LaplTo0GW2}
 \Delta \tilde{w}_\varepsilon\to 0 \text{ in }L^\infty_{loc}(\mathbb{R}^2\backslash \{\tilde{x}\})\,,
 \end{equation}
 as $\varepsilon\to 0$. By \eqref{GW1}, \eqref{GW2Bis} and \eqref{ComplHarnGW}, given $R>1$ and $\tilde{z}_\varepsilon\in \tilde{\Omega}_{R,\varepsilon}\backslash B_{\tilde{x}}(1/R)$, we get that
 \begin{equation}\label{GradEstGW}
 |\nabla \tilde{w}_\varepsilon(\tilde{z}_\varepsilon)|=\frac{|\nabla u_\varepsilon(z_\varepsilon)|}{|\nabla u_\varepsilon(y_\varepsilon)|}\le \frac{u_\varepsilon(y_\varepsilon)}{u_\varepsilon(z_\varepsilon)}\frac{1}{|\tilde{x}_\varepsilon-\tilde{z}_\varepsilon|}\le \frac{1+o(1)}{|\tilde{x}_\varepsilon-\tilde{z}_\varepsilon|}
 \end{equation}
for all $0<\varepsilon\ll 1$. Then, by \eqref{LaplTo0GW2}, \eqref{GradEstGW} and since $\tilde{w}_\varepsilon(0)=0$, there exists a harmonic function $\mathcal{H}$ in $\mathbb{R}^2\backslash \{\tilde{x}\}$ such that $\lim_{\varepsilon\to 0}\tilde{w}_\varepsilon=\mathcal{H}\text{ in }C^1_{loc}(\mathbb{R}^2\backslash \{\tilde{x}\})\,.$ Now, for all given $\beta>0$, integrating by parts, we get that
\begin{equation}\label{CClGW}
\begin{split}
&\int_{\partial B_{x_\varepsilon}(\beta \nu_\varepsilon)} u_\varepsilon \partial_\nu u_\varepsilon d\sigma\\
&=O\left(\int_{\Omega}|\nabla u_\varepsilon|^2 dy\right)+O\left(\int_{\Omega}u_\varepsilon (\Delta u_\varepsilon)^+ dy\right)=O(1)\,,\\
&=C_\varepsilon\left(\int_{\partial B_{\tilde{x}}(\beta)} \partial_\nu \mathcal{H} d\sigma+o(1) \right)\,,
\end{split}
\end{equation}
using \eqref{GW1} and \eqref{ComplHarnGW}, as $\varepsilon\to 0$. Since $C_\varepsilon\to +\infty$, we get from \eqref{CClGW} that $\int_{\partial B_{\tilde{x}}(\beta)} \partial_\nu \mathcal{H} d\sigma=0$. Then, also by \eqref{GradEstGW}, $\beta$ being arbitrary, $\mathcal{H}$ is bounded around $\tilde{x}$ and then the singularity at $\tilde{x}$ is removable. By the Liouville theorem, $\mathcal{H}$ is constant in $\mathbb{R}^2$, which is not possible since $|\nabla \tilde{w}_\varepsilon(0)|=|\nabla \mathcal{H}(0)|=1$. This concludes the proof of \eqref{WeakGradEst}.

\end{proof}

\begin{rem}\label{RemBdry}
Note that we do not assume that the continuous function $\Psi'_{N_\varepsilon}$ is positive and increasing in $[0,+\infty)$. Then, standard moving plane techniques \cites{AdimurthiDruet,GidasNiNiremb,Han,deFigLions} do not apply. We use in the proof below the variational characterization \eqref{ExtremUEps} of the $u_\varepsilon$'s to get that $\bar{x}\in K_\Omega$, $K_\Omega$ as in \eqref{MaxRobFun}, and that, in particular, $\bar{x}\not \in\partial \Omega$ in \eqref{XNotToBDry}.
\end{rem}

 Let $B_{\varepsilon}$ be the radial solution around $x_\varepsilon$ of
\begin{equation}\label{B1Eps}
\begin{cases}
&\Delta B_{\varepsilon}=\frac{\lambda_\varepsilon}{2} \Psi'_{N_\varepsilon}(B_\varepsilon)\,,\\
&B_\varepsilon(x_\varepsilon)=\gamma_\varepsilon\,,
\end{cases}
\end{equation}
where $\gamma_\varepsilon$ is still given by \eqref{XEps}. Let $\bar{u}_\varepsilon$ be given by 
\begin{equation}\label{UBar}
\bar{u}_\varepsilon(z)=\frac{1}{2\pi |x_\varepsilon-z|}\int_{\partial B_{x_\varepsilon}(|x_\varepsilon-z|)} u_\varepsilon~d\sigma\,,
\end{equation}
for all $z\neq x_\varepsilon$ and $\bar{u}_\varepsilon(x_\varepsilon)=u_\varepsilon(x_\varepsilon)=\gamma_\varepsilon$. Let $\varepsilon_0\in(\sqrt{1/e},1)$ be given. Let $\rho_{\varepsilon}>0$ be given by
\begin{equation}\label{RhoEpsDelta}
t_\varepsilon(\rho_{\varepsilon})=(1-\varepsilon_0) \gamma_\varepsilon^2\,.
\end{equation}
By \eqref{EqHBehavior}, \eqref{LambdaTo0}, \eqref{ScalRel} and \eqref{MinorLambdaEps}, we have that
\begin{equation}\label{RhoEpsDeltaEst}
\rho_{\varepsilon}^2=\exp(-(\varepsilon_0+o(1))\gamma_\varepsilon^2)\,.
\end{equation}
Let $r_\varepsilon$ be given by
\begin{equation}\label{DefREpsLem}
r_\varepsilon=\sup\left\{r\in(0,\rho_\varepsilon] \text{ s.t. }|\bar{u}_\varepsilon-B_\varepsilon|\le \frac{1}{\gamma_\varepsilon}\text{ in }B_{x_\varepsilon}(r)\right\}\,.
\end{equation} 
Observe that $r_\varepsilon\gg \mu_\varepsilon$ by Step \ref{StMinor} and Appendix \ref{Appendix}. Then, we state the following key result.
\begin{Step}\label{StDruThi}
We have that
\begin{equation}\label{PropAverage}
\bar{u}_\varepsilon(r_{\varepsilon})=B_{\varepsilon}(r_{\varepsilon})+o\left(\frac{1}{\gamma_\varepsilon}\right)\,,
\end{equation}
and then that $r_\varepsilon=\rho_\varepsilon$ for all $0<\varepsilon\ll 1$. Moreover, there exists $C>0$ such that
\begin{equation}\label{GradCrucEst}
|\nabla(B_\varepsilon-u_\varepsilon)|\le \frac{C}{\rho_{\varepsilon}\gamma_\varepsilon}\text{ in }B_{x_\varepsilon}(\rho_{\varepsilon})\,,
\end{equation}
for all $0<\varepsilon\ll 1$, where $(x_\varepsilon)_\varepsilon$ is as in \eqref{XEps}, $B_\varepsilon$ as in \eqref{B1Eps}, $\bar{u}_\varepsilon$ as in \eqref{UBar}, $\rho_{\varepsilon}$ as in \eqref{RhoEpsDelta} and $r_\varepsilon$ as in \eqref{DefREpsLem}. 
\end{Step}
Since $B_\varepsilon(x_\varepsilon)=u_\varepsilon(x_\varepsilon)=\gamma_\varepsilon$, \eqref{GradCrucEst} obviously implies that 
\begin{equation}\label{ContrHaut1}
|B_\varepsilon-u_\varepsilon|\le C\frac{|\cdot-x_\varepsilon|}{\rho_{\varepsilon}\gamma_\varepsilon}\text{ in }B_{x_\varepsilon}(\rho_{\varepsilon})\,,
\end{equation}
for all $0<\varepsilon\ll 1$. Then, combined with Appendix \ref{Appendix}, Step \ref{StDruThi} provides pointwise estimates of the $u_\varepsilon$'s in $B_{x_\varepsilon}(\rho_\varepsilon)$. 
\begin{proof}[Proof of Step \ref{StDruThi}] The proof of Lemma \ref{StDruThi} follows the lines of \cite[Section 3]{DruThiI}. We only recall here the argument in the more delicate $(\text{\bf Case 1})$. Let $v_\varepsilon$ be given by 
\begin{equation}\label{DefVProof}
u_\varepsilon=B_\varepsilon+v_\varepsilon\,.
\end{equation}
 By Appendix \ref{Appendix}, we have that $B_\varepsilon$ is well defined, radially decreasing in $B_{x_\varepsilon}(\rho_\varepsilon)$, and that 
\begin{equation}\label{DevBubbleRough}
B_\varepsilon=\gamma_\varepsilon-\frac{t_\varepsilon}{\gamma_\varepsilon}+o\left(\frac{t_\varepsilon}{\gamma_\varepsilon}\right)
\end{equation}
 uniformly in $B_{x_\varepsilon}(\rho_\varepsilon)$ as $\varepsilon\to 0$. Then, we get from \eqref{RhoEpsDeltaEst} and \eqref{DefREpsLem} that
\begin{equation}\label{MinorUSt}
\min_{B_{x_\varepsilon}(r_\varepsilon)} u_\varepsilon\ge \gamma_\varepsilon(\varepsilon_0+o(1))\,. 
\end{equation}
First, \eqref{MinorUSt} combined with \eqref{MinorGammaByNEps}, with \eqref{UInOmegaEps} and with our assumption $\varepsilon_0^2>{1/e}$ implies that $B_{x_\varepsilon}(r_\varepsilon)\subset \Omega_\varepsilon$. Then, we can use \eqref{WeakGradEst} to get also from \eqref{MinorUSt} that 
\begin{equation}\label{CondDruThi1}
\left\||x_\varepsilon-\cdot||\nabla u_\varepsilon| \right\|_{L^\infty(B_{x_\varepsilon}(r_\varepsilon))}=O\left(\frac{1}{\gamma_\varepsilon}\right)\,,
\end{equation}
which implies by \eqref{DefREpsLem} that
\begin{equation}\label{BdVSt}
\|v_\varepsilon\|_{L^\infty(B_{x_\varepsilon}(r_\varepsilon))}=O\left(\frac{1}{\gamma_\varepsilon} \right)\,,
\end{equation}
by the mean value property. Therefore, since 
\begin{equation}\label{BLeGamma}
B_\varepsilon\le \gamma_\varepsilon
\end{equation}
 in $B_{x_\varepsilon}(r_\varepsilon)$ and by \eqref{Propg}, \eqref{InftyBehavior}, Lemma \ref{HBehavior}, \eqref{ELEqNEpsBis}, \eqref{PsiPrime}, \eqref{B1Eps}, \eqref{DefREpsLem}, \eqref{DevBubbleRough} and \eqref{MinorUSt}, we get that there exists $C, C'>0$ such that
\begin{equation*}
|\Delta v_\varepsilon|\le C \lambda_\varepsilon \gamma_\varepsilon^2 \varphi_{N_\varepsilon-2}\left(\gamma_\varepsilon^2-2t_\varepsilon(1+o(1))+\frac{t_\varepsilon^2}{\gamma_\varepsilon^2}\right) |v_\varepsilon|\text{ in }B_{x_\varepsilon}(r_\varepsilon)\,,
\end{equation*}
an then that
\begin{equation}\label{CondDruThi2}
|\Delta v_\varepsilon|\le C' \frac{\exp\left(-2 t_\varepsilon(1+o(1))+\frac{t_\varepsilon^2}{\gamma_\varepsilon^2}\right)}{\mu_\varepsilon^2} |v_\varepsilon|\text{ in }B_{x_\varepsilon}(r_\varepsilon)
\end{equation}
by \eqref{MinorPhiNEps}, \eqref{ScalRel} and \eqref{TechnicalRk}. Observe that, for all $\Gamma,\delta>0$,
\begin{equation}\label{TechnicalRk4}
\varphi_{N}(\Gamma)=\delta \exp(\Gamma)\implies \forall T\in[0,\Gamma]\,,\quad \varphi_{N}(T)\le\delta \exp(T)\,,
\end{equation}
since $\varphi'_{N}\ge \varphi_{N}$ in $[0,+\infty]$. Starting now from \eqref{CondDruThi1}-\eqref{CondDruThi2}, we can compute and argue as in \cite[Section 3]{DruThiI} in order to get \eqref{PropAverage}-\eqref{GradCrucEst}. 
\end{proof}

\begin{proof}[Conclusion of the proof of Lemma \ref{LemBlowUpAnalysis}] In order to conclude the proof of Lemma \ref{LemBlowUpAnalysis}, by Steps \ref{Step1}-\ref{StDruThi}, it remains to prove \eqref{Level}, \eqref{EqualLambdaEps}-\eqref{XNotToBDry}, and \eqref{AFirstPointwEst}-\eqref{ExtEstimate} below. Let $\varepsilon_0'\in (\varepsilon_0,1)$ be fixed and let $\rho_\varepsilon'>0$ be given by
\begin{equation}\label{RhoPrimeEps}
t_\varepsilon(\rho_\varepsilon')=(1-\varepsilon_0')\gamma_\varepsilon^2\,,
\end{equation}
so that, by \eqref{EquivLog1surMu2},
\begin{equation}\label{RhoPrimeEpsEst}
(\rho_\varepsilon')^2=\exp(-\varepsilon'_0(1+o(1))\gamma_\varepsilon^2)\,.
\end{equation}\\
\noindent \textbullet~(1) In this first point, we aim to get pointwise estimates of the $u_\varepsilon$'s out of $B_{x_\varepsilon}(\rho_\varepsilon')$. Let $G$ be the Green's function in \eqref{GreenFun}. It is known that (see for instance \cite[Appendix B]{DruThiI}) there exists $C>0$ such that
\begin{equation}\label{GreenFunEstim}
\begin{split}
|\nabla_y G_x(y)|\le \frac{C}{|x-y|}\,,\text{ and }
0<G_x(y)\le \frac{1}{2\pi} \log\frac{C}{|x-y|}\,,
\end{split}
\end{equation}
for all $x,y\in\Omega$, $x\neq y$. By \eqref{GradCrucEst} and since $\|u_\varepsilon\|_{H^1_0}^2\le 4\pi$, it is possible to prove (see for instance the proof of \cite[Claim 4.6]{DruThiI}) that, given $p<1/\varepsilon_0'$,
\begin{equation}\label{LpBd}
\|\exp(u_\varepsilon^2)\|_{L^p(B_{x_\varepsilon}(\rho_\varepsilon'/2)^c)}=O(1)
\end{equation}
for all $\varepsilon$, where $B_{x_\varepsilon}(\rho_\varepsilon'/2)^c=\Omega\backslash B_{x_\varepsilon}(\rho_\varepsilon'/2)$. In the sequel, $p'>1$ is choosen such that
\begin{equation}\label{ChoosePPrime}
\frac{1}{p}+\frac{1}{p'}<1\,.
\end{equation}
 Let now $(z_\varepsilon)_\varepsilon$ be any sequence of points in $B_{x_\varepsilon}(\rho_\varepsilon')^c$. By the Green's representation formula and \eqref{ELEqNEpsBis}, we can write that
\begin{equation}\label{GreensRepInterLem}
u_\varepsilon(z_\varepsilon)=\frac{\lambda_\varepsilon}{2}\int_\Omega G_{z_\varepsilon}(y) \Psi'_{N_\varepsilon}(u_\varepsilon(y))~dy\,.
\end{equation}
By \eqref{GreenFunEstim}, we have that there exists $C>0$ such that
\begin{equation}\label{GreenFunDiff}
|G_{z_\varepsilon}(x_\varepsilon)-G_{z_\varepsilon}|\le C\frac{|x_\varepsilon-\cdot|}{\rho_\varepsilon'}
\end{equation}
in $B_{x_\varepsilon}(\rho_\varepsilon'/2)$, for all $\varepsilon$. By \eqref{EquivLog1surMu2} and \eqref{RhoEpsDeltaEst}, we have that 
\begin{equation}\label{TildeOmegaEps}
\frac{|\cdot-x_\varepsilon|}{\gamma_\varepsilon \rho_\varepsilon}=o\left(\frac{t_\varepsilon}{\gamma_\varepsilon^5} \right)\text{ in }\tilde{\Omega}_\varepsilon:=\{y\text{ s.t. } t_\varepsilon(y)\le \gamma_\varepsilon\}\,,
\end{equation}
as $\varepsilon\to 0$, and then, by \eqref{ContrHaut1}, \eqref{LaplBPert} holds true for $v_\varepsilon$ as in \eqref{DefVProof}. Independently, using \eqref{RoughEstimPsiPrime}, \eqref{ScalRel}, \eqref{ContrHaut1} and \eqref{ExpansionSi} with \eqref{ExpBubble}, we clearly get that there exists $C>0$ such that
\begin{equation}\label{LaplBpert2}
\lambda_\varepsilon|\Psi_{N_\varepsilon}'(u_\varepsilon)|\le C \frac{\exp\left(-2t_\varepsilon+\frac{t_\varepsilon^2}{\gamma_\varepsilon^2} \right)}{\mu_\varepsilon^2 \gamma_\varepsilon} \text{ in }B_{x_\varepsilon}(\rho_\varepsilon'/2)\backslash \tilde{\Omega}_\varepsilon\,,
\end{equation}
for all $\varepsilon$. Then, we get that
\begin{equation}\label{GreensRep2Proof}
\begin{split}
u_\varepsilon(z_\varepsilon) ~~&= G_{z_\varepsilon}(x_\varepsilon) \int_{B_{x_\varepsilon}(\rho_\varepsilon'/2)} \frac{\lambda_\varepsilon \Psi_{N_\varepsilon}'(u_\varepsilon)}{2} dy\\
&+O\left(\int_{B_{x_\varepsilon}(\rho_\varepsilon'/2)} \frac{\exp\left(-2t_\varepsilon+\frac{t_\varepsilon^2}{\gamma_\varepsilon^2} \right) |\cdot-x_\varepsilon|}{\mu_\varepsilon^2 \rho_\varepsilon'} dy \right) +O\left(\lambda_\varepsilon \|u_\varepsilon\|_{L^{p'}} \right)\\
&=G_{z_\varepsilon}(x_\varepsilon)\frac{4\pi}{\gamma_\varepsilon}\left(1+\frac{1}{\gamma_\varepsilon^2}+\frac{A(\gamma_\varepsilon)-2\xi_\varepsilon}{2}+o(\tilde{\zeta}_\varepsilon) \right)\\
&~~\quad \quad+o\left(\frac{1}{\gamma_\varepsilon}+\|u_\varepsilon\|_{L^{p'}}\right)\,,
\end{split}
\end{equation}
where $p'$ is fixed in \eqref{ChoosePPrime}, $\tilde{\zeta}_\varepsilon$ is given by \eqref{TildeZeta} and $x_\varepsilon$ by \eqref{DefXi}. Concerning the first estimate of \eqref{GreensRep2Proof}, \eqref{GreenFunDiff}, \eqref{LaplBpert2} and a rough version of \eqref{LaplBPert} are used to get the first two terms, while \eqref{RoughEstimPsiPrime}, \eqref{GreenFunEstim}, \eqref{LpBd} and Hölder's inequality are used to get the last one. Concerning the second estimate of \eqref{GreensRep2Proof}, \eqref{LambdaTo0}, \eqref{EquivLog1surMu2}, \eqref{RhoEpsDeltaEst}, \eqref{EquationApp}-\eqref{NoteSi}, $\varepsilon_0>1/2$, the dominated convergence theorem, \eqref{LaplBPert} and \eqref{LaplBpert2} are used. Using first that $u_\varepsilon\le \gamma_\varepsilon$ and \eqref{RhoEpsDeltaEst} in $B_{x_\varepsilon}(\rho_\varepsilon)$, and then \eqref{GreensRep2Proof} with \eqref{GreenFunEstim} in $\Omega\backslash B_{x_\varepsilon}(\rho_\varepsilon)$, we get that
\begin{equation}\label{LpBd2}
\|u_\varepsilon\|_{L^{p'}}=o\left(\frac{1}{\gamma_\varepsilon}+\|u_\varepsilon\|_{L^{p'}}\right)+O\left(\frac{1}{\gamma_\varepsilon} \right)\,.
\end{equation}
Summarizing, we get from \eqref{GreensRep2Proof} and \eqref{LpBd2} that
\begin{equation}\label{CClPartGreensFun}
u_\varepsilon(z_\varepsilon)=\frac{4\pi G_{z_\varepsilon}(x_\varepsilon)}{\gamma_\varepsilon}\left(1+\frac{1}{\gamma_\varepsilon^2}+\frac{A(\gamma_\varepsilon)-2\xi_\varepsilon}{2}+o(\tilde{\zeta}_\varepsilon) \right)\\
+o\left(\frac{1}{\gamma_\varepsilon}\right)\,.
\end{equation}

\noindent \textbullet~(2) In this second point, we prove that
\begin{equation}\label{ContrSupLambda}
\lambda_\varepsilon\le \frac{4+o(1)}{\gamma_\varepsilon^2\exp(1+M)}\,,
\end{equation}
as $\varepsilon\to 0$, for $M$ as in \eqref{MaxRobFun}. Observe that \eqref{CClPartGreensFun} implies that
\begin{equation}\label{ExtEst2}
u_\varepsilon=(1+o(1))\frac{4\pi G_{x_\varepsilon}+o(1)}{\gamma_\varepsilon}
\end{equation}
in $\Omega\backslash B_{x_\varepsilon}(\rho_\varepsilon)$.  Then, by \eqref{Propg}, \eqref{GreenFunEstim} and \eqref{ExtEst2}, our definition of $\rho_\varepsilon$ and the dominated convergence theorem, we get that
\begin{equation}\label{EstimExt1}
\lim_{\varepsilon\to 0}\int_{\Omega\backslash B_{x_\varepsilon}(\rho_\varepsilon)}\Psi_{N_\varepsilon}(u_\varepsilon) dy=|\Omega|(1+g(0))\,.
\end{equation}
 Independently, \eqref{ExpBubble} and \eqref{ContrHaut1} give that
\begin{equation}\label{IntEst2}
u_\varepsilon=\gamma_\varepsilon-\frac{(1+o(1))t_\varepsilon}{\gamma_\varepsilon}
\end{equation}
in $B_{x_\varepsilon}(\rho_\varepsilon)$, since $\mu_\varepsilon\ll \rho_\varepsilon$. Then, using \eqref{MinorPhiNEps}, \eqref{TechnicalRk}, $\varepsilon_0^2> 1/e$ and resuming the arguments to get \eqref{EquivPsiN}, we have that
\begin{equation}\label{IntEst3}
\Psi_{N_\varepsilon}(u_\varepsilon)=(1+o(1))\varphi_{N_\varepsilon-1}(u_\varepsilon^2)\text{ and }\Psi_{N_\varepsilon}'(u_\varepsilon)=2(1+o(1))~ u_\varepsilon \varphi_{N_\varepsilon-1}(u_\varepsilon^2)
\end{equation}
in $B_{x_\varepsilon}(\rho_\varepsilon)$. Independently, we get that
\begin{equation}\label{EstimInt1}
\int_{B_{x_\varepsilon}(\rho_\varepsilon)} \Psi_{N_\varepsilon}(u_\varepsilon) dy =\frac{4\pi(1+o(1))}{\gamma_\varepsilon^2 \lambda_\varepsilon}
\end{equation}
as $\varepsilon\to 0$, by \eqref{MinorPhiNEps}, \eqref{ScalRel}, \eqref{IntEst2}, \eqref{IntEst3}, with \eqref{AlgRelat} for $|y-x_\varepsilon|\lesssim \mu_\varepsilon$, or with \eqref{TechnicalRk4} and the dominated convergence theorem for $|y-x_\varepsilon|\gg \mu_\varepsilon$. Then, because of \eqref{ExtremUEps}, we get that \eqref{ContrSupLambda} holds true, by combining \eqref{EstimExt1}, \eqref{EstimInt1} with \eqref{TestFunPrelBis}. \\

\noindent \textbullet~(3) In this point, we conclude the proof of \eqref{EqualLambdaEps}, and prove \eqref{Level} and \eqref{XNotToBDry}. For $R>1$, let $\chi_{\varepsilon,R}$ be given in $\Omega_{\varepsilon,R}:=\Omega\backslash B_{x_\varepsilon}(R\mu_\varepsilon)$ by
\begin{equation}\label{DefChiEpsR}
\chi_{\varepsilon,R}=4\pi \Lambda_{\varepsilon,R} G_{x_\varepsilon}\,,
\end{equation}
for $\Lambda_{\varepsilon,R}>0$ to be chosen later such that
\begin{equation}\label{ConLambdaEpsR}
\chi_{\varepsilon,R}\le u_\varepsilon \text{ on }\partial B_{x_\varepsilon}(R\mu_\varepsilon)\,.
\end{equation}
Integrating by parts, we can write that
\begin{equation}\label{Minor11}
\begin{split}
\int_{\Omega_{\varepsilon,R}} |\nabla u_\varepsilon|^2 dy ~&= \int_{\Omega_{\varepsilon,R}} |\nabla \chi_{\varepsilon,R}|^2 dy-2\int_{\partial B_{x_\varepsilon}(R\mu_\varepsilon)} (\partial_\nu \chi_{\varepsilon,R}) (u_\varepsilon-\chi_{\varepsilon,R}) d\sigma\\
&+\int_{\Omega_{\varepsilon,R}} |\nabla(u_\varepsilon-\chi_{\varepsilon,R})|^2 dy \,,\\
&\ge\int_{\Omega_{\varepsilon,R}} |\nabla \chi_{\varepsilon,R}|^2 dy\,,
\end{split}
\end{equation} 
where $\nu$ is the unit outward normal to the boundary of $B_{x_\varepsilon}(R\mu_\varepsilon)$, using \eqref{ConLambdaEpsR}. Indeed, by \cite[Appendix B]{DruThiI} for instance, since $d(x_\varepsilon,\partial\Omega)\gg \mu_\varepsilon$ by Step \ref{StMinor}, we have that
\begin{equation}\label{Compl2}
\partial_\nu G_{x_\varepsilon}=-\frac{1}{2\pi R\mu_\varepsilon}+O\left(\frac{1}{d(x_\varepsilon,\partial\Omega)} \right)\text{ on }\partial B_{x_\varepsilon}(R \mu_\varepsilon) \,.
\end{equation}
Now, by \eqref{EqHBehavior}, \eqref{ScalRel}, \eqref{FirstResc}, \eqref{TechnicalRk}, \eqref{RhoEpsDeltaEst}, in order to have \eqref{ConLambdaEpsR}, we can choose $\Lambda_{\varepsilon,R}$ such that
\begin{equation}\label{ChoiceLambda}
\begin{split}
&\Lambda_{\varepsilon,R}=\frac{1}{\gamma_\varepsilon}\left(1-\frac{\log(1+R^2)+o(1)}{\gamma_\varepsilon^2}\right)\times\\ 
&\quad\left(1+\frac{\log\frac{\delta_\varepsilon \lambda_\varepsilon \gamma_\varepsilon^2}{4R^2}+\mathcal{H}_{x_\varepsilon}(x_\varepsilon)+O\left(\frac{\mu_\varepsilon^2}{\rho_\varepsilon^2} \right)}{\gamma_\varepsilon^2} \right)^{-1}\,,
\end{split}
\end{equation}
with $\delta_\varepsilon\in (0,1]$ as in \eqref{MinorPhiNEps}. In \eqref{ChoiceLambda}, the term $$\frac{\mu_\varepsilon^2}{\rho_\varepsilon^2}=o(1)$$ by \eqref{RhoEpsDelta}, arguing as in \eqref{Compl1}, since $d(x_\varepsilon,\partial \Omega)>\rho_\varepsilon$ by Step \ref{StDruThi}. Now, by \eqref{GreenFun}, \eqref{Compl2}, \eqref{EquivLog1surMu2} and \eqref{RhoEpsDeltaEst} again, we compute and get that
\begin{equation}\label{Minor12}
\begin{split}
&\int_{\Omega_{\varepsilon,R}} |\nabla \chi_{\varepsilon,R}|^2 dy \\
&~\ge -\int_{\partial B_{x_\varepsilon}(R\mu_\varepsilon)} (\partial_\nu \chi_{\varepsilon,R}) \chi_{\varepsilon,R}~d\sigma\,, \\
&~\ge 4\pi\left(1-\frac{2\log(1+R^2)+o(1)}{\gamma_\varepsilon^2}\right)\left(1+\frac{\log\frac{\delta_\varepsilon \lambda_\varepsilon \gamma_\varepsilon^2}{4R^2}+\mathcal{H}_{x_\varepsilon}(x_\varepsilon)+o(1)}{\gamma_\varepsilon^2}  \right)^{-1}
\end{split}
\end{equation}
using also \eqref{ChoiceLambda}. Independently, we compute 
\begin{equation}\label{Minor13}
\int_{B_{x_\varepsilon}(R\mu_\varepsilon)} |\nabla u_\varepsilon|^2 dy=\frac{4\pi}{\gamma_\varepsilon^2} \left(\log(1+R^2) -\frac{R^2}{1+R^2}+o(1)\right)\,,
\end{equation}
by \eqref{FirstResc}. Then, since $\|u_\varepsilon\|_{H^1_0}^2\le 4\pi$ by \eqref{ExtremUEps}, by \eqref{Minor11}, \eqref{Minor12} and \eqref{Minor13}, we get that $$\frac{\log{\delta_\varepsilon \lambda_\varepsilon } +\mathcal{H}_{x_\varepsilon}(x_\varepsilon)}{\gamma_\varepsilon^2}\ge o(1)\,.$$ Moreover, using also the definition \eqref{MaxRobFun} of $M$, \eqref{ContrSupLambda}, $\delta_\varepsilon\le 1$ and that $R>0$ may be arbitrarily large, we get together that 
\begin{equation}\label{DeltaTo1}
\delta_\varepsilon\to 1\,,
\end{equation}
 and that \eqref{EqualLambdaEps} and \eqref{XNotToBDry} hold true. Observe that \eqref{DeltaTo1} can be obtained directly ({\bf Case 2}). Then, \eqref{Level} follows from \eqref{EqualLambdaEps}, \eqref{EstimExt1} and \eqref{EstimInt1}.\\

\noindent \textbullet~(4) Now we prove \eqref{AGammaSmall1}. Since $\varepsilon_0'>\varepsilon_0$, we get from \eqref{RhoEpsDeltaEst}, \eqref{ContrHaut1}, \eqref{RhoPrimeEpsEst} and \eqref{ExpBubble} that
\begin{equation}\label{ExpUInt}
u_\varepsilon= \gamma_\varepsilon-\frac{t_\varepsilon}{\gamma_\varepsilon}-\frac{t_\varepsilon}{\gamma_\varepsilon^3}-(A(\gamma_\varepsilon)-2\xi_\varepsilon)\frac{t_\varepsilon}{2\gamma_\varepsilon}+o\left(\frac{t_\varepsilon \tilde{\zeta}_\varepsilon}{\gamma_\varepsilon} \right)
\end{equation}
uniformly in $\{y\in B_{x_\varepsilon}(\rho'_\varepsilon) \text{ s.t. }t_\varepsilon\ge \gamma_\varepsilon/4\}$, using also \eqref{ExpansionSi}. Then, noting that the averages of \eqref{CClPartGreensFun} and \eqref{ExpUInt} have to match on $\partial B_{x_\varepsilon}(\rho'_\varepsilon)$, we compute and get that
\begin{equation}\label{LambdaRelFirstLevel}
\lambda_\varepsilon=\frac{4}{\gamma_\varepsilon^2\exp\left(1+M+\frac{\gamma_\varepsilon^2 (A(\gamma_\varepsilon)-2\xi_\varepsilon)}{2}+o(\tilde{\zeta}_\varepsilon\gamma_\varepsilon^2)\right)}\,,
\end{equation}
by \eqref{XNotToBDry}, \eqref{DeltaTo1} and \eqref{ScalRel} with \eqref{EqHBehavior} and \eqref{TechnicalRk}, observing that $$1\lesssim \gamma_\varepsilon^2 G_{x_\varepsilon} \lesssim 1\,,\quad 1\lesssim \gamma_\varepsilon^2 t_\varepsilon \lesssim 1$$ on $\partial B_{x_\varepsilon}(\rho'_\varepsilon)$, by \eqref{RhoPrimeEps} and \eqref{RhoPrimeEpsEst} with \eqref{GreenFun} and \eqref{XNotToBDry}. By \eqref{EqualLambdaEps} and \eqref{LambdaRelFirstLevel}, \eqref{AGammaSmall1} is proved.\\

\noindent \textbullet~(5) Here, we conclude the proof of Lemma \ref{LemBlowUpAnalysis}. As an immediate consequence of \eqref{CClPartGreensFun}, we get that
\begin{equation}\label{AFirstPointwEst}
\left|u_\varepsilon(y)-\frac{4\pi G_{x_\varepsilon}(y)}{\gamma_\varepsilon} \right|=o\left(\frac{G_{x_\varepsilon}(y)}{\gamma_\varepsilon}\right)
\end{equation}
as $\varepsilon\to 0$, uniformly in $B_{x_\varepsilon}(\rho'_\varepsilon)^c$. Pushing now one step further the above computations with very similar arguments, we easily get that
\begin{equation}\label{ExpUInt2}
u_\varepsilon=\gamma_\varepsilon-\frac{t_\varepsilon}{\gamma_\varepsilon}+\frac{S_{0,\varepsilon}}{\gamma_\varepsilon^3}+\frac{S_{1,\varepsilon}}{\gamma_\varepsilon^5}+(A(\gamma_\varepsilon)-2\xi_\varepsilon)\frac{S_{2,\varepsilon}}{\gamma_\varepsilon}+o\left(t_\varepsilon\frac{\zeta_\varepsilon}{\gamma_\varepsilon} \right)\,,
\end{equation}
in $B_{x_\varepsilon}(\rho'_\varepsilon)$, where the $S_{i,\varepsilon}$'s are as in \eqref{DefSiEps}. At last, using in particular \eqref{EqualLambdaEps} with \eqref{ZeroBehavior} to improve the estimates in Point (1) of this proof, we get that
\begin{equation}\label{ExtEstimate}
\begin{split}
u_\varepsilon(y)= &~G_{x_\varepsilon}(y)\left(\frac{4\pi}{\gamma_\varepsilon}+\sum_{i=0}^1 \frac{A_i}{\gamma_\varepsilon^{3+2i}}+\frac{A_2 (A(\gamma_\varepsilon)-2\xi_\varepsilon)}{\gamma_\varepsilon}\right)\\
&+\frac{4 B(\gamma_\varepsilon)}{\gamma_\varepsilon^2 \exp(1+\mathcal{H}_{x_\varepsilon}(x_\varepsilon))} \int_{\Omega}G_{y}(x) F\left(4\pi G_{x_\varepsilon}(x) \right) dx\\
&+o\left(\frac{\zeta_\varepsilon}{\gamma_\varepsilon} G_{x_\varepsilon}(y)+\frac{|B(\gamma_\varepsilon)|}{\gamma_\varepsilon^2}\right)\,,
\end{split}
\end{equation}
in $B_{x_\varepsilon}(\rho'_\varepsilon)^c$, where $F$ and $B(\gamma_\varepsilon)$ are given in \eqref{ZeroBehavior}, where the $A_i$'s are as in \eqref{ExpansionSi}, and  where $\zeta_\varepsilon$ is given in \eqref{DefZetaApp}. 
\end{proof}
\noindent Lemma \ref{LemBlowUpAnalysis} is proved.
\end{proof}

\section{Proof of Proposition \ref{EnerExp}}\label{SectEnerExp}
\begin{proof}[Proof of Proposition \ref{EnerExp}]
Let $\Omega$ be a smooth bounded domain of $\mathbb{R}^2$. Let $g$ be such that \eqref{Propg} and \eqref{InftyBehavior}-\eqref{ZeroBehavior} hold true, for $H$ as in \eqref{DefH}, and let $A(\gamma)$, $B(\gamma)$ and $F$ be thus given. Let $(u_\varepsilon)_\varepsilon$ be a sequence of nonnegative functions such that $u_\varepsilon$ is a maximizer for $(I_{4\pi(1-\varepsilon)}^g(\Omega))$, for all $0<\varepsilon\ll 1$. Assume that \eqref{WeakConvToZero} holds true. Then, we apply Lemma \ref{LemBlowUpAnalysis}, $(\text{\bf Case 2})$: \eqref{SaturCond} holds true for $\alpha_\varepsilon=4\pi(1-\varepsilon)$; there exists a sequence $(\lambda_\varepsilon)_\varepsilon$ of real numbers such that $u_\varepsilon$ is $C^{1,\theta}$ and solves \eqref{ELEq} in $H^1_0$, using \eqref{ELEqNEps}; \eqref{LossCompactness} holds true by  \eqref{XEps}, \eqref{Level} is also true. Moreover, \eqref{EqualLambdaEps}-\eqref{XEps}, \eqref{AFirstPointwEst}-\eqref{ExtEstimate} and \eqref{LaplBPert} ($v_\varepsilon$ as in \eqref{DefVProof}) hold true still by Lemma \ref{LemBlowUpAnalysis}. In order to conclude the proof of Proposition \eqref{EnerExp}, it remains to prove \eqref{EqEnerExp}-\eqref{Eq2EnerExp}. At last, we let $\mu_\varepsilon$ be given by \eqref{ScalRel}, for $N_\varepsilon=1$, since we consider here ({\bf Case 2})\\

\noindent  In view of \eqref{ExtEstimate}, for $z\in\Omega$, we let now $U_{\varepsilon,z}$ be given by
\begin{equation}\label{Model}
\begin{split}
&U_{\varepsilon,z}(x)\\
&=\frac{1}{\gamma_\varepsilon}\left( \log \frac{1}{|x-z|^2+\tilde{\mu}_\varepsilon^2}+\tilde{\mathcal{H}}_{-1,\varepsilon,z}(x)\right)_{(\star)}\\
&\quad+\sum_{i=0}^1\frac{1}{\gamma_\varepsilon^{3+2i}}\left(S_i\left(\frac{x-z}{\tilde{\mu}_\varepsilon}\right)+\frac{A_i}{4\pi}\left(\log \frac{1}{\tilde{\mu}_\varepsilon^2}+\tilde{\mathcal{H}}_{i,\varepsilon,z}(x)\right)-B_i\right)_{(\star\star)}\\
&\quad+\frac{A(\gamma_\varepsilon)}{\gamma_\varepsilon}\left(S_2\left(\frac{x-z}{\tilde{\mu}_\varepsilon}\right)+\frac{A_2}{4\pi}\left(\log \frac{1}{\tilde{\mu}_\varepsilon^2}+\tilde{\mathcal{H}}_{2,\varepsilon,z}(x)\right)-B_2\right)_{(\star\star\star)}\\
&+\frac{4 B(\gamma_\varepsilon)}{\gamma_\varepsilon^2 \exp(1+\mathcal{H}_z(z))} \int_{\Omega}G_{x}(y) F\left(4\pi G_{z}(y) \right) dy
\end{split}
\end{equation}
where the $A_i,B_i$ are as in \eqref{ExpansionSi}, where $\mathcal{H}$ is as in \eqref{GreenFun}, where the $\tilde{\mathcal{H}}_{i,\varepsilon}$ are the unique harmonic functions in $\Omega$ such that the expressions involved in brackets $(\star),(\star\star),(\star\star\star)$ of \eqref{Model} were zero at the boundary, and where $\tilde{\mu}_\varepsilon$ is given by
\begin{equation}\label{TildeMu}
U_{\varepsilon,z}(z)=\gamma_\varepsilon\,.
\end{equation}

\noindent The following result concludes the proof of Proposition \ref{EnerExp}.

\begin{lem}\label{Step3ExistExtremal}
We have that
\begin{equation}\label{TestFunctionsComputations}
S=\int_\Omega G_{\bar{x}}(y) F(4\pi G_{\bar{x}}(y)) ~dy\,,\text{ if }\frac{\gamma_\varepsilon^{-3}B(\gamma_\varepsilon)}{\gamma_\varepsilon^{-4}+|A(\gamma_\varepsilon)|}\not\to 0\,,
\end{equation}
as $\varepsilon\to 0$, where $S$ is as in \eqref{MaxRobFun} and $\bar{x}$ as in \eqref{XNotToBDry}. Moreover, \eqref{EqEnerExp} holds true in any case.
\end{lem}

\begin{proof}[Proof of Lemma \ref{Step3ExistExtremal}] Let $K$ be a compact subset of $\Omega$ and $(z_\varepsilon)_\varepsilon$ be a given sequence of points of $K$. For simplicity, we let in the proof below $\check{\zeta}_\varepsilon$ be given by
\begin{equation}\label{CheckZeta}
\check{\zeta}_\varepsilon=\max\left(\frac{1}{\gamma_\varepsilon^4}, |A(\gamma_\varepsilon)|,\frac{|B(\gamma_\varepsilon)|}{\gamma_\varepsilon^3} \right)\,.
\end{equation}
Observe also that we get from \eqref{AGammaSmall}, \eqref{ExpUInt2} and \eqref{ExpansionSi} that
\begin{equation}\label{ContrHaut2}
\left|u_\varepsilon(y)-\left(\gamma_\varepsilon-\frac{t_\varepsilon(y)}{\gamma_\varepsilon} \right) \right|\le \frac{C}{\gamma_\varepsilon}\,,
\end{equation}
in $\{y\text{ s.t. }\gamma_\varepsilon/2\le t_\varepsilon(y)\le \gamma_\varepsilon(\gamma_\varepsilon-1/2)\}$, as $\varepsilon\to 0$.
\\
\noindent \textbullet~(1) We first derive the following more explicit expression of the $\tilde{\mu}_\varepsilon$ from \eqref{TildeMu}:
\begin{equation}\label{TildeMuMoreExplicit}
\begin{split}
&\frac{4}{\tilde{\mu}_\varepsilon^2\exp(\gamma_\varepsilon^2)\gamma_\varepsilon^2}\\
&=\frac{4}{\gamma_\varepsilon^2\exp(1+\mathcal{H}_{z_\varepsilon}(z_\varepsilon))}\left(1+O\left(\check{\zeta}_\varepsilon+\gamma_\varepsilon^4|A(\gamma_\varepsilon)|^2 \right) \right)\times\\
&\quad\left(1-\frac{\gamma_\varepsilon^2 A(\gamma_\varepsilon)}{2}- \frac{4 B(\gamma_\varepsilon)}{\gamma_\varepsilon \exp(1+\mathcal{H}_{z_\varepsilon}(z_\varepsilon))} \int_{\Omega}G_{z_\varepsilon}(y) F\left(4\pi G_{z_\varepsilon}(y) \right) dy\right)
\end{split}
\end{equation}
as $\varepsilon\to 0$. By the maximum principle and \eqref{ExpansionSi}, we get that there exists $C_K>0$ such that $|\mathcal{\tilde{H}}_{j,\varepsilon,z_\varepsilon}|\le C_K$ in $\Omega$, so that, by elliptic theory, the $\mathcal{\tilde{H}}_{j,\varepsilon,z_\varepsilon}$'s are also bounded in $C^1_{loc}(\Omega)$ for all $\varepsilon$ and $j$. We get from \eqref{TildeMu} that $\left|\log \frac{1}{\tilde{\mu}_\varepsilon^2}-\gamma_\varepsilon^2 \right|\le C_K'$, and then that
\begin{equation}\label{InterTildeMu}
\left|\mathcal{\tilde{H}}_{j,\varepsilon,z_\varepsilon}-\mathcal{H}_{z_\varepsilon}\right|\le C_K''{\gamma_\epsilon^4}\exp\left(-2\gamma_\varepsilon^2 \right)\text{ in }\Omega\,, 
\end{equation}
for all $0<\varepsilon\ll 1$ and $j\in\{-1,...,2\}$, by the maximum principle, \eqref{GreenFun} and \eqref{ExpansionSi}. Rewriting then \eqref{TildeMu} as
\begin{equation}\label{TildeMuSuite}
\begin{split}
\gamma_\varepsilon^2=&\log \frac{1}{\tilde{\mu}_\varepsilon^2}\left(1+\frac{A_0}{4\pi \gamma_\varepsilon^2}+\frac{A_1}{4\pi \gamma_\varepsilon^4}+\frac{A(\gamma_\varepsilon)A_2}{4\pi} \right)+\mathcal{H}_{z_\varepsilon}(z_\varepsilon)\left(1+\frac{A_0}{4\pi \gamma_\varepsilon^2} \right)\\
&-\frac{B_0}{\gamma_\varepsilon^2}+\frac{4 B(\gamma_\varepsilon)}{\gamma_\varepsilon \exp(1+\mathcal{H}_{z_\varepsilon}(z_\varepsilon))} \int_{\Omega}G_{z_\varepsilon}(y) F\left(4\pi G_{z_\varepsilon}(y) \right) dy\\
&+O\left(\gamma_\varepsilon^{-4}+|A(\gamma_\varepsilon)| \right)\,,
\end{split}
\end{equation}
we easily get \eqref{TildeMuMoreExplicit}, using \eqref{AGammaSmall} and \eqref{ExpansionSi} with $\frac{A_1}{4\pi}-\frac{A_0^2}{16\pi^2}-B_0=0$.\\

\noindent\textbullet~(2) We prove now that
\begin{equation}\label{TestFunPushed}
\int_\Omega |\nabla U_{\varepsilon,z_\varepsilon}|^2 dx=4\pi\left(1+I_{z_\varepsilon}(\gamma_\varepsilon) +o\left(\check{\zeta}_\varepsilon \right)\right)\,,
\end{equation}
as $\varepsilon\to 0$, where $I_{z_\varepsilon}(\gamma_\varepsilon)$ is given by
\begin{equation}\label{IZEpsilon}
I_{z_\varepsilon}(\gamma_\varepsilon)=\gamma_\varepsilon^{-4}+\frac{A(\gamma_\varepsilon)}{2}+ \frac{4 B(\gamma_\varepsilon)}{\gamma_\varepsilon^3 \exp(1+\mathcal{H}_{z_\varepsilon}(z_\varepsilon))}\int_{\Omega} G_{z_\varepsilon}(y)F(4\pi G_{z_\varepsilon}(y))~dy\,,
\end{equation}
 and where $U_{\varepsilon,z_\varepsilon}$ is given by \eqref{Model}-\eqref{TildeMu}. By \eqref{ZeroBehavior} and elliptic theory, 
\begin{equation} \label{BoundC1}
 \left(x\mapsto\int_{\Omega}G_{x}(y) F\left(4\pi G_{z_\varepsilon}(y) \right) dy\right)_\varepsilon\text{ is a bounded sequence in }C^1(\bar{\Omega})\,.
\end{equation} 
 By construction of the $\tilde{\mathcal{H}}_{j,\varepsilon,z_\varepsilon}$, we can write that
\begin{equation}\label{TestFunEnerComput}
\begin{split}
&\int_\Omega |\nabla U_{\varepsilon,z_\varepsilon}(y)|^2 dy=\int_{\Omega}  \Delta U_{\varepsilon,z_\varepsilon}(y) ~U_{\varepsilon,z_\varepsilon}(y)~dy\,,\\
&=\int_{\{y:\tilde{t}_\varepsilon(y)\le \gamma_\varepsilon\}} \left(\frac{\Delta(-\tilde{t}_\varepsilon)}{\gamma_\varepsilon}+\frac{\Delta \tilde{S}_{0,\varepsilon}}{\gamma_\varepsilon^3}+\frac{\Delta \tilde{S}_{1,\varepsilon}}{\gamma_\varepsilon^5}+\frac{A(\gamma_\varepsilon)\Delta \tilde{S}_{2,\varepsilon}}{\gamma_\varepsilon} \right)\times\\
& \quad\quad\quad\left(\gamma_\varepsilon-\frac{\tilde{t}_\varepsilon}{\gamma_\varepsilon}+\frac{\tilde{S}_{0,\varepsilon}}{\gamma_\varepsilon^3}+O\left( \left(\frac{|A(\gamma_\varepsilon)|}{\gamma_\varepsilon}+\frac{1}{\gamma_\varepsilon^5}\right)(1+\tilde{t}_\varepsilon)+\frac{|y-z_\varepsilon|}{\gamma_\varepsilon}\right)\right) dy\\
&\quad\quad+o(\gamma_\varepsilon^{-4})\\
&+\int_{\{y:\tilde{t}_\varepsilon(y)\ge \gamma_\varepsilon(\gamma_\varepsilon-1)\}} \left(O\left(\tilde{\mu}_\varepsilon^2\gamma_\varepsilon^4 \right)+\frac{4 B(\gamma_\varepsilon)}{\gamma_\varepsilon^2\exp(1+\mathcal{H}_{z_\varepsilon}(z_\varepsilon))}F(4\pi G_{z_\varepsilon}(y))\right)\times\\
&\quad\quad\quad\quad\quad\quad\quad\quad\quad\quad\quad\quad \left(\frac{4\pi G_{z_\varepsilon}(y)}{\gamma_\varepsilon}+O\left(\frac{G_{z_\varepsilon}(y)}{\gamma_\varepsilon^3}+\frac{|B(\gamma_\varepsilon)|}{\gamma_\varepsilon^2} \right) \right) dy\,,
\end{split}
\end{equation}
where $\tilde{t}_\varepsilon(y)=\log\left(1+|y-z_\varepsilon|^2/\tilde{\mu}_\varepsilon^2 \right)$ and $\tilde{S}_{i,\varepsilon}=S_i(|y-z_\varepsilon|/\tilde{\mu}_\varepsilon)$. We use also here \eqref{GreenFun} with \eqref{AGammaSmall}, and the estimates of Point (1) of this proof, including \eqref{TildeMuMoreExplicit}-\eqref{InterTildeMu}. The integral on $\{\tilde{t}_\varepsilon\in (\gamma_\varepsilon,\gamma_\varepsilon(\gamma_\varepsilon-1))\}$ gives a $o(\gamma_\varepsilon^{-4})$ term. Estimate \eqref{TestFunPushed} follows from \eqref{TestFunEnerComput}, Appendix \ref{Appendix} and some computations that we do not develop here again (see also \cite{MartMan}, §5).\\
\noindent \textbullet~(3) We prove now that
\begin{equation}\label{UEpsPushed}
\int_\Omega |\nabla u_\varepsilon|^2 dx=4\pi\left(1+I_{x_\varepsilon}(\gamma_\varepsilon)+o\left(\check{\zeta}_\varepsilon\right) \right)\,,
\end{equation}
as $\varepsilon\to 0$, where $I_{x_\varepsilon}(\gamma_\varepsilon)$ is given by \eqref{IZEpsilon}, for $(x_\varepsilon)_\varepsilon$ as in \eqref{XEps}. Now, we can push one step further the argument involving \eqref{LambdaRelFirstLevel}, writing now that both formulas \eqref{ExpUInt2} and \eqref{ExtEstimate} must also  coincide on $\partial B_{x_\varepsilon}(\rho'_\varepsilon)$, where $\rho'_\varepsilon>0$ is as in \eqref{RhoPrimeEps}. We compute and then get for $\mu_\varepsilon$ in \eqref{ScalRel} the analogue of \eqref{TildeMuMoreExplicit} for $\tilde{\mu}_\varepsilon$
\begin{equation}\label{TildeMuMoreExplicit2}
\begin{split}
&\lambda_\varepsilon H(\gamma_\varepsilon)=\frac{4}{{\mu}_\varepsilon^2\exp(\gamma_\varepsilon^2)\gamma_\varepsilon^2}\left(1+o\left(\frac{1}{\gamma_\varepsilon^4}\right)\right)\\
&=\frac{4}{\gamma_\varepsilon^2\exp(1+\mathcal{H}_{x_\varepsilon}(x_\varepsilon))}\left(1+o\left(\gamma_\varepsilon^2 \check{\zeta}_\varepsilon \right) \right)\times\\
&\quad\left(1-\frac{\gamma_\varepsilon^2 A(\gamma_\varepsilon)}{2}- \frac{4 B(\gamma_\varepsilon)}{\gamma_\varepsilon \exp(1+\mathcal{H}_{x_\varepsilon}(x_\varepsilon))} \int_{\Omega}G_{x_\varepsilon}(y) F\left(4\pi G_{x_\varepsilon}(y) \right) dy\right)\,,
\end{split}
\end{equation}
using \eqref{GreenFun}, \eqref{AGammaSmall}, \eqref{ExpansionSi}-\eqref{ExpBubble}. Independently, integrating by parts, resuming some computations in Appendix \ref{Appendix} and using \eqref{ELEq}, \eqref{XNotToBDry}, \eqref{EquivLog1surMu2}, Point (1), and \eqref{AFirstPointwEst}-\eqref{ExtEstimate} (see also \eqref{DefVProof} and \eqref{LaplBPert}), we get that
\begin{equation}\label{IPPUEps}
\begin{split}
\int_\Omega |\nabla u_\varepsilon|^2 dx~&=\int_\Omega u_\varepsilon \left(\lambda_\varepsilon H(u_\varepsilon)  u_\varepsilon \exp(u_\varepsilon^2)\right) dx\,,\\
&=\int_\Omega U_{\varepsilon,x_\varepsilon} \Delta U_{\varepsilon,x_\varepsilon} dx+o\left(\check{\zeta}_\varepsilon\right)\,.
\end{split}
\end{equation}
 In order to get the second equality and to apply the dominated convergence theorem, it may be useful to split $\Omega$ according 
\begin{equation*}
\begin{split}
\Omega=~ &\left\{y\text{ s.t. }t_\varepsilon(y)\le \gamma_\varepsilon\right\} \cup \left\{y\text{ s.t. } t_\varepsilon(y)>\gamma_\varepsilon\text{ and }\log\frac{1}{|x_\varepsilon-y|^2}\ge \frac{1-\delta'_0}{2} \gamma_\varepsilon^2 \right\}\\
& \cup \left\{y\text{ s.t. }\log\frac{1}{|x_\varepsilon-y|^2}< \frac{1-\delta'_0}{2} \gamma_\varepsilon^2 \right\} \,,
\end{split}
\end{equation*}
where $\delta'_0$ is as in \eqref{ZeroBehavior}, and to use the first line of \eqref{TildeMuMoreExplicit2} with \eqref{InftyBehavior} (resp. with \eqref{RoughEstimPsiPrime}) in the first region (resp. in the second region), or \eqref{ZeroBehavior}-\eqref{LimitB} in the last region. Observe that the argument here is to show that $U_{\varepsilon,x_\varepsilon}$ (resp. $\Delta U_{\varepsilon,x_\varepsilon}$) is in some sense the main part of the expansion of $u_\varepsilon$ (resp. $\Delta u_\varepsilon$). Thus we get \eqref{UEpsPushed} from \eqref{TestFunPushed} and \eqref{IPPUEps}.\\
\noindent \textbullet~(4) We prove now that, for any fixed sequence $(\eta_\varepsilon)_\varepsilon$ of real numbers such that $\eta_\varepsilon=o(\gamma_\varepsilon^{-2})$, we have that
\begin{equation}\label{ExpFunctional}
\begin{split}
&\int_\Omega \left(1+g(V_{\varepsilon,z_\varepsilon}) \right)\exp\left(V_{\varepsilon,z_\varepsilon}^2 \right) dy\\
&\quad=|\Omega|(1+g(0))+\pi\exp(1+\mathcal{H}_{z_\varepsilon}(z_\varepsilon))(1-\eta_\varepsilon\gamma_\varepsilon^2)\times\\
&\quad\quad\quad H(\gamma_\varepsilon)\left(1+\gamma_\varepsilon^2 I_{z_\varepsilon}(\gamma_\varepsilon)+\frac{1}{\gamma_\varepsilon^2}+o\left(\gamma_\varepsilon^2\left(\check{\zeta}_\varepsilon+|\eta_\varepsilon|\right)\right) \right)\times\\
&\quad\quad\quad\left(1+\frac{8 B(\gamma_\varepsilon)}{\gamma_\varepsilon\left(\kappa+1\right)\exp(1+\mathcal{H}_{z_\varepsilon}(z_\varepsilon))} \int_\Omega G_{z_\varepsilon}(y) F\left(4\pi G_{z_\varepsilon}(y) \right)~ dy \right)\,, 
\end{split}
\end{equation}
where $\kappa$ is as in \eqref{ZeroBehavior} and where $V_{\varepsilon,z_\varepsilon}\ge 0$ is given by
\begin{equation}\label{DefVEps}
V_{\varepsilon,z_\varepsilon}^2=(1-\eta_\varepsilon)U_{\varepsilon,z_\varepsilon}^2\,,
\end{equation}
where $U_{\varepsilon,z_\varepsilon}$ is given in \eqref{Model}. Computations in the spirit of the proof of \eqref{IPPUEps} give that
\begin{equation}\label{ParticularCaseExpFun}
\begin{split}
\int_\Omega \left(1+g(U_{\varepsilon,x_\varepsilon}) \right)\exp\left(U_{\varepsilon,x_\varepsilon}^2 \right) dy~&~= \int_\Omega \left(1+g(u_\varepsilon) \right) \exp\left(u_\varepsilon^2 \right) dy+o\left(\gamma_\varepsilon^2 \check{\zeta}_\varepsilon\right)\,,
\end{split}
\end{equation}
not only by combining \eqref{Propg}, \eqref{InftyBehavior}-\eqref{ZeroBehavior}, Lemma \ref{HBehavior}, \eqref{XNotToBDry}, \eqref{AFirstPointwEst}-\eqref{ExtEstimate} and Appendix \ref{Appendix}, and by splitting $\Omega$ as in \eqref{TestFunEnerComput}, but also by using \eqref{TildeMuMoreExplicit} and \eqref{TildeMuMoreExplicit2}. In particular, once \eqref{ExpFunctional} is proved, choosing $\eta_\varepsilon=0$ and $z_\varepsilon=x_\varepsilon$, we get from \eqref{ParticularCaseExpFun} that
\begin{equation}\label{ParticularCaseExpFun2}
\begin{split}
&\int_\Omega \left(1+g(u_\varepsilon) \right)\exp\left(u_\varepsilon^2 \right) dy=|\Omega|(1+g(0))+\pi\exp(1+\mathcal{H}_{x_\varepsilon}(x_\varepsilon))H(\gamma_\varepsilon)\times\\
&\quad\quad\quad\quad \left(1+\gamma_\varepsilon^2 I_{x_\varepsilon}(\gamma_\varepsilon)+\frac{1}{\gamma_\varepsilon^2}+o\left(\gamma_\varepsilon^2 \check{\zeta}_\varepsilon \right) \right)\times\\
&\quad\quad\quad \quad\left(1+\frac{8 B(\gamma_\varepsilon)}{\gamma_\varepsilon\left(\kappa+1\right)\exp(1+\mathcal{H}_{x_\varepsilon}(x_\varepsilon))} \int_\Omega G_{x_\varepsilon}(y) F\left(4\pi G_{x_\varepsilon}(y) \right)~ dy \right)\,.
\end{split}
\end{equation}
 It remains to prove \eqref{ExpFunctional}. We compute and get that
\begin{equation}\label{USquare}
U_{\varepsilon,z_\varepsilon}(y)^2=\gamma_\varepsilon^2-2\tilde{t}_\varepsilon+\frac{\tilde{t}_\varepsilon^2}{\gamma_\varepsilon^2}+\frac{2\tilde{S}_{0,\varepsilon}}{\gamma_\varepsilon^2}+O\left((|A(\gamma_\varepsilon)|+\gamma_\varepsilon^{-4})(1+\tilde{t}_\varepsilon(y)^2)+|y-z_\varepsilon| \right)
\end{equation}
for all $y$ such that $\tilde{t}_\varepsilon(y)\le \gamma_\varepsilon$, using \eqref{LimitB}, \eqref{Model}-\eqref{TildeMu}, \eqref{TildeMuMoreExplicit}, \eqref{BoundC1} and \eqref{ExpansionSi}. Then we get
\begin{equation*}
\begin{split}
&\int_{\{\tilde{t}_\varepsilon\le \gamma_\varepsilon\}} (1+g(V_{\varepsilon,z_\varepsilon}))\exp(V_{\varepsilon,z_\varepsilon}^2) dy\\
&=\int_{\{\tilde{t}_\varepsilon\le \gamma_\varepsilon\}} H(\gamma_\varepsilon)(1+O(|A(\gamma_\varepsilon)|\exp(\delta_0\tilde{t}_\varepsilon)))\exp(\gamma_\varepsilon^2) \exp(-2\tilde{t}_\varepsilon) \exp(-\eta_\varepsilon \gamma_\varepsilon^2)\times\\
&\quad\exp\left(\frac{\tilde{t}_\varepsilon^2+2\tilde{S}_{0,\varepsilon}}{\gamma_\varepsilon^2} \right)\exp\left(O\left(\left(|\eta_\varepsilon|+|A(\gamma_\varepsilon)|+\gamma_\varepsilon^{-4}\right)(1+\tilde{t}^{2}_\varepsilon) \right)+|y-z_\varepsilon|\right)~dy
\end{split}
\end{equation*}
using \eqref{AsymptG} and \eqref{DefVEps} with \eqref{USquare}. Then combining $\eta_\varepsilon=o(\gamma_\varepsilon^{-2})$, \eqref{AGammaSmall}, \eqref{TildeMuMoreExplicit}, computing explicitly $\int_{\mathbb{R}^2}\exp(-2T_0)S_0 dy=0$ and $\int_{\mathbb{R}^2}\exp(-2T_0) T_0^2 dy=2\pi$, we get that
\begin{equation}\label{IntegraleInteriorPart}
\begin{split}
&\int_{\{\tilde{t}_\varepsilon\le \gamma_\varepsilon\}} (1+g(V_{\varepsilon,z_\varepsilon}))\exp(V_{\varepsilon,z_\varepsilon}^2) dy\\
&=\frac{(1-\eta_\varepsilon \gamma_\varepsilon^2)H(\gamma_\varepsilon)\exp(\mathcal{H}_{z_\varepsilon}(z_\varepsilon)+1)}{4}\left(1+o\left(\gamma_\varepsilon^2\left(|A(\gamma_\varepsilon)|+|\eta_\varepsilon|\right)+\gamma_\varepsilon^{-2} \right) \right)\\
&\times\Bigg(1+\frac{\gamma_\varepsilon^2 A(\gamma_\varepsilon)}{2}+\frac{4B(\gamma_\varepsilon)}{\gamma_\varepsilon\exp(\mathcal{H}_{z_\varepsilon}(z_\varepsilon)+1)}\times\\
&\quad\quad\quad\quad\quad\quad\int_\Omega G_{z_\varepsilon}(x)F(4\pi G_{z_\varepsilon}(x))dx+o\left(\frac{B(\gamma_\varepsilon)}{\gamma_\varepsilon} \right) \Bigg)\times 4\pi\left(1+\frac{2}{\gamma_\varepsilon^2} \right)\,.
\end{split}
\end{equation}
Independently, we get from \eqref{ZeroBehavior}, \eqref{ZeroBehaviorG} (part $c)$ in $\{y\,,4\pi G_{z_\varepsilon}(y)\le \gamma_\varepsilon/2 \}$, or parts $a)$ and $b)$ otherwise), \eqref{Model}, \eqref{TildeMuMoreExplicit} and the dominated convergence theorem that
\begin{equation}\label{IntegraleExteriorPart}
\begin{split}
&\int_{\{\tilde{t}_\varepsilon\ge \gamma_\varepsilon\}} (1+g(V_{\varepsilon,z_\varepsilon}))\exp(V_{\varepsilon,z_\varepsilon}^2) dy\\
&\quad \quad=|\Omega|\left(1+g(0) \right)+\frac{8\pi B(\gamma_\varepsilon)}{\gamma_\varepsilon\left(\kappa+1\right)} \int_\Omega G_{z_\varepsilon}(y) F\left(4\pi G_{z_\varepsilon}(y) \right)~ dy\\
&\quad \quad \quad\quad\quad\quad\quad\quad\quad\quad\quad\quad\quad\quad\quad +o\left(\frac{|B(\gamma_\varepsilon)|}{\gamma_\varepsilon}+\frac{1}{\gamma_\varepsilon^2} \right)\,.
\end{split}
\end{equation}
Combining \eqref{IntegraleInteriorPart} and \eqref{IntegraleExteriorPart}, we conclude that \eqref{ExpFunctional} holds true, using \eqref{EqHBehavior} and \eqref{TildeMuMoreExplicit}.
 \\
\noindent \textbullet~(5) We are now in position to conclude the proof of Lemma \ref{Step3ExistExtremal}. Let $\bar{x}_0$ be a point in the compact $K_\Omega\subset\subset \Omega$ where $S$ is attained in the third equation of \eqref{MaxRobFun}. Let $\eta_\varepsilon$ be given by
\begin{equation}\label{DefEtaEps}
(1-\eta_\varepsilon)=\frac{4\pi(1-\varepsilon)}{\|U_{\varepsilon,\bar{x}_0}\|_{H^1_0}^2}\,.
\end{equation}
 First, we can check that 
\begin{equation}\label{ExpEta2}
\eta_\varepsilon=I_{\bar{x}_0}(\gamma_\varepsilon)-I_{x_\varepsilon}(\gamma_\varepsilon)+o( \check{\zeta}_\varepsilon)\,,
\end{equation}
so that the condition $\eta_\varepsilon=o(\gamma_\varepsilon^{-2})$ above \eqref{ExpFunctional} is satisfied, using \eqref{LimitB}, \eqref{ExtremUEps}, \eqref{AGammaSmall}, \eqref{TestFunPushed} and \eqref{UEpsPushed}. Besides, we have that $\|V_{\varepsilon,\bar{x}_0}\|_{H_0^1}^2=4\pi(1-\varepsilon)$, by our choice \eqref{DefEtaEps} of $\eta_\varepsilon$, and then, by \eqref{ExtremUEps}, that
\begin{equation*}
\int_\Omega (1+g(u_\varepsilon)) \exp(u_\varepsilon^2)~ dy\ge \int_\Omega (1+g(V_{\varepsilon,\bar{x}_0}))\exp(V_{\varepsilon,\bar{x}_0}^2)~dy\,;
\end{equation*}
this implies, in view of \eqref{ExpFunctional}, \eqref{ParticularCaseExpFun2}, \eqref{ExpEta2} and of our choice of $\bar{x}_0$, that \eqref{TestFunctionsComputations} is true and then, by \eqref{UEpsPushed} again,  that \eqref{EqEnerExp}-\eqref{Eq2EnerExp} are true as well. This concludes the proof of Lemma \ref{Step3ExistExtremal}.
\end{proof}
\noindent Proposition \ref{EnerExp} is proved.
\end{proof}

\begin{proof}[Proof of Proposition \ref{NonExExtProp}]
Let $\Omega$ be a smooth bounded domain of $\mathbb{R}^2$. Let $g$ be such that \eqref{Propg} and \eqref{InftyBehavior}-\eqref{ZeroBehavior} hold true, for $H$ as in \eqref{DefH}, and let $A$, $B$ and $F$ be thus given. Assume that $\Lambda_g(\Omega)<\pi \exp(1+M)$, where $M$ is as in \eqref{MaxRobFun} and $\Lambda_g(\Omega)$ as in \eqref{NonlinEigen}. Assume that there exists a sequence of positive integers $(N_\varepsilon)_\varepsilon$ such that \eqref{NToInfty} holds true and such that $(I_{4\pi}^{g_{N_\varepsilon}}(\Omega))$ admits a nonnegative extremal $u_\varepsilon$ for all $\varepsilon>0$, where $g_{N_\varepsilon}$ is as in \eqref{gN}. Then, by Lemma \ref{LemBlowUpAnalysis} in $(\text{\bf Case 1})$, we have \eqref{WeakConvToZero} and that \eqref{SaturCond} holds true for $\alpha_\varepsilon=4\pi$, for all $0<\varepsilon\ll 1$. Moreover, we have $u_\varepsilon\in C^{1,\theta}(\bar{\Omega})$ ($0<\theta<1$) and \eqref{LossCompactness} by \eqref{XEps}. In order to conclude the proof of Proposition \ref{NonExExtProp}, it remains to prove \eqref{EqEnerExp2}. Still by Lemma \ref{LemBlowUpAnalysis} in $(\text{\bf Case 1})$, \eqref{AFirstPointwEst}-\eqref{ExtEstimate} and \eqref{LaplBPert} ($v_\varepsilon$ as in \eqref{DefVProof}) hold true. Concerning \eqref{AFirstPointwEst}-\eqref{ExtEstimate} and \eqref{LaplBPert}, observe that, contrary to $(\text{\bf Case 2})$, the term $\xi_\varepsilon$ cannot be neglected in $(\text{\bf Case 1})$ we are facing here. {Indeed, using also now \eqref{MinorPhiNEps}, \eqref{ScalRel}, \eqref{TechnicalRk4} and \eqref{LaplBPert}, we can resume computations of \eqref{TestFunEnerComput}, \eqref{IPPUEps} and Appendix \ref{Appendix}} to get that
\begin{equation*}
\|u_\varepsilon\|_{H^1_0}^2=4\pi\left(1+\check{I}(\gamma_\varepsilon)+o\left(\gamma_\varepsilon^{-4}+|A(\gamma_\varepsilon)|+\gamma_\varepsilon^{-3}|B(\gamma_\varepsilon)|+\xi_\varepsilon \right)\right)
\end{equation*}
as $\varepsilon\to 0$, where 
\begin{equation*}
\check{I}(\gamma_\varepsilon):=\gamma_\varepsilon^{-4}+(A(\gamma_\varepsilon)-2\xi_\varepsilon)/2+4\gamma_\varepsilon^{-3}\exp(-1-M) B(\gamma_\varepsilon) S\,,
\end{equation*}
so that \eqref{EqEnerExp2} holds true, which concludes. 
\end{proof}

\appendix
\section{Radial analysis}\label{Appendix} Let $(x_\varepsilon)_\varepsilon$ be a sequence of points in $\mathbb{R}^2$ and $(\gamma_\varepsilon)_\varepsilon$ be a sequence of positive real numbers such that \eqref{GammaToInfty} holds true. Let $g$ be such that \eqref{Propg} and \eqref{InftyBehavior} holds true for $H$ as in \eqref{DefH}, and let $A$ be thus given. Let $(N_\varepsilon)_\varepsilon$ be a sequence of integers. We assume that we are in one of the following two cases:
\begin{equation}\tag{Case 1}
N_\varepsilon\to +\infty \text{ as }\varepsilon\to 0, \text { and \eqref{MinorPhiNEps}-\eqref{MinorGammaByNEps} hold true},
\end{equation}
\begin{equation}\tag{Case 2}
N_\varepsilon=1\text{ for all }\varepsilon\,.
\end{equation}
 Let $B_{\varepsilon}$ be the radial solution around $x_\varepsilon$ in $\mathbb{R}^2$ of \eqref{B1Eps}, for $\Psi_N$ as in \eqref{DefPsiN}, where $(\lambda_\varepsilon)_\varepsilon$ is any given sequence of positive real numbers. Let $T_0$ be given in $\mathbb{R}^2$ by
\begin{equation}\label{DefT0}
T_0(x)=\log\left(1+|x|^2 \right)\,.
\end{equation}
Let $S_i$, $i=0,1,2$, be the radially symmetric solutions around $0$ in $\mathbb{R}^2$ of
\begin{equation}\label{EquationApp}
\begin{split}
&\Delta S_0-8\exp(-2 T_0)S_0=4\exp(-2T_0)\left(T_0^2-T_0 \right)\,,\\
&\Delta S_1-8\exp(-2 T_0)S_1=4\exp(-2T_0) \left(S_0+2S_0^2-4T_0 S_0+2S_0T_0^2-T_0^3+\frac{T_0^4}{2} \right)\,,\\
&\Delta S_2-8\exp(-2 T_0)S_2=4\exp(-2T_0) T_0\,,
\end{split}
\end{equation}
such that $S_i(0)=0$. In the sequel, we will use the asymptotic expansions of the $S_i$'s given by
\begin{equation}\label{ExpansionSi}
\begin{split}
&S_0(r)=\frac{A_0}{4\pi}\log\frac{1}{r^2}+B_0+O\left(\log(r)^2r^{-2} \right)\text{ where }A_0=4\pi, \quad B_0=\frac{\pi^2}{6}+2\,,\\
&S_1(r)=\frac{A_1}{4\pi} \log\frac{1}{r^2}+B_1+O\left(\log(r)^4r^{-2}\right)\text{ where }A_1=4\pi\left( 3+\frac{\pi^2}{6}\right), \quad B_1\in\mathbb{R}\,,\\
&S_2(r)=\frac{A_2}{4\pi} \log\frac{1}{r^2}+B_2+O\left(\log(r)r^{-2}\right)\text{ where }A_2=2\pi, \quad B_2\in\mathbb{R}\,,
\end{split}
\end{equation}
as $r=|x|\to +\infty$. Note that in particular
\begin{equation}\label{NoteSi}
A_i=\int_{\mathbb{R}^2} \Delta S_i dx\,.
\end{equation}
 The explicit formula for $S_0$
$$S_0(r)=-T_0(r)+\frac{2r^2}{1+r^2}-\frac{1}{2}T_0(r)^2+\frac{1-r^2}{1+r^2}\int_1^{1+r^2}\frac{\log t}{1-t} dt\,, $$
and the expansions in \eqref{ExpansionSi} are derived in \cites{MalchMartJEMS,MartMan}. Let $\varepsilon_0\in (\sqrt{1/e},1)$ be given. Let $\mu_\varepsilon$ be given by \eqref{ScalRel} and $t_\varepsilon$ by \eqref{TEps}. Let $\rho_\varepsilon>0$ be given by \eqref{RhoEpsDelta} and satisfying \eqref{RhoEpsDeltaEst}. Let $S_{i,\varepsilon}$ be then given by
\begin{equation}\label{DefSiEps}
S_{i,\varepsilon}(x)=S_i\left(\frac{|x-x_\varepsilon|}{\mu_\varepsilon} \right)\,,
\end{equation}
for $i=0,1,2$. Let $\xi_\varepsilon>0$ be given by \eqref{DefXi}. In (Case 1) where $N_\varepsilon\to +\infty$ as $\varepsilon\to 0$, we get that $\xi_\varepsilon=O(N_\varepsilon^{-1/2})$ by  \eqref{MinorPhiNEps} and \eqref{TechnicalRk}. Then, in any case, we clearly have that
\begin{equation}\label{XiTo0App}
\xi_\varepsilon\to 0
\end{equation}
as $\varepsilon\to 0$. Then we are in position to state the main result of this section.

\begin{prop}\label{ClaimDevBubble}
We have that
\begin{equation}\label{ExpBubble}
B_{\varepsilon}=\gamma_\varepsilon-\frac{t_\varepsilon}{\gamma_\varepsilon}+\frac{S_{0,\varepsilon}}{\gamma_\varepsilon^3}+\frac{S_{1,\varepsilon}}{\gamma_\varepsilon^5}+(A(\gamma_\varepsilon)-2\xi_\varepsilon)\frac{S_{2,\varepsilon}}{\gamma_\varepsilon}+o\left(t_\varepsilon\left(\frac{1}{\gamma_\varepsilon^5}+\frac{|A(\gamma_\varepsilon)|+\xi_\varepsilon}{\gamma_\varepsilon} \right) \right)\,,
\end{equation}
uniformly in $[0,\rho_\varepsilon]$, as $\varepsilon\to 0$.
\end{prop}
 In particular, using also \eqref{Propg} and \eqref{EqHBehavior}, it can be checked that $B_\varepsilon$ is positive and radially decreasing in $[0,\rho_\varepsilon]$. Observe also that $\xi_\varepsilon\ll \gamma_\varepsilon^{-4}$ can be seen as a remainder term in $(\text{Case 2})$. Let $\zeta_\varepsilon>0$ be given by
\begin{equation}\label{DefZetaApp}
\zeta_\varepsilon=\max\left({\frac{1}{\gamma_\varepsilon^4},|A(\gamma_\varepsilon)|,\xi_\varepsilon}\right)\,.
\end{equation}
 Resuming the computations below, we get as a by product of Proposition \ref{ClaimDevBubble} that, $v_\varepsilon=o\left(\frac{ t_\varepsilon}{\gamma_\varepsilon^5} \right)$ implies that
  \begin{equation}\label{LaplBPert}
 \begin{split}
  &\frac{\lambda_\varepsilon \Psi'_\varepsilon(B_\varepsilon+v_\varepsilon)}{2}=\frac{4~\exp(-2t_\varepsilon)}{\mu_\varepsilon^2 \gamma_\varepsilon}\Bigg[1+\frac{(\Delta S_0)\left(\frac{\cdot-x_\varepsilon}{\mu_\varepsilon} \right)}{\gamma_\varepsilon^2}+\frac{(\Delta S_1)\left(\frac{\cdot-x_\varepsilon}{\mu_\varepsilon} \right)}{\gamma_\varepsilon^4}\\
 &\quad\quad\quad\quad\quad\quad\quad+\left(A(\gamma_\varepsilon)-2\xi_\varepsilon \right)(\Delta S_2)\left(\frac{\cdot-x_\varepsilon}{\mu_\varepsilon} \right)+o\left(\zeta_\varepsilon \exp(\tilde{\delta}_0 t_\varepsilon) \right)\Bigg]\,,
 \end{split}
 \end{equation}
 uniformly in $\{y\text{ s.t. }t_\varepsilon(y)\le \gamma_\varepsilon\}$, for some given {$\tilde{\delta}_0\in(\delta_0,1)$}, for $\delta_0$ as in \eqref{InftyBehavior}. 
 
 \begin{proof}[Proof of Proposition \ref{ClaimDevBubble}] Since both arguments are very similar to prove $(\text{Case 1})$ and (Case 2), for the sake of readability, we only write the proof of Claim \ref{ClaimDevBubble} in the more delicate (Case 1). Then, assume that we are in (Case 1). We let $\tau_\varepsilon$ be given by
\begin{equation}\label{DefTauApp}
B_\varepsilon=\gamma_\varepsilon-\frac{\tau_\varepsilon}{\gamma_\varepsilon}\,.
\end{equation} 
Let $\bar{w}_\varepsilon$ be given by
\begin{equation}\label{DefBarWApp}
B_\varepsilon=\gamma_\varepsilon-\frac{t_\varepsilon}{\gamma_\varepsilon}+\frac{S_{0,\varepsilon}}{\gamma_\varepsilon^3}+\frac{S_{1,\varepsilon}}{\gamma_\varepsilon^5}+(A(\gamma_\varepsilon)-2\xi_\varepsilon)\frac{S_{2,\varepsilon}}{\gamma_\varepsilon}+\frac{\zeta_\varepsilon \bar{w}_\varepsilon}{\gamma_\varepsilon}\,.
\end{equation}
  Let $\bar{\delta}>0$ be fixed and let $\bar{r}_\varepsilon\ge 0$ be given by
 \begin{equation}\label{DefRBarApp}
 \bar{r}_\varepsilon=\sup\left\{r>0\text{ s.t. }|\bar{w}_\varepsilon|\le \bar{\delta} t_\varepsilon \text{ in }[0,r] \right\}\,.
 \end{equation}
Now, since $\bar{\delta}>0$ may be arbitrarily small, in order to get Claim \ref{ClaimDevBubble}, it is sufficient to prove that $\bar{r}_\varepsilon=\rho_\varepsilon$, for all $0<\varepsilon\ll 1$. Using \eqref{DefRBarApp}, we perform computations in $[0,\bar{r}_\varepsilon]$ and the subsequent $o(1)$ are uniformly small in this set as $\varepsilon\to 0$. First, by \eqref{InftyBehavior}, \eqref{ExpansionSi}, \eqref{XiTo0App} and \eqref{DefRBarApp}, we have that
\begin{equation}\label{EstimTauApp}
\tau_\varepsilon=t_\varepsilon(1+o(1))\,.
\end{equation}
Observe that, as soon as we have $\Delta B_\varepsilon >0$ in $[0,\bar{r}_\varepsilon]$, then $B_\varepsilon$ is radially decreasing and \eqref{BLeGamma} holds true in $[0,\bar{r}_\varepsilon]$. Let $L^H_{\varepsilon}$ and $L_\varepsilon^g$ be given by 
\begin{equation}\label{DefLHLgApp}
H(B_\varepsilon)=H(\gamma_\varepsilon)\left(1+L^H_\varepsilon \right)\text{ and then, }(1+g(B_\varepsilon))=H(\gamma_\varepsilon)\left(1+L^H_\varepsilon+L^g_\varepsilon \right)\,.
\end{equation}
In view of \eqref{DefTauApp} and \eqref{EstimTauApp}, estimates of $L^H_\varepsilon,L^g_\varepsilon$ are given by \eqref{InftyBehavior} and \eqref{AsymptG}, respectively. We are now in position to expand the right-hand side of \eqref{B1Eps}. From now on, it is convenient to denote
\begin{equation}\label{NTilde}
\tilde{N}_\varepsilon=N_\varepsilon-1\,.
\end{equation}
 Going back to \eqref{PsiPrime}, we have that
\begin{equation}\label{ExpanPsi1App}
\frac{\Psi'_{N_\varepsilon}(B_\varepsilon)}{2}=B_\varepsilon H(\gamma_\varepsilon)\left[\left(1+L^H_\varepsilon \right)\left(1+\varphi_{\tilde{N}_\varepsilon}(B_\varepsilon^2) \right)+L^g_\varepsilon \left(\frac{B_\varepsilon^{2N_\varepsilon}}{N_\varepsilon !}-B_\varepsilon^2 \right)\right]
\end{equation}
By \eqref{RhoEpsDelta}, \eqref{DefTauApp} and \eqref{EstimTauApp} and since $\bar{r}_\varepsilon\le \rho_\varepsilon$, we have that \begin{equation}\label{MinorBApp}
\min_{[0,\bar{r}_\varepsilon]}B_\varepsilon\ge (\varepsilon_0+o(1))\gamma_\varepsilon \to +\infty
\end{equation}
as $\varepsilon\to 0$. Thus, by Stirling's formula, we get that 
$$B_\varepsilon^{2N_\varepsilon}/(N_\varepsilon !)\ge \exp\left(N_\varepsilon\left(\log\frac{\gamma_\varepsilon^2}{N_\varepsilon}+\left(\log \varepsilon_0^2+1\right)+o(1) \right)\right)$$
and then, for all given integer $k\ge 0$, that
 \begin{equation}\label{CompApp}
B_\varepsilon^k=o(1)\times \frac{B_\varepsilon^{2N_\varepsilon}}{N_\varepsilon !}
 \end{equation}
in $[0,\bar{r}_\varepsilon]$, as $\varepsilon\to 0$, using $\varepsilon_0^2>{1/e}$ with \eqref{MinorGammaByNEps}. Similarly, for all given integer $k\ge 0$, we have that 
\begin{equation}\label{EasyCompApp}
\frac{B_\varepsilon^k}{\varphi_{{N}_\varepsilon}(B_\varepsilon^2)}=o(1)
\end{equation}
 in $[0,\bar{r}_\varepsilon]$, as $\varepsilon\to 0$. Then, by \eqref{ScalRel}, \eqref{DefTauApp}, \eqref{EasyCompApp} and \eqref{CompApp}, we may rewrite \eqref{ExpanPsi1App} as
\begin{equation}\label{ExpanPsi2App}
\begin{split}
\frac{\lambda_\varepsilon \Psi'_{N_\varepsilon}(B_\varepsilon)}{2}= &\frac{4}{\mu_\varepsilon^2 \gamma_\varepsilon}\left(1-\frac{\tau_\varepsilon}{\gamma_\varepsilon^2} \right) \Bigg[O(\exp(-\gamma_\varepsilon^2))+\frac{\varphi_{\tilde{N}_\varepsilon}(B_\varepsilon^2)}{\varphi_{\tilde{N}_\varepsilon}(\gamma_\varepsilon^2)}\times\\
&\left(1+L^H_\varepsilon+O\left(\frac{B_\varepsilon^{2N_\varepsilon}}{N_\varepsilon!~ \varphi_{\tilde{N}_\varepsilon}(B_\varepsilon^2)}L^g_\varepsilon\right)\right)\Bigg]
\end{split}
\end{equation}
in $[0,\bar{r}_\varepsilon]$, as $\varepsilon\to 0$. Indeed, by \eqref{MinorBApp}, we have that
\begin{equation}\label{ContrL}
L^H_\varepsilon=o(1)\text{ and }L^g_\varepsilon=o(1)
\end{equation}
in $[0,\bar{r}_\varepsilon]$ as $\varepsilon\to 0$, using \eqref{Propg}, \eqref{EqHBehavior} and \eqref{DefLHLgApp}. In \eqref{ExpanPsi2App}, the term $O(\exp(-\gamma_\varepsilon^2))$ equals $(1+L^H_\varepsilon)/\varphi_{\tilde{N}_\varepsilon}(\gamma_\varepsilon^2)$ and we thus get this control by \eqref{MinorPhiNEps} and \eqref{ContrL}. In the following lines, we expand the terms of \eqref{ExpanPsi2App}. By \eqref{AlgRelat} with $\Gamma=\gamma_\varepsilon^2$ and $T=B_\varepsilon^2$, we get that
\begin{equation}\label{CrucQuotApp}
\begin{split}
\frac{\varphi_{\tilde{N}_\varepsilon}(B_\varepsilon^2)}{\varphi_{\tilde{N}_\varepsilon}(\gamma_\varepsilon^2)}=\exp(B_\varepsilon^2-\gamma_\varepsilon^2)-F_\varepsilon\,,
\end{split}
\end{equation}
where $F_\varepsilon$ satisfies in $[0,\bar{r}_\varepsilon]$ 
\begin{equation}\label{DefFApp}
\begin{split}
F_\varepsilon &=\frac{B_\varepsilon^{2\tilde{N}_\varepsilon}}{\tilde{N}_\varepsilon ! \varphi_{\tilde{N}_\varepsilon}(\gamma_\varepsilon^2)}\int_0^{\gamma_\varepsilon^2-B_\varepsilon^2}\exp\left(-u\right) {\left(1+\frac{u}{B_\varepsilon^2}\right)^{\tilde{N}_\varepsilon}}du\,,\\
&=\frac{\exp(B_\varepsilon^2)}{\varphi_{\tilde{N}_\varepsilon}(\gamma_\varepsilon^2)} \int_{B_\varepsilon^2}^{\gamma_\varepsilon^2} \exp(-s)\frac{s^{\tilde{N}_\varepsilon}}{\tilde{N}_\varepsilon !} ds\,,\\
&=\xi_\varepsilon \exp(B_\varepsilon^2-\gamma_\varepsilon^2) \int_{B_\varepsilon^2-\gamma_\varepsilon^2}^0 \exp(-y)\left(1+\frac{y}{\gamma_\varepsilon^2} \right)^{\tilde{N}_\varepsilon} dy\,.
\end{split}
\end{equation}
By \eqref{DefTauApp} and \eqref{DefBarWApp}, we may write
$$\tau_\varepsilon=t_\varepsilon-\frac{S_{0,\varepsilon}}{\gamma_\varepsilon^2}-\frac{S_{1,\varepsilon}}{\gamma_\varepsilon^4}-\left(A(\gamma_\varepsilon)-2\xi_\varepsilon \right)S_{2,\varepsilon}-\zeta_\varepsilon \bar{w}_\varepsilon \,.$$
We set $\bar{t}_\varepsilon=1+t_\varepsilon$. Then, keeping in mind \eqref{ExpansionSi}, \eqref{XiTo0App}, \eqref{DefRBarApp},  \eqref{EstimTauApp} and $t_\varepsilon\le \gamma_\varepsilon^2$, we may compute
\begin{equation}\label{Interm1App}
\begin{split}
&\exp(B_\varepsilon^2-\gamma_\varepsilon^2)\\
&=\exp\left(-2\tau_\varepsilon+\frac{\tau_\varepsilon^2}{\gamma_\varepsilon^2}\right)\\
&= \exp\Bigg[-2\tau_\varepsilon+\frac{1}{\gamma_\varepsilon^2}\left(t_\varepsilon^2-\frac{2t_\varepsilon S_{0,\varepsilon}}{\gamma_\varepsilon^2}+O\left(\zeta_\varepsilon \bar{t}_\varepsilon^2 \right) \right)\Bigg]
\end{split}
\end{equation}
in $[0,\bar{r}_\varepsilon]$, as $\varepsilon\to 0$. Observe that
\begin{equation}\label{ExpoRemApp}
\left|\exp(y)-\sum_{j=0}^{N} \frac{y^j}{j!} \right|\le \frac{|y|^{N+1}}{(N+1)!}\exp(|y|)\,, 
\end{equation}
for all $y\in\mathbb{R}$ and all integer $N\ge 0$. Then we draw from \eqref{Interm1App} that
\begin{equation}\label{Interm2App}
\begin{split}
&\left(1-\frac{\tau_\varepsilon}{\gamma_\varepsilon^2} \right)\exp(B_\varepsilon^2-\gamma_\varepsilon^2)\\
&=\exp(-2t_\varepsilon)\Bigg[1+\frac{1}{\gamma_\varepsilon^2}\left(2S_{0,\varepsilon}+t_\varepsilon^2-t_\varepsilon \right)+\\
&\frac{1}{\gamma_\varepsilon^4}\left(2S_{1,\varepsilon}+2 S_{0,\varepsilon}^2+\frac{t_\varepsilon^4}{2}+2S_{0,\varepsilon}t_\varepsilon^2-4S_{0,\varepsilon} t_\varepsilon-t_\varepsilon^3+S_{0,\varepsilon}\right)\\
&+2\left(A(\gamma_\varepsilon)-2\xi_\varepsilon \right)S_{2,\varepsilon}+2\zeta_\varepsilon \bar{w}_\varepsilon\\
&+O\left(\left(\frac{\bar{t}_\varepsilon^6}{\gamma_\varepsilon^6}+\frac{\zeta_\varepsilon \bar{t}_\varepsilon^3}{\gamma_\varepsilon^2}+\zeta_\varepsilon^2\bar{t}_\varepsilon^3 \right)\exp\left(o(t_\varepsilon)+\frac{t_\varepsilon^2}{\gamma_\varepsilon^2} \right) \right)\Bigg]
\end{split}
\end{equation}
in $[0,\bar{r}_\varepsilon]$, as $\varepsilon\to 0$. Independently, by \eqref{MinorPhiNEps}, \eqref{TechnicalRk}, \eqref{DefTauApp}, \eqref{DefRBarApp}, \eqref{EstimTauApp} and since $B_\varepsilon(x_\varepsilon)=\gamma_\varepsilon$, for all given $R>0$, we have that
\begin{equation}\label{OtherEstimatesApp}
\begin{split}
&\quad\quad\quad\left\|\frac{B_\varepsilon^{2\tilde{N}_\varepsilon}}{\tilde{N}_\varepsilon!~ \varphi_{\tilde{N}_\varepsilon}(B_\varepsilon^2)}+\frac{B_\varepsilon^{2N_\varepsilon}}{N_\varepsilon!~ \varphi_{\tilde{N}_\varepsilon}(B_\varepsilon^2)}\right\|_{L^\infty([0,\min(R \mu_\varepsilon,\bar{r}_\varepsilon)]}=O\left(\frac{1}{\sqrt{N_\varepsilon}}\right)\\
&\text{and}\\
&\quad\quad\quad\frac{B_\varepsilon^{2N_\varepsilon}}{N_\varepsilon!~ \varphi_{\tilde{N}_\varepsilon}(B_\varepsilon^2)}\le 1\,,
\end{split}
\end{equation}
in $[0,\bar{r}_\varepsilon]$, the second inequality being obvious by \eqref{DefPhiN} and \eqref{NTilde}. In the sequel, by \eqref{MinorGammaByNEps}, we may assume that
\begin{equation}\label{ConvBetaApp}
\beta_\varepsilon:=\frac{\tilde{N}_\varepsilon}{\gamma_\varepsilon^2}\text{ satisfies }\lim_{\varepsilon\to 0} \beta_\varepsilon=\beta_0\in [0,1]\,,
\end{equation}
up to a subsequence. Now, we give estimates for $F_\varepsilon$ given in \eqref{DefFApp}. Up to a subsequence, we can split our results according to the following two cases
\begin{equation}\label{CasesEstimFApp}
\begin{split}
& \text{Case A: } \lim_{\varepsilon\to 0} \frac{\gamma_\varepsilon^2-\tilde{N}_\varepsilon}{\sqrt{\tilde{N}_\varepsilon}}=+\infty\,, \\
&\text{Case B: } \frac{\gamma_\varepsilon^2-\tilde{N}_\varepsilon}{\sqrt{\tilde{N}_\varepsilon}}=O(1)\,.
\end{split}
\end{equation}
Observe that, since we assume \eqref{MinorGammaByNEps}, all the possible situations are considered in \eqref{CasesEstimFApp}. Let $(r_\varepsilon)_\varepsilon$ be any sequence such that 
\begin{equation}\label{CondRApp}
r_\varepsilon\in[0,\bar{r}_\varepsilon]
\end{equation}
 for all $\varepsilon$. We prove that, in (Case A):
\begin{equation}\label{EstimFCaseA}
\begin{split}
F_\varepsilon(r_\varepsilon)=
\begin{cases}
&O\left( \xi_\varepsilon \gamma_\varepsilon\exp(-2 t_\varepsilon(r_\varepsilon) (\beta_0+o(1))) \right)\,, \text{ if }B_\varepsilon(r_\varepsilon)^2\ge \tilde{N}_\varepsilon+\sqrt{\tilde{N}_\varepsilon}\,, \\
&O\left(\exp\left(-(1+\varepsilon_0+o(1))t_\varepsilon(r_\varepsilon) \right) \right)\,,\text{ if }B_\varepsilon(r_\varepsilon)^2< \tilde{N}_\varepsilon+\sqrt{\tilde{N}_\varepsilon}\,, \\
\end{cases}
\end{split}
\end{equation}
while we get in (Case B):
\begin{equation}\label{EstimFCaseB}
\begin{split}
F_\varepsilon(r_\varepsilon)=
\begin{cases}
&2t_\varepsilon(r_\varepsilon)\xi_\varepsilon \exp(-2t_\varepsilon(r_\varepsilon)(1+o(1)))\,, \text{ if }t_\varepsilon(r_\varepsilon)=o(\gamma_\varepsilon) \,,\\
&O\left(t_\varepsilon(r_\varepsilon)\xi_\varepsilon\exp\left(-(1+\varepsilon_0+o(1))t_\varepsilon(r_\varepsilon) \right)\right)\text{ if }\gamma_\varepsilon=O\left(t_\varepsilon(r_\varepsilon) \right)\,.
\end{cases}
\end{split}
\end{equation}
Now we prove \eqref{EstimFCaseA}. We start with the first estimate of \eqref{EstimFCaseA}. Then, we assume that $B_\varepsilon(r_\varepsilon)^2\ge \tilde{N}_\varepsilon+\sqrt{\tilde{N}_\varepsilon}$, and thus in particular that
\begin{equation}\label{InterEstimCaseAApp}
1-\frac{\tilde{N}_\varepsilon}{B_\varepsilon(r_\varepsilon)^2}\ge \frac{1+o(1)}{\sqrt{\tilde{N}_\varepsilon}}\,.
\end{equation}
Writing now $F_\varepsilon$ according to the first formula of \eqref{DefFApp}, using \eqref{BLeGamma}, \eqref{MinorBApp} and
\begin{equation}\label{MinorLogApp}
\log(1+t)\le t \text{ for all }t>-1\,,
\end{equation}
we get first  that
\begin{equation}\label{InterEstimCaseAApp2}
F_\varepsilon(r_\varepsilon)\le \xi_\varepsilon \exp(-2\tau_\varepsilon(r_\varepsilon)\beta_\varepsilon)\int_0^{\gamma_\varepsilon^2-B_\varepsilon^2}\exp\left(-y\left(1-\frac{\tilde{N}_\varepsilon}{B_\varepsilon(r_\varepsilon)^2} \right) \right) dy\,,
\end{equation}
and conclude the proof of the first estimate of \eqref{EstimFCaseA}, by \eqref{MinorGammaByNEps}, \eqref{EstimTauApp} and \eqref{InterEstimCaseAApp}. In order to prove the second estimate of \eqref{EstimFCaseA}, it is sufficient to write $F_\varepsilon$ according to the second formula of \eqref{DefFApp}, to check that
$$\int_{\mathbb{R}} \exp(-s)\frac{s^{\tilde{N}_\varepsilon}}{\tilde{N}_\varepsilon !} ds=1\,, $$
 that $r_\varepsilon\le\bar{r}_\varepsilon\le \rho_\varepsilon$ imply
\begin{equation}\label{TrivIneqApp}
t_\varepsilon(\bar{r}_\varepsilon)\le (1-\varepsilon_0)\gamma_\varepsilon^2\,,
\end{equation} 
and to use \eqref{DefTauApp}, \eqref{EstimTauApp} and \eqref{MinorPhiNEps}. Now we turn to the proof of \eqref{EstimFCaseB}. Then, we assume that (Case B) in \eqref{CasesEstimFApp} holds true and in particular that 
\begin{equation}\label{Triv2InApp}
1-\beta_\varepsilon=O\left(\frac{1}{\gamma_\varepsilon} \right)\text{ in (Case B)}\,.
\end{equation}
 Writing $F_\varepsilon$ according to the third estimate of \eqref{DefFApp}, we get that
\begin{equation}\label{PfCaseB1}
\begin{split}
F_\varepsilon=& \xi_\varepsilon \exp\left(-\tau_\varepsilon\left(2-\frac{\tau_\varepsilon}{\gamma_\varepsilon^2} \right) \right)(\gamma_\varepsilon^2-B_\varepsilon^2)\times\\
&~~~~~~ \int_0^1 \exp\left((\gamma_\varepsilon^2-B_\varepsilon^2)y+\tilde{N}_\varepsilon \log\left(1-\frac{(\gamma_\varepsilon^2-B_\varepsilon^2)y}{\gamma_\varepsilon^2} \right) \right) dy
\end{split}
\end{equation}
at $r_\varepsilon$. Expanding the $\log$, we easily get the first estimate of \eqref{EstimFCaseB} from \eqref{EstimTauApp}, \eqref{Triv2InApp}, \eqref{PfCaseB1} and the assumption $t_\varepsilon(r_\varepsilon)=o(\gamma_\varepsilon)$. The second estimate of \eqref{EstimFCaseB} can also be obtained from \eqref{PfCaseB1} by \eqref{EstimTauApp}, \eqref{MinorLogApp}, \eqref{TrivIneqApp} and \eqref{Triv2InApp}. This concludes the proof of \eqref{EstimFCaseB}. Now, we prove that, in (Case A) of \eqref{CasesEstimFApp}, we have that
\begin{equation}\label{IntegEst}
\int_0^{\bar{r}_\varepsilon} F_\varepsilon(r) r dr=o\left(\frac{\mu_\varepsilon^2}{\gamma_\varepsilon^4}  \right)\,.
\end{equation}
Since $r_\varepsilon\le \rho_\varepsilon$, we get from \eqref{DefXi}, \eqref{MinorPhiNEps}, \eqref{MinorGammaByNEps}, \eqref{EstimFCaseA} and by Stirling's formula that
\begin{equation}\label{IntegEstA1}
\begin{split}
&\int_{\left\{r\in[0,\bar{r}_\varepsilon],B_\varepsilon(r)^2\ge \tilde{N}_\varepsilon+\sqrt{\tilde{N}_\varepsilon}\right\}} F_\varepsilon(r) r dr\\
&\lesssim \exp\left(\gamma_\varepsilon^2[f(\beta_\varepsilon)+O((\log\gamma_\varepsilon)/\gamma_\varepsilon^2)] \right)\times \\
&\begin{cases}
&\mu_\varepsilon^2 \text{ if }\beta_0>1/2\,,\\
&\mu_\varepsilon^2 \exp(\gamma_\varepsilon^2(1-\varepsilon_0)(1-2\beta_0+o(1)))\text{ if }\beta_0\le 1/2\,,
\end{cases}
\end{split}
\end{equation}
where $f$ is the continuous function in $[0,1]$ given for $\beta\in (0,1]$ by
$$f(\beta)=\beta \log\frac{1}{\beta}+\beta-1 \,.$$
Independently, since $\bar{r}_\varepsilon\le \rho_\varepsilon$, if $$r_\varepsilon\in  J_\varepsilon:=\left\{r\in[0,\bar{r}_\varepsilon],B_\varepsilon(r)^2< \tilde{N}_\varepsilon+\sqrt{\tilde{N}_\varepsilon}\right\}\,,$$
then $J_\varepsilon\neq \emptyset$ and $\gamma_\varepsilon^2\lesssim\tilde{N}_\varepsilon$, by \eqref{DefTauApp}, \eqref{EstimTauApp} and \eqref{TrivIneqApp}. Thus we have that $$\gamma_\varepsilon\lesssim \sqrt{\tilde{N}_\varepsilon}\ll t_\varepsilon(r_\varepsilon)\,,$$ using that we are in (Case A) for the last estimate. Then, we get from \eqref{EstimFCaseA} that
\begin{equation}\label{IntegEstA2}
\int_{J_\varepsilon} F_\varepsilon(r) r dr\lesssim \int_{\{r\le \rho_\varepsilon,t_\varepsilon\ge \gamma_\varepsilon\}}\exp\left(-(1+\varepsilon_0+o(1))t_\varepsilon(r) \right) r dr=o\left(\frac{\mu_\varepsilon^2}{\gamma_\varepsilon^4} \right)\,.
\end{equation}
Observe that $f$ and $\beta\mapsto f(\beta)+(1-2\beta)/2$ are negative in $[0,1)$ and $[0,1/2]$ respectively. Moreover, because of (Case A) and by \eqref{MinorGammaByNEps}, we can check that $$\beta_\varepsilon=\frac{\tilde{N}_\varepsilon}{\gamma_\varepsilon^2}\le \frac{1}{1+\frac{1}{\sqrt{\tilde{N}_\varepsilon}}}\le 1-\frac{1+o(1)}{\sqrt{\tilde{N}_\varepsilon}}\le 1-\frac{1+o(1)}{\gamma_\varepsilon}\,,$$ since $\gamma_\varepsilon^2\ge \tilde{N}_\varepsilon+\sqrt{\tilde{N}_\varepsilon}$, and then that 
\begin{equation}\label{ConcludCaseA}
0<-f(\beta_\varepsilon)\lesssim 1/\gamma_\varepsilon\,.
\end{equation}
 Thus, we get \eqref{IntegEst} from the first estimate of \eqref{IntegEstA1} with \eqref{ConcludCaseA}, from the second estimate of \eqref{IntegEstA1} with $1-\varepsilon_0<1-\sqrt{1/e}<1/2$ and from \eqref{IntegEstA2}. Computing as in \eqref{IntegEstA1}, we get also that
 \begin{equation}\label{EstimXiCaseAApp}
 \xi_\varepsilon=o\left(\frac{1}{\gamma_\varepsilon^4}\right)
\end{equation}  
in (Case A) (see \eqref{ConcludCaseA}). By \eqref{EstimTauApp} and the second part of \eqref{OtherEstimatesApp}, using that $\bar{r}_\varepsilon\le \rho_\varepsilon$, we may rewrite \eqref{ExpanPsi2App} as
\begin{equation}\label{ExpanPsi3App}
\begin{split}
\frac{\lambda_\varepsilon \Psi'_{N_\varepsilon}(B_\varepsilon)}{2}= &\frac{4}{\mu_\varepsilon^2 \gamma_\varepsilon} \Bigg[\left(1-\frac{\tau_\varepsilon}{\gamma_\varepsilon^2}+L_\varepsilon^H \right)\exp(B_\varepsilon^2-\gamma_\varepsilon^2)-F_\varepsilon\\
&+O\left(\frac{t_\varepsilon}{\gamma_\varepsilon^2}|F_\varepsilon|+\exp(-\gamma_\varepsilon^2)\right)\\
&+O\left(\left(\frac{t_\varepsilon}{\gamma_\varepsilon^2}\exp(B_\varepsilon^2-\gamma_\varepsilon^2)+|F_\varepsilon|\right)\left(|L_\varepsilon^H|+|L_\varepsilon^g|\right) \right)\\
&+O\left(|L_\varepsilon^g|\exp(B_\varepsilon^2-\gamma_\varepsilon^2)\frac{B_\varepsilon^{2N_\varepsilon}}{N_\varepsilon ! \varphi_{\tilde{N}_\varepsilon}(B_\varepsilon^2)}\right) \Bigg]\,.
\end{split}
\end{equation}
By \eqref{RhoEpsDeltaEst}, we clearly have that
\begin{equation}\label{RoughConstatApp}
\int_0^{{\rho}_\varepsilon}  \exp(-\gamma_\varepsilon^2) r dr=o\left(\frac{\mu_\varepsilon^2}{\gamma_\varepsilon^4} \right)\,.
\end{equation}
Integrating by parts, observe that $\bar{w}_\varepsilon$ given by \eqref{DefBarWApp} satisfies 
\begin{equation}\label{ObservationWApp}
\bar{w}_\varepsilon(0)=0\text{ and }-r_\varepsilon \bar{w}'_{\varepsilon}(r_\varepsilon)=\int_0^{r_\varepsilon} (\Delta \bar{w}_\varepsilon)~ r dr \,,
\end{equation}
where, still using radial notations, $\bar{w}'_\varepsilon(r)=\frac{d\bar{w}_\varepsilon}{dr}(r)$. Now we estimate $\bar{w}_\varepsilon$ in $[0,\bar{r}_\varepsilon]$, by using \eqref{ObservationWApp}. By \eqref{B1Eps}, \eqref{DefBarWApp} and \eqref{ExpanPsi3App}, we are in position to estimate the RHS of \eqref{ObservationWApp}, for $r_\varepsilon$ still as in \eqref{CondRApp}. Assume first that we are in (Case A) of \eqref{CasesEstimFApp}. By plugging \eqref{InftyBehavior}, \eqref{AsymptG}, \eqref{EquationApp}, \eqref{ExpansionSi}, \eqref{DefLHLgApp}, \eqref{ContrL}, \eqref{Interm2App}, \eqref{OtherEstimatesApp}, \eqref{EstimFCaseA}, \eqref{IntegEst}, \eqref{EstimXiCaseAApp}, \eqref{RoughConstatApp} in \eqref{ExpanPsi3App}, by using the dominated convergence theorem and by coming back to the definition \eqref{DefZetaApp} of $\zeta_\varepsilon$,  we get that
\begin{equation}\label{EqCaseACCl}
\begin{split}
\int_0^{r_\varepsilon} |(\Delta\bar{w}_\varepsilon)| rdr= &O\left(\|\bar{w}'_\varepsilon\|_{L^\infty([0,r_\varepsilon])} \int_0^{r_\varepsilon/\mu_\varepsilon} \frac{\mu_\varepsilon r^2 dr}{(1+r^2)^{1+\varepsilon_0+o(1)}} \right)\\
&~~~~~+o\left(\int_0^{r_\varepsilon/\mu_\varepsilon} \frac{rdr}{(1+r^2)^{1+\varepsilon_0+o(1)}} \right)\,.
\end{split}
\end{equation} 
The first term in the right hand side of \eqref{EqCaseACCl} uses that, for all $r\in[0,r_\varepsilon]$, 
$$ |\bar{w}_\varepsilon(r)| \le r \|\bar{w}'_\varepsilon\|_{L^\infty([0,r_\varepsilon])}\,. $$
Observe now that \eqref{EqCaseACCl} still holds true in (Case B) of \eqref{CasesEstimFApp}, replacing \eqref{EstimFCaseA}, \eqref{IntegEst} and \eqref{EstimXiCaseAApp} by \eqref{EstimFCaseB} in the above argument. Since $\varepsilon_0>1/2$, we clearly get from \eqref{ObservationWApp} and \eqref{EqCaseACCl} that, in (Case A) and in (Case B),
\begin{equation}\label{WSharpEstApp}
\begin{split}
r_\varepsilon |\bar{w}'_\varepsilon(r_\varepsilon)|= &O\left(\|\bar{w}'_\varepsilon\|_{L^\infty([0,r_\varepsilon])} \frac{{\mu_\varepsilon(r_\varepsilon/\mu_\varepsilon)^3}}{1+(r_\varepsilon/\mu_\varepsilon)^3} \right)+o\left(\frac{(r_\varepsilon/\mu_\varepsilon)^2}{1+(r_\varepsilon/\mu_\varepsilon)^2} \right)\,.
\end{split}
\end{equation}
Now we prove that
\begin{equation}\label{LinftyBdApp}
\mu_\varepsilon \|\bar{w}'_\varepsilon\|_{L^\infty([0,\bar{r}_\varepsilon])}=o(1)\,.
\end{equation}
If \eqref{LinftyBdApp} does not hold true, then, by \eqref{WSharpEstApp}, there exists $s_\varepsilon\in [0,\bar{r}_\varepsilon]$ such that $s_\varepsilon=O(\mu_\varepsilon)$, $\mu_\varepsilon=O(s_\varepsilon)$, 
\begin{equation}\label{LaContradApp}
|\bar{w}'_\varepsilon(s_\varepsilon)|=\|\bar{w}'_\varepsilon\|_{L^\infty([0,\bar{r}_\varepsilon])}\text{ and }\limsup_{\varepsilon\to 0}\mu_\varepsilon |\bar{w}'_\varepsilon(s_\varepsilon)|>0\,.
\end{equation}
In particular, up to a subsequence, we may assume that there exists $\alpha_0\in (0,+\infty]$ such that $\bar{r}_\varepsilon/\mu_\varepsilon\to \alpha_0$ as $\varepsilon\to 0$. Let $\tilde{w}_\varepsilon$ be given by $$\tilde{w}_\varepsilon(y)=\bar{w}_\varepsilon(\mu_\varepsilon y)/(\mu_\varepsilon\|\bar{w}'_\varepsilon\|_{L^\infty([0,\bar{r}_\varepsilon])}).$$ By \eqref{WSharpEstApp} and \eqref{LaContradApp}, we get that $(\|(1+\cdot)\tilde{w}'_\varepsilon\|_{L^\infty([0,\bar{r}_\varepsilon/\mu_\varepsilon])})_\varepsilon$ is a bounded sequence. Then, computing as in \eqref{EqCaseACCl} and by radial elliptic theory with \eqref{B1Eps}, we get that $\tilde{w}_\varepsilon\to \tilde{w}_0$ in $C^2([0,\alpha_0])$ if $\alpha_0<+\infty$ or in $C^2_{loc}([0,\alpha_0))$ if $\alpha_0=+\infty$, where $\tilde{w}_0$ solves
\begin{equation*}
\begin{cases}
&\Delta \tilde{w}_0=8\exp(-2T_0) \tilde{w}_0 \text{ in }B_0(\alpha_0)\,,\\
&\tilde{w}_0(0)=0\,,\\
&\tilde{w}_0\text{ is radial around }0\in\mathbb{R}^2\,,
\end{cases}
\end{equation*}
still making usual radial identifications, and where $T_0$ is given in \eqref{DefT0}. By standard theory of radial elliptic equation, this implies $\tilde{w}_0\equiv 0$, which contradicts \eqref{LaContradApp} and proves \eqref{LinftyBdApp}. Then, since $\bar{w}_\varepsilon(0)=0$ and by the fundamental theorem of calculus, we get from \eqref{WSharpEstApp} with \eqref{LinftyBdApp} that $\bar{r}_\varepsilon=\rho_\varepsilon$ in \eqref{DefRBarApp}. By the discussion just above \eqref{EstimTauApp}, this concludes the proof of Proposition \ref{ClaimDevBubble}.

 \end{proof}


\end{document}